\documentclass[10pt]{amsart}
\usepackage{amsmath}
\usepackage{amsfonts}
\usepackage{amssymb}
\usepackage{amsthm}
\usepackage{url}
\usepackage{dsfont}
\usepackage{graphicx}
\usepackage{caption}
\usepackage{subcaption}
\usepackage{comment} 
\usepackage{stmaryrd}
\usepackage{hyperref}
\usepackage{todonotes}
\usepackage{color}
\usepackage{enumerate,tikz-cd}
\numberwithin{equation}{section}

\newtheorem{proposition}{Proposition}[section]
\newtheorem{lemma}[proposition]{Lemma}
\newtheorem{theorem}[proposition]{Theorem}
\newtheorem{corollary}[proposition]{Corollary}
\newtheorem{conjecture}{Conjecture}[section]

\theoremstyle{definition}
\newtheorem{remark}[proposition]{Remark}
\newtheorem{definition}[proposition]{Definition}
\newtheorem{example}[proposition]{Example}

\DeclareMathOperator{\Aut}{Aut}
\DeclareMathOperator{\Ric}{Ric}

\DeclareMathOperator{\DF}{DF}

\DeclareMathOperator{\Ima}{Im}
\DeclareMathOperator{\Rea}{Re}
\DeclareMathOperator{\Lie}{Lie}

\DeclareMathOperator{\BC}{BC}

\newcommand{\N}{\mathbb{N}}
\newcommand{\R}{\mathbb{R}}
\newcommand{\C}{\mathbb{C}}

\newcommand{\Z}{\mathbb{Z}}
\newcommand{\Q}{\mathbb{Q}}
\newcommand{\G}{\mathcal{G}}
\newcommand{\pr}{\mathbb{P}}
\renewcommand{\epsilon}{\varepsilon}

\newcommand{\D}{\mathcal{D}}
\newcommand{\M}{\mathcal{M}}
\newcommand{\scN}{\mathcal{N}}
\newcommand{\scP}{\mathcal{P}}
\newcommand{\scO}{\mathcal{O}}

\newcommand{\ddb}{i\partial \bar\partial}
\newcommand{\scL}{\mathcal{L}}
\newcommand{\J}{\mathcal{J}}
\newcommand{\mfh}{\mathfrak{h}}

\newcommand{\mfk}{\mathfrak{k}}
\newcommand{\F}{\mathcal{F}}

\renewcommand{\L}{\mathcal{L}}
\newcommand{\X}{\mathcal{X}}
\newcommand{\Y}{\mathcal{Y}}
\renewcommand{\H}{\mathcal{H}}
\newcommand{\U}{\mathcal{U}}

\renewcommand{\G}{\mathcal{G}}

\newcommand{\scS}{\mathcal{S}}
\newcommand{\ddbar}{\partial\overline{\partial}}
\renewcommand{\phi}{\varphi}

\title[Stability conditions for polarised varieties]{Stability conditions for polarised varieties}

\author[Ruadha\'i Dervan]{Ruadha\'i Dervan}

\address{Current address: Ruadha\'i Dervan, School of Mathematics and Statistics, University of Glasgow, University Place, Glasgow G12 8QQ, United Kingdom}\email{ruadhai.dervan@glasgow.ac.uk}
\address{Previous address: DPMMS, Centre for Mathematical Sciences, Wilberforce Road, Cambridge CB3 0WB, United Kingdom}\email{R.Dervan@dpmms.cam.ac.uk}

\begin{document}

\begin{abstract} We introduce an analogue of Bridgeland's stability conditions for polarised varieties. Much as Bridgeland stability is modelled on slope stability of coherent sheaves, our notion of $Z$-stability is modelled on the notion of K-stability of polarised varieties. We then introduce an analytic counterpart to stability, through the notion of a $Z$-critical K\"ahler metric, modelled on the constant scalar curvature K\"ahler condition. Our main result shows that a polarised variety which is analytically K-semistable and asymptotically $Z$-stable admits $Z$-critical K\"ahler metrics in the large volume regime. We also prove a local converse, and explain how these results can be viewed in terms of local wall crossing. A special case of our framework gives a manifold analogue of the deformed Hermitian Yang-Mills equation.
\end{abstract}

\maketitle

\section{Introduction}

Two notions of stability have dominated much of algebraic geometry over the last twenty years: these are the notions of \emph{K-stability} of a polarised variety \cite{tian-inventiones, donaldson-toric} and \emph{Bridgeland stability} of an object in a triangulated category \cite{bridgeland}. Bridgeland stability is modelled on the more classical notion of \emph{slope stability} of a coherent sheaf over a polarised variety, and slope stability can be viewed as the ``large volume limit'' of Bridgeland stability. One then expects to obtain moduli spaces of Bridgeland stable objects (and one frequently does \cite{toda-i,toda-ii,arcara-bertram}), with the usefulness of Bridgeland stability arising from the fact that one can vary the stability condition, which often leads to a good geometric understanding of the birational geometry of these moduli spaces. This, in turn, frequently leads to interesting geometric consequences \cite{bayer-survey}.

In the simplest case that the object of the triangulated category in question is a holomorphic vector bundle, there is a differential-geometric counterpart to Bridgeland stability, though the dictionary is not exact and theory is in its infancy. This counterpart is the notion of a $Z$-\emph{critical connection} \cite{DMS}, recently introduced by the author, McCarthy and Sektnan, which concretely is a solution to a partial differential equation on the space of Hermitian metrics on the holomorphic vector bundle. $Z$-critical connections should play an analogous role to Hermite-Einstein metrics in the study of slope stability of vector bundles, and indeed the ``large volume limit'' of the $Z$-critical condition is the Hermite-Einstein condition.

K-stability of a polarised variety originated directly through from K\"ahler geometry, through the search for \emph{constant scalar curvature K\"ahler (cscK)  metrics} on smooth polarised varieties, whose existence is conjectured by Yau, Tian and Donaldson is to be equivalent to K-stability \cite{yau,tian-inventiones, donaldson-toric}. Already through the early work of Fujiki and Schumacher it was apparent that the cscK condition (hence,  \emph{a posteriori}, the K-stability condition) should be the appropriate condition to form moduli of polarised varieties, and there is now much compelling evidence for this \cite{fujiki, fujiki-schumacher, moduli, inoue}, especially in the Fano setting \cite{odaka-moduli, LWX, xu}. With these moduli spaces being increasingly well understood, it is natural to ask what the geometry of these spaces is, and whether their birational geometry can be understood through other notions of stability; this is a heavily studied problem for moduli spaces of curves \cite{flip}. Thus one is led to the question: is there an analogue of Bridgeland stability for polarised varieties?

Here we begin a programme to answer this question. The definitions and techniques in the present work are most relevant in the ``large volume'' regime, where categorical input is less necessary, and the links with differential geometry are currently strongest. 

The main input into a Bridgeland stability condition is a \emph{central charge}; our analogue for varieties is essentially a complex polynomial in cohomology classes of the polarised variety $(X,L)$, including Chern classes of $X$. Fixing such a central charge $Z$, one obtains a complex number $Z(X,L)$ with \emph{phase} $\phi(X,L) = \arg Z(X,L)$, which we always assume to be non-zero. On the differential-geometric side, we introduce the notion of a $Z$\emph{-critical K\"ahler metric}, which  is a solution to a partial differential equation of the form $$\Ima(e^{-i\phi(X,L)} \tilde Z(\omega)) = 0,$$ where $\tilde Z(\omega)$ is a complex-valued  function defined using representatives of the cohomology classes associated to the central charge $Z(X,L)$, with appropriate Chern-Weil representatives chosen to represent the Chern classes. We also require the positivity condition $\Rea(e^{-i\phi(X,L)} \tilde Z(\omega))>0$. The $Z$-critical condition is then equivalent to asking that the function $$\tilde Z(\omega): X \to \C$$ has constant argument. The equation has formal similarities to the notion of a $Z$-critical connection on a holomorphic vector bundle, leading us to mirror the terminology.

On the algebro-geometric side, the notion of stability involves \emph{test configurations}, which are the $\C^*$-degenerations $(\X,\L)$ of $(X,L)$ crucial to the definition of K-stability. We associate a numerical invariant $Z(\X,\L)$ to each test configuration, which is again a complex number whose phase we denote $\phi(\X,\L)$. The notion of \emph{Z}-stability we introduce, which is roughly analogous to Bridgeland stability, means that for each test configuration the phase inequality $$\Ima\left(\frac{Z(\X,\L)}{Z(X,L)}\right)>0$$ holds. These definitions allow us to state the following analogue of the Yau-Tian-Donaldson conjecture:

\begin{conjecture}\label{intromainconjecture} Let $(X,L)$ be a smooth polarised variety with discrete automorphism group. Then the existence of a $Z$-critical K\"ahler metric in $c_1(L)$ is equivalent to $Z$-stability of $(X,L)$.\end{conjecture}

We should say immediately that this conjecture is only plausible in sufficiently ``large volume'' regions of the space of central charges; this is a condition which we expect to be explicit in concrete situations. Away from this region, categorical phenomena should enter. Thus Conjecture \ref{intromainconjecture} should be seen as a first approximation of a larger conjecture involving a more categorical framework. When the values $Z(X,L)$ and $Z(\X,\L)$ lie in the upper half plane, the inequality is equivalent to asking for the phase inequality $\phi(\X,\L) > \phi(X,L)$ to hold, and the ``large volume'' hypothesis should imply that for the relevant test configuration, $Z(\X,\L)$ does lie in the upper half plane. We also note that, much as with the Yau-Tian-Donaldson conjecture, it seems reasonable that one may need to impose a uniform notion of stability \cite[Conjecture 1.1]{uniform}; see \cite{li} for recent progress.   

Here we prove the ``large volume limit'' of this conjecture, for what seems to be the most interesting class of central charge. For this admissible class of central charge defined in Section \ref{mainresults}, when one scales the polarisation $L$ to $kL$ for $k \gg 0$, the central charge takes values in the upper half plane and  the leading order term in $k$ of the phase inequalities $\phi_k(X,L) < \phi_k(\X,\L)$ is simply the usual inequality on the Donaldson-Futaki invariant involved in the definition of K-stability. It follows that the natural notion of \emph{asymptotic $Z$-stability} implies \emph{K-semistability}. A K-semistable polarised variety conjecturally admits a test configuration with central fibre K-polystable, and we say that $(X,L)$ is \emph{analytically K-semistable} if there is a test configuration whose central fibre is a smooth polarised variety admitting a cscK metric. We in addition assume that the deformation theory of the central fibre is unobstructed, to aid the analytic argument in the following, which is our main result:

\begin{theorem}\label{intromainthm} Let $(X,L)$ be an analytically K-semistable variety which has discrete automorphism group. Then $(X,kL)$ admits $Z$-critical K\"ahler metrics for all $k \gg 0$ provided it is asymptotically $Z$-stable.
\end{theorem}

In particular when $(X,L)$ itself admits a cscK metric and has discrete automorphism group, we prove the existence of $Z$-critical K\"ahler metrics for all $k\gg 0$. The converse, namely that existence of $Z$-critical K\"ahler metrics implies asymptotic $Z$-stability, also holds in a weak, local sense. To discuss the sense in which this is true, we must discuss some of the elements of the proof of Theorem \ref{intromainthm}. We denote the cscK degeneration of $(X,L)$ by $(X_0,L_0)$, and consider the Kuranishi space $B$ of $(X_0,L_0)$; that the deformation theory of $(X_0,L_0)$ is unobstructed implies that $B$ is smooth. This space admits a universal family $(\X,\L) \to B$, and from its construction $\L$ admits a relatively K\"ahler metric which induces the cscK metric on $(X_0,L_0)$. There are then three steps:

\begin{enumerate}[(i)]

\item We reduce to the above finite dimensional moment map problem on $B$ by perturbing the relatively K\"ahler metric on $\L$ in such a way that the only obstruction to solving the $Z$-critical equation arise from the automorphisms of the central fibre $(X_0,L_0)$. This uses a quantitative version of the implicit function theorem, and occupies much of the paper.

\item We show that the $Z$-critical equation can, locally, be viewed as a moment map on a given orbit. More precisely, the automorphism group of $(X_0,L_0)$ acts on $B$, and on each orbit in $B$ we show that with respect to a natural K\"ahler metric we produce on $B$, the condition that the K\"ahler metric on the fibre is $Z$-critical is essentially the moment map for the action of the associated maximal compact subgroup action. This can be viewed as an orbit-wise analogue of the Fujiki-Donaldson moment map picture for the cscK equation \cite{fujiki, donaldson-moment}, but we take a new approach that gives weaker results but much greater flexibility. It is then important that the phase inequalities involved in the definition of $Z$-stability correspond exactly to  the weight inequalities arising from the finite dimensional moment map problem.

\item We show that, in our local finite-dimensional moment map problem, stability implies the existence of a zero of the moment map, which thus produces $Z$-critical K\"ahler metrics by the first step. This relies on a local version of the Kempf-Ness theorem proven in \cite[Section 4.2]{DMS}.
\end{enumerate}

This basic strategy is analogous to work of Br\"onnle and Sz\'ekelyhidi \cite{bronnle, szekelyhidi-deformations}, with the difference arising from the fact we consider a \emph{sequence} of moment maps and a strictly K-semistable manifold. 

As part of step $(ii)$, we obtain analogues of several important tools in the study of cscK metrics, such as the Futaki invariant associated to holomorphic vector fields, and an energy functional analogous to Mabuchi's K-energy. The local moment map picture also quite formally produces the local converse to Theorem \ref{intromainthm}. Let us say that $(X,L)$ is \emph{locally asymptotically $Z$-stable} if the phase inequality holds for all test configurations produced from the Kuranishi space $B$ of its cscK degeneration.

\begin{theorem} With the above setup, $(X,kL)$ admits $Z$-critical K\"ahler metrics for all $k \gg 0$ if and only if it is locally asymptotically $Z$-stable.
\end{theorem}

Thus we have proven a version of the large volume limit of Conjecture \ref{intromainconjecture}. There is an interesting interpretation of this result in terms of local wall-crossing. Wall-crossing phenomena arise when one can vary the stability condition, and one then expects the resulting moduli spaces to undergo birational transformations. The strictly stable locus is unchanged by suitably small changes of the stability condition, and the interesting question concerns the semistable locus. The above then demonstrates that the algebro-geometric walls, governed by $Z$-stability, agree with the differential-geometric walls, governed by the existence of $Z$-critical K\"ahler metrics.

Our results can be seen as manifold analogues of results established in \cite{DMS} for holomorphic vector bundles. There it is proven that the existence of $Z$-critical connections on a holomorphic vector bundle is equivalent to asymptotic $Z$-stability of the bundle; the latter notion is a variant of Bridgeland stability. The strategy employed in \cite{DMS} is different: there, a local version of the Kempf-Ness theorem is used to provide a good choice of initial connection \cite[Section 4.2.1]{DMS}, after which analytic aspects of $Z$-critical connections enters. Here the analysis is considerably more involved, leading us to perform the key analytic step first. To ensure that we stay in the realm of K\"ahler geometry, we perturb the fibrewise K\"ahler metric rather than perturbing the almost complex structure (the latter approach has its origins in the fundamental work of Sz\'ekelyhidi \cite{szekelyhidi-deformations}); this new approach is crucial to allowing us to employ the local version of the Kempf-Ness theorem.

Continuing with the comparison with the bundle story, we must mention that the general notion of a $Z$-critical connection is modelled on the specific notion of a \emph{deformed Hermitian Yang-Mills connection} associated with a special central charge of particular relevance to mirror symmetry. Indeed, the deformed Hermitian Yang-Mills equation was introduced through SYZ mirror symmetry to be the mirror of the special Lagrangian equation \cite{LYZ}. The quite beautiful theory of this equation on holomorphic line bundles has developed with speed over the past few years \cite{jacob-yau, chen, collins-jacob-yau, collins-yau}, and these developments have emphasised that the special form of the central charge in this case has significant geometric implications. We thus emphasise that there is a direct analogue of the deformed Hermitian Yang-Mills equation for manifolds, which one might call the \emph{deformed cscK equation}  and which seems to be the natural avenue for further research. Fixing normal coordinates for the K\"ahler metric $\omega$ in which $\Ric\omega$ is diagonal, let $\lambda_1,\hdots,\lambda_n$ be the eigenvalues of $\Ric\omega$ and let $\sigma_j(\omega)$ denote the $j^{th}$ elementary symmetric polynomial in these eigenvalues. Then this equation takes the form $$\Ima\left(e^{-i\phi(X,L)}\left(\sum_{j=0}^n (-i)^j(\sigma_j(\omega) - \Delta\sigma_{j-1}(\omega))\right)\right) =0.$$ We remark that the name is misleading, as it is only truly a ``deformation'' of the cscK equation in the large volume limit. We also remark that the phase range in which existence of solutions to the deformed Hermitian Yang-Mills equation is equivalent to stability is the full supercritical phase range \cite{chen}, which emphasises that in explicit situations one should expect the large volume hypothesis of Conjecture \ref{intromainconjecture} to be similarly explicit.

The simplest new PDE we consider in the present work, to which Theorem \ref{intromainthm} applies, takes the form $$S(\omega) + \frac{1}{k}\left( \frac{2}{n(n-2)}\Delta S(\omega) - \frac{\Ric\omega^2\wedge\omega^{n-2}}{\omega^n}\right) = const.,$$ which is an elliptic, sixth order fully-nonlinear PDE in the K\"ahler potential  and which is exactly the $Z$-critical equation for a special ($k$-dependent) central charge. An important feature of the equation is that in the large volume regime $k\to \infty$, the constant of ellipticity degenerates to zero. Much of our analytic work is devoted to this PDE, and in the large volume regime $k \gg 0$ we view the general $Z$-critical equation as a perturbation of this model equation. 

\subsection*{Categorification} The approach we take in the present work is to consider explicitly defined central charges, as opposed to an axiomatic approach more closely analogous to the theory of Bridgeland stability conditions. In a sequel to this paper \cite{git-stabilityconditions}, an axiomatic approach to stability conditions on general stacks is developed (with the relevant stack here being the stack of polarised schemes), motivated by the more explicit approach taken here. To explain this, it is clearer to view the central charge as a function on schemes endowed with a $\C^*$-action. The key properties are then \emph{additivity} of the central charge, which essentially asks that the central charge is additive under composition of commuting one-parameter subgroups, and \emph{equivariant constancy} of the central charge, which asks that the value of the central charge is constant in equivariant flat families. We refer to \cite{git-stabilityconditions} for further details.

\subsection*{Stability of maps} While we have thus far emphasised the case of polarised varieties, and while our main result only holds in that setting, the basic framework is more general and links with interesting questions in enumerative geometry. While for a broad and interesting class of central charge, the ``large volume condition'' is K-stability, in general one obtains the notion of twisted K-stability \cite{uniform}, which is linked to the existence of twisted cscK metrics. The appropriate geometric context in which to study twisted K-stability is when one has a map $p: (X,L) \to (Y,H)$ of polarised varieties, where it is essentially equivalent to K-stability of the map $p$ \cite{stablemaps, extremal}. 

From the moduli theoretic point of view, one expects to be able to form moduli of K-stable maps to a fixed $(Y,H)$. The definition of K-stability of maps generalises Kontsevich's notion when $(X,L)$ is a curve, and the resulting (entirely conjectural) higher dimensional moduli spaces would thus be higher dimensional analogues of the moduli space of stable maps; there is also a version of theory involving divisors, as a higher dimensional analogue of the maps of marked curves used in Gromov-Witten theory \cite{alexeev}\cite[Section 5.3]{stablemaps}. What seems most interesting is that our work suggests that there should be variants of stability of maps even in the curves case, which may even lead to an understanding of wall-crossing phenomena for Gromov-Witten invariants; this seems likely to require developing a more categorical approach to the problem as discussed above.

\subsection*{Acknowledgements} I thank  Frances Kirwan, John McCarthy,  Jacopo Stoppa, G\'abor Sz\'ekelyhidi and especially Lars Sektnan for several interesting discussions on this circle of ideas, Michael Hallam, Yoshi Hashimoto and Eiji Inoue for technical advice. I was funded by a Royal Society University Research Fellowship for the duration of this work.

\section{$Z$-stability and $Z$-critical K\"ahler metrics} 

Here we define the key algebro-geometric and differential-geometric criteria of interest to us: $Z$-stability and $Z$-critical K\"ahler metrics. The definitions involve a central charge, which involves various Chern classes of $X$. The differential geometry is substantially more complicated when higher Chern classes (rather than merely the first Chern class) appear in the central charge, and so we postpone the definitions and results in that case to Section \ref{sec:higherrank}. The difference is roughly analogous to the difference between the theory of $Z$-critical connections on holomorphic line bundles and bundles of higher rank, and so we call the situation in which higher Chern classes appear the ``higher rank case''. The analogy is far from exact, and the case in which only the first Chern class and its powers appear in the central charge already exhibits many of the main difficulties in the study of $Z$-critical connections on arbitrary rank vector bundles.

\subsection{Stability conditions}

\subsubsection{$Z$-stability}

We work throughout over the complex numbers, in order to preserve links with the complex differential geometry. We also fix a normal polarised variety $(X,L)$ of dimension $n$, with $L$ an ample $\Q$-line bundle. Normality implies that the canonical class $K_X$ of $X$ exists as a Weil divisor, we always assume that $K_X$ exists as a $\Q$-line bundle. 

In addition to our ample line bundle, we will fix a \emph{stability vector}, a \emph{unipotent cohmology class} and a \emph{polynomial Chern form}; we define these in turn.

\begin{definition} A \emph{stability vector} is a sequence of complex numbers $$\rho = (\rho_0, \hdots, \rho_n) \in \C^{n+1}$$ such that $\rho_n = i= \sqrt{-1}$.\end{definition}

The condition $\rho_n = i $ is a harmless normalisation condition which, when it is not satisfied, can be achieved by multiplying the stability vector by a fixed complex number. In Bridgeland stability, one normally assumes $\rho \in (\C^*)^{n+1}$; this will be unnecessary for us.

\begin{definition} A \emph{unipotent cohomology class} is a complex cohomology class $\Theta \in \oplus_j H^{j,j}(X,\C)$  which is of the form $\Theta = 1 + \Theta'$, where $\Theta' \in H^{>0}(X,\C)$. \end{definition}

Note that $\Theta'$ must satisfy $$\overbrace{\Theta' \cdot \hdots \cdot \Theta'}^{j \text{ times}}=0$$ for $j \geq n+1$. A typical example of a choice of $\Theta$ is to fix a class $\beta \in H^{1,1}(X,\R)$ and set $\Theta = e^{-\beta}$, which is analogous to  a ``$B$-field'' in Bridgeland stability.

\begin{definition}A \emph{polynomial Chern form} is a sum of the form $$f(K_X) = \sum_{j=0}^n a_j K_X^j,$$ where $a_j \in \C$ and $K_X^j$ denotes the $j^{th}$-intersection product $K_X \cdot \hdots \cdot K_X$, viewed as a cycle. We always assume the normalisation condition $a_0 = a_1=1$, and interpret $K_X^0 = 1$ as a cycle.\end{definition}

As mentioned above, in the current section we restrict ourselves to central charges only involving $c_1(X) = c_1(-K_X)$, with the case of higher Chern classes postponed to Section \ref{sec:higherrank}.

\begin{definition} A \emph{polynomial central charge} is a function $Z: \N \to \C$ taking the form $$Z_{k}(X,L) =\sum_{l=0}^n \rho_l k^{l}   \int_X  L^l \cdot f(K_X) \cdot \Theta,$$ for some $\rho$ and $\Theta$. A \emph{central charge} is a polynomial central charge with $k$ fixed, such that $Z(X,L)\neq 0$. We often set $\epsilon = k^{-1}$ and denote the induced quantity by $Z_{\epsilon}(X,L)$.
\end{definition}

We will sometimes simply call a polynomial central charge a central charge when the dependence on $k$ is clear from context. The definition is motivated by an analogous definition of Bayer in the bundle setting \cite[Theorem 3.2.2]{bayer-polynomial}. For a polynomial central charge it is automatic that $Z_{k}(X,L)$ lies in the upper half plane in $\C$ for $k \gg 0$,  since $\Ima(\rho_n)>0$. Thus we can make the following definition.

\begin{definition} We define the \emph{phase} of $X$ to be $$\phi_k(X,L) = \arg Z_{k}(X,L),$$ the argument of the non-zero complex number. We denote this by $\phi(X,L)$ when $k$ is fixed, and for fixed $(X,L)$ often simply denote this by $\phi$.
 \end{definition}

Here we consider $\arg$ as a function $\arg: \C \to \R$ by setting $\arg(1)=0$. We now turn to our definition of stability, which depends on a choice of central charge $Z$. As in the definition of K-stability of polarised varieties, we require the notion of a test configuration, which is essentially a $\C^*$-degeneration of $(X,L)$ to another polarised scheme.

\begin{definition}\cite{tian-inventiones}\cite[Definition 2.1.1]{donaldson-toric} A \emph{test configuration} for $(X,L)$ consists of a pair $\pi: (\X,\L)\to \C$ where:
\begin{enumerate}[(i)]
\item $\X$ is a normal polarised variety such that $K_{\X}$ is a $\Q$-line bundle;
\item $\L$ is a relatively ample $\Q$-line bundle;
\item there is a $\C^*$-action on $(\X,\L)$ making $\pi$ an equivariant flat map with respect to the standard $\C^*$-action on $\C$;
\item the fibres $(\X_t,\L_t)$ are each isomorphic to $(X,L)$ for each $t \neq 0 \in \C$.
\end{enumerate}
A test configuration is a \emph{product} if $(\X_0,\L_0) \cong (X,L)$, hence inducing a $\C^*$-action on $(X,L)$; it is further \emph{trivial} if this $\C^*$-action is the trivial one.
\end{definition}

\begin{remark} One typically does not require $K_{\X}$ to be a $\Q$-line bundle in the usual definition of a test configuration, but one should not expect this discrepancy to play a significant role in either K-stability or the theory of $Z$-stability we are describing.\end{remark}

A test configuration admits a canonical compactification to a family over $\pr^1$ by equivariantly compactifying trivially over infinity \cite[Section 3]{wang}. This compactification produces a flat family endowed with a $\C^*$-action, which we abusively denote $(\X,\L) \to \pr^1$, such that each fibre over $t \neq \infty \in \pr^1$ is isomorphic to $(X,L)$. The reason to compactify is that it allows us to perform intersection theory on the resulting projective variety $\X$. 

It will also be convenient to be able to consider classes on $X$ as inducing classes on $\X$, so we pass to a variety with a surjective map to $X$ as follows. There is a natural equivariant birational map $$f:(X\times\pr^1,p_1^*L)\dashrightarrow (\X,\scL),$$ with $p_1: X \times \pr^1 \to X$ the projection, so we take an equivariant resolution of indeterminacy of the form:

\[
\begin{tikzcd}
\Y \arrow[swap]{d}{q} \arrow{dr}{r} &  \\
X\times\pr^1 \arrow[dotted]{r}{} & \X,
\end{tikzcd}
\]

\noindent where we may assume $\Y$ is smooth. In particular the unipotent cohomology class $\Theta$ on $X$ involved in the definition of a central charge induces a class $(q \circ p_1)^*\Theta$ on $\Y$, which we still denote $\Theta$. The classes $\L$ and $K_{\X}$ on $\X$ induce also classes $r^*\L$ and $r^*K_{\X/\pr^1}$ on $\Y$, we in addition set $K_{\X/\pr^1} = K_{\X} - \pi^*K_{\pr^1}$ to be the relative canonical class. Thus to a given intersection number $L^d \cdot K_X^j \cdot U$ we can associate the intersection number on $\Y$ which we (slightly abusively) denote $$\int_{\X} \L^{l+1} \cdot K_{\X/\pr^1}^j \cdot \Theta= \int_{\Y} (r^*\L)^{l+1}\cdot r^*(K_{\X/\pr^1}^j ) \cdot \Theta,$$ which is computed in $\Y$. In computing this intersection number, note that $\dim \X = \dim \Y = n+1$. The following elementary result justifies the notation omitting $\Y$.

\begin{lemma} This intersection number is independent of resolution of indeterminacy $\Y$ chosen. \end{lemma}

\begin{proof} Given two such resolutions of indeterminacy $\Y$ and $\Y'$, there is a third resolution of indeterminacy $\Y''$ with commuting maps to both $\Y$ and $\Y'$. The result  then follows from an application of the push-pull formula in intersection theory. \end{proof}

\begin{definition} Let $(\X,\L)$ be a test configuration and $Z$ be a polynomial central charge. We define the \emph{central charge} of $(\X,\L)$ to be $$Z_k(\X,\L) = \sum_{l=0}^n \frac{\rho_l k^{l}}{l+1}\int_{\X}   \L^{l+1} \cdot f(K_{\X/\pr^1}) \cdot \Theta,$$ and set $\phi_k(\X,\L) = \arg Z_{k}(\X,\L)$ when $Z_{k}(\X,\L) \neq 0$. Note that $f(K_{\X/\pr^1}) = \sum_{j=0}^{n}a_jK_{\X/\pr^1}^j$ arises from the polynomial Chern form. With $k$ fixed we denote these by $Z(\X,\L)$ and $\phi(\X,\L)$ respectively.
\end{definition}

The stability condition, for fixed $k$, is then the following. 

\begin{definition} We say that $(X,L)$ is 
\begin{enumerate}[(i)]
\item \emph{$Z$-stable} if for all non-trivial test configurations $(\X,\L)$ we have $$\Ima\left(\frac{Z(\X,\L)}{Z(X,L)}\right)>0.$$
\item \emph{$Z$-polystable} if for all test configurations $(\X,\L)$ we have $$\Ima\left(\frac{Z(\X,\L)}{Z(X,L)}\right)\geq 0,$$ with equality holding only for product test configurations;
\item \emph{$Z$-semistable} if for all test configurations $(\X,\L)$ we have $$\Ima\left(\frac{Z(\X,\L)}{Z(X,L)}\right)\geq 0.$$
\item \emph{$Z$-unstable} otherwise.
\end{enumerate} 
\end{definition} 

The natural asymptotic notion is the following.

\begin{definition} We say that $(X,L)$ is \emph{asymptotically $Z$-stable} if for all non-trivial test configurations $(\X,\L)$ and for all $k \gg 0$ we have $$\Ima\left(\frac{Z_{k}(\X,\L)}{Z_{k}(X,L)}\right)>0.$$
\emph{Asymptotic $Z$-polystability, semistability} and \emph{instability} are defined similarly.
\end{definition} 

Note that, as $\Ima(\rho_n)>0$ by assumption, both $Z_k(X,L)$ and $Z_k(\X,\L)$ are non-vanishing and lie in the upper half plane for $k \gg 0$. Here, strictly speaking to ensure that $Z_k(\X,\L)$ lies in the upper half plane we may need to modify $\L$ to $\L+\scO(m)$ for some $\scO(m)$ pulled back from $\pr^1$; this leaves the various stability inequalities unchanged by Lemma \ref{unchanged} below. Thus asymptotic $Z$-stability can be rephrased as asking for all test configurations $(\X,\L)$ to have for $k \gg 0$ $$\phi_k(\X,\L) > \phi_k(X,L).$$ 

\begin{remark} In Bridgeland stability much work goes into ensuring that the central charge has image in the upper half plane, and this is one of the most challenging aspects of constructing Bridgeland stability conditions. We have essentially ignored this, at the expense of having a notion that should only be the correct one near the ``large volume regime'' when $k$ is taken to be large; this should be thought of as producing a ``large volume'' region in the space of central charges. 

We note that in the better understood story of deformed Hermitian Yang-Mills connections, the link between analysis and a simpler (non-categorical) stability conditions holds in the ``supercritical phase'' \cite{collins-jacob-yau, chen}, which can be thought of as an explicit description of the ``large volume regime''. Away from the large volume situation, it seems likely that categorical techniques must be used and, for example, more structure should be required of the stability vector by analogy with Bayer's hypotheses \cite[Theorem 3.2.2]{bayer-polynomial}. Thus our algebro-geometric definitions should be seen as the first approximation of a larger story, which is appropriate only in an explicit large volume region. 
\end{remark}

The factor $l+1$ in the definition of $Z_{k}(\X,\L)$ ensures that the key inequality defining stability is invariant under certain changes of $\L$. For this, note that one can modify the polarisation of a test configuration $(\X,\L)$ by adding the pullback $\scO(m)$ of the ($m$\textsuperscript{th} tensor power of the) hyperplane line bundle from $\pr^1$ for any $j$.

\begin{lemma}\label{unchanged} The phase inequality remains unchanged under the addition of $\scO(m)$. That is, $$\Ima\left(\frac{Z(\X,\L+\scO(m)))}{Z(X,L)}\right) = \Ima\left(\frac{Z(\X,\L)}{Z(X,L)}\right).$$
\end{lemma}

\begin{proof} A single intersection number changes as $$\int_{\X}(\L+\scO(m))^{l+1} \cdot K_{\X/\pr^1}^j\cdot \Theta = \int_{\X}\L^{l+1} \cdot K_{\X/\pr^1}^j\cdot \Theta + m(l+1) \int_X  L^l \cdot K_{X}^j\cdot \Theta,$$ since by flatness intersecting with $\scO(1)$ can be viewed as intersecting with a fibre $\X_t \cong X$ for $t \neq 0$, and $\L, K_{\X/\pr^1}$ and $\Theta$ restrict to $L, K_X$ and $\Theta$ respectively on $X$. It follows that $$Z(\X,\L+\scO(m)) = Z(\X,\L) + mZ(X,L),$$ which means since $m \in \Q$ is real \begin{align*}\Ima\left(\frac{Z(\X,\L+\scO(m)))}{Z(X,L)}\right) &= \Ima\left(\frac{Z(\X,\L) + mZ(X,L)}{Z(X,L)}\right), \\ &=  \Ima\left(\frac{Z(\X,\L)}{Z(X,L)}\right).\end{align*}\end{proof}

\begin{example}\label{dcsck-centralcharge}A central charge of special interest is \begin{align*}Z_k(X,L) &= -\int_X e^{-ikL}\cdot  e^{-K_X}, \\ &= - \sum_{j=0}^n \frac{(-i)^j }{j!(n-j)!}\int_X(kL)^j\cdot (-K_X)^{n-j}.\end{align*} This can be viewed as an analogue of the central charge on the Grothendieck group $K(X)$ (in the sense of Bridgeland stability) associated to the deformed Hermitian Yang-Mills equation on a holomorphic line bundle \cite[Section 9]{collins-yau}. 
\end{example}

We will not consider a completely arbitrary central charge in the present work, as we require that the large volume limit of our conditions is ``non-degenerate'' in a suitable sense. Let $\Theta_1$ denote the $(1,1)$-part of the unipotent cohomology class $\Theta \in \oplus_j H^{j,j}(X,\C)$.

\begin{definition}\label{non-degenerate-def} We say that $Z$ is
\begin{enumerate}[(i)]
\item \emph{non-degenerate} if   $\Rea(\rho_{n-1})<0$ and $\Theta_1$ vanishes;
\item \emph{of map type} if  $\Rea(\rho_{n-1})<0$ and there is a map $p: X \to Y$ such that $\Theta$ is the pullback of a cohomology class from $Y$ and with $-\Theta_1$ is the class of the pullback of an ample line bundle from $Y$.
\end{enumerate}
\end{definition}

The motivation for these definition is through the link with K-stability and its variants.

\subsubsection{K-stability} The definition of asymptotic $Z$-stability given is motivated not only by the vector bundle theory, but also by the notion of \emph{K-stability} of polarised varieties due to Tian and Donaldson \cite{tian-inventiones, donaldson-toric}. As before, we take $(X,L)$ to be a normal polarised variety such that $K_X$ is a $\Q$-line bundle.

\begin{definition}  We define the \emph{slope} of $(X,L)$ to be the topological invariant, computed as an integral over $X$ $$\mu(X,L) = \frac{-K_X.L^{n-1}}{L^n}.$$ We further define the \emph{Donaldson-Futaki invariant} of a test configuration $(\X,\L)$ to be $$\DF(\X,\L) =\int_{\X}\left( \frac{n\mu(X,L)}{n+1}\L^{n+1} + \L^n.K_{\X/\pr^1}\right).$$ 
\end{definition}

We remark that this is not Donaldson's original definition, but  rather is proven by Odaka and Wang to be an equivalent one \cite[Theorem 3.2]{odaka} \cite[Section 3]{wang} (see also \cite[Proposition 4.2.1]{donaldson-toric}).

\begin{definition} We say that $(X,L)$ is 

\begin{enumerate}[(i)]
\item \emph{K-stable} of for all non-trivial  test configurations $(\X,\L)$ for $(X,L)$ we have $\DF(\X,\L) >0$;
\item \emph{K-polystable} of for all  test configurations we have $\DF(\X,\L) \geq 0$, with equality exactly when $(\X,\L)$ is a product;
\item \emph{K-semistable} of for all  test configurations we have $\DF(\X,\L) \geq 0$;
\item \emph{K-unstable} otherwise. 
\end{enumerate}
\end{definition}

The following is immediate from the definitions.

\begin{lemma} K-semistability is equivalent to asymptotic $Z$-semistability where $$Z_k(X,L) = \int_X (ik^nL^n -k^{n-1}K_X.L^{n-1}).$$ That is, with $\rho = (0,0, \hdots, -1,i)$, $\Theta=0$.
\end{lemma}

Of course, the same is true for K-stability and K-polystability, modulo our slightly non-standard requirement that $K_{\X}$ is a $\Q$-line bundle, which is irrelevant for K-semistability as in that situation one can assume $\X$ is smooth.

\begin{example} K-semistability of \emph{maps} can recovered as a special cases of $Z$-stability. Indeed, supposing $p: (X,L) \to (Y,H)$ is a map of polarised varieties, then setting $$Z_{k}(X,L) = \int_X (ik^nL^n -k^{n-1}(K_X+p^*H).L^{n-1})$$ recovers the notion of \emph{K-semistability of the map $p$} \cite[Definition 2.9]{stablemaps}. That is, we take $\Theta$ to be (the class of) $p^*H$.
\end{example}

Slightly more generally, twisted K-stability fits into this picture \cite[Definition 2.7]{uniform}, though this notion is less geometric than K-stability of maps and we hence do not discuss it. Similarly, the ``fully degenerate'' case $a_j = 0$ for $j\leq n-1$ produces variants of J-stability \cite[Section 2]{lejmi-szekelyhidi} and has links with $Z$-stability of holomorphic line bundles \cite[Conjecture 1.6]{DMS}. In general, asymptotic $Z$-stability is related to K-stability as follows:

\begin{proposition}\label{prop:largevolstability} For an arbitrary central charge $Z$, asymptotic $Z$-semistability implies
\begin{enumerate}[(i)]
\item K-semistability if $Z$ is non-degenerate;
\item K-semistability of the map $p$ if $Z$ is of map type.
\end{enumerate}
\end{proposition}

\begin{proof} We only give the proof for K-semistability, as the proof is the same for the map type situation. By non-degeneracy, there is an expansion $$Z_k(X,L) = k^n  i\int_{X}L^n + k^{n-1} \rho_{n-1} \int_{X}K_X.L^{n-1} + O(k^{n-2}),$$ where we have used that $\Theta_1=0$ and that our normalisation for the polynomial Chern form assumes $a_0=a_1=1$. Thus $$Z_k(\X,\L) = \frac{i}{n+1}k^n\int_{X}\L^{n+1} + \frac{\rho_{n-1}}{n}k^{n-1} \int_{X}K_{\X/\pr^1}.L^{n-1} + O(k^{n-2}),$$ meaning that $$\Ima\left(\frac{Z_{k}(\X,\L)}{Z_{k}(X,L)}\right) =\frac{-\Rea(\rho_{n-1})}{nL^n}\DF(\X,\L)k^{-1} + O(k^{-2}).$$ Thus since $\Rea(\rho_{n-1})<0$ by non-degeneracy, the asymptotic $Z$-stability hypothesis demands that this be negative for $k \gg 0$, forcing $\DF(\X,\L) \geq 0$.
\end{proof}

\subsection{$Z$-critical K\"ahler metrics} We now turn to the differential-geometric counterpart of stability, and thus assume that $(X,L)$ is a \emph{smooth} polarised variety. We wish to define a notion of a ``canonical metric'' in $c_1(L)$, adapted to the central charge $Z$. We recall our notation that the central charge takes the form $$Z_{k}(X,L) =\sum_{l=0}^n \rho_l k^l   \int_X  L^l \cdot \left(\sum_{j=0}^n a_jK_X^j\right) \cdot \Theta,$$ with the induced phase being denoted $\phi_k(X,L) = \arg Z_{k}(X,L);$ we take $k$ to be fixed and omit it from our notation. 

Associated to any K\"ahler metric $\omega \in c_1(L)$ is its Ricci form $$\Ric \omega = -\frac{i}{2\pi} \ddbar \log \omega^n \in c_1(X) = c_1(-K_X)$$ and a Laplacian operator $\Delta$. We also fix a representative of the unipotent class $\Theta$, which we denote $\theta \in \Theta$. When $Z$ is non-degenerate in the sense of Definition \ref{non-degenerate-def}, so that $\Theta_1=0$, we always take the $(1,1)$-component $\theta_1 \in \Theta_1$ to vanish, and similarly when $Z$ is of map type we take $\theta_1$ to be the pullback of a K\"ahler metric from $Y$. To the intersection number $L^l \cdot (-K_X)^j \cdot \Theta$ we associate the  function \begin{equation}\label{eqndef}\frac{\omega^l\wedge \Ric\omega^j\wedge \theta}{\omega^n}-\frac{j}{l+1}\Delta\left( \frac{\omega^{l+1}\wedge \Ric\omega^{j-1}\wedge\theta}{\omega^n}\right) \in C^{\infty}(X,\C),\end{equation} with the second term taken to be zero when $j=0$. The presence of the Laplacian terms  will be crucial to link with the algebraic geometry. By linearity, this produces a function $\tilde Z(\omega)$ defined in such a way that $$\int_X \tilde Z(\omega) \omega^n = Z(X,L);$$ as with our algebro-geometric discussion, we always assume that $Z(X,L) \neq 0$.

\begin{definition} We say that $\omega$ is a $Z$-\emph{critical K\"ahler metric} if $$\Ima(e^{-i\phi(X,L)} \tilde Z(\omega)) = 0$$ and the positivity condition $\Rea(e^{-i\phi(X,L)} \tilde Z(\omega))>0$ holds.
\end{definition}

When we consider a $k$-dependent central charge $Z_k$, we define $\tilde Z_k(\omega)$ by replacing $\omega$ with $k\omega$. We view this as a partial differential equation on the space of K\"ahler metrics in $c_1(L)$, or equivalently on the space of K\"ahler potentials with respect to a fixed K\"ahler metric. Viewed on the space of K\"ahler potentials, for a generic choice of central charge ensuring the presence of a non-zero term involving the Laplacian, the equation is a sixth-order fully-nonlinear partial differential equation. The condition is equivalent to asking that the function $$\tilde Z(\omega): X \to \C$$ has constant argument, which must then equal that of $Z(X,L) \in \C$, as we have assumed the positivity condition $\Rea(e^{-i\phi(X,L)} \tilde Z(\omega))>0$ (in fact one only needs that this function is never zero, and the sign is irrelevant).

\begin{remark} The presence of the Laplacian term is crucial to obtain a link with algebraic geometry, and in practice arises when deriving the $Z$-critical equation as the Euler-Lagrange equation of an associated energy functional in Proposition \ref{euler-lagrange-derivation}. \end{remark}

\begin{remark} In the vector bundle theory, rather than working with arbitrary connections one works with ``almost-calibrated connections'' \cite[Section 8.1]{collins-yau}. This is a positivity condition which depends on the choice of $\theta \in \Theta$ and which is trivial in the large volume limit \cite[Lemma 2.8]{DMS}, and is analogous to the positivity condition $\Rea(e^{-i\phi(X,L)} \tilde Z(\omega))>0$ that we have imposed. The notion of a ``subsolution'' also plays a prominent role in the bundle theory \cite{collins-jacob-yau}, which for example forces the equation to be elliptic in that situation \cite[Lemma 2.32]{DMS}. We note that, also in the manifold case, ellipticity of the $Z$-critical equation cannot hold in general, and hence for this reason and others it is natural to ask if there is a manifold analogue of the notion of a subsolution. \end{remark}

The appearance of the phase is justified by the following.

\begin{lemma}\label{integral-vanishes} For any K\"ahler metric $\omega \in c_1(L)$, the integral $$\int_X\Ima(e^{-i\phi(X,L)} \tilde Z(\omega))\omega^n = 0$$ vanishes. \end{lemma}

\begin{proof} Since $\int_X\tilde Z(\omega) =  Z(X,L)$ and $\phi(X,L) = \arg(Z(X,L)$, we see $$e^{-i\phi(X,L)} = \frac{r(X,L)}{Z(X,L)}$$ with $r(X,L)$ real. Thus $$\int_X\Ima(e^{-i\phi(X,L)} \tilde Z(\omega))\omega^n  =\Ima\left( \frac{r(X,L)}{Z(X,L)} Z(X,L)\right)  = 0.$$ \end{proof}

The $Z$-critical condition can be reformulated as follows. The analogous reformulation, in the special case of the deformed Hermitian Yang-Mills equation \cite{jacob-yau}, has been crucial to all progress in understanding the equation geometrically, and an analogous reformulation holds for $Z$-critical connections on holomorphic line bundles \cite[Example 2.24]{DMS}.

\begin{lemma} Write $$\tilde Z(\omega) = \Rea\tilde Z(\omega) + i\Ima \tilde Z(\omega).$$ Then $\omega$ is a $Z$-critical K\"ahler metric if and only if $$\arctan \left(\frac{\Ima \tilde Z(\omega)}{\Rea\tilde Z(\omega)}\right) = \phi(\omega) \textrm{ mod } 2\pi \Z.$$\end{lemma}

\begin{proof} We calculate $$\Ima(e^{-i\phi(X,L)} \tilde Z(\omega)) = \Ima\left(e^{-i\phi(X,L)}  \exp\left(i\arctan\left( \frac{\Ima \tilde Z(\omega)}{\Rea\tilde Z(\omega)}\right)\right) \right),$$ which vanishes if and only if  $$\arctan \left(\frac{\Ima \tilde Z(\omega)}{\Rea\tilde Z(\omega)}\right) = \phi(X,L) \textrm{ mod } 2\pi \Z.$$\end{proof}

\begin{example}\label{dcscK-anal} Consider the central charge $$Z(X,L) = -\int_X e^{-iL}\cdot  e^{-K_X} = - \sum_{j=0}^n \frac{(-i)^j }{j!(n-j)!}\int_XL^j\cdot (-K_X)^{n-j}$$ described in Example \ref{dcsck-centralcharge}. The induced representative $\tilde Z(\omega)$ is given by $$\tilde Z(\omega) = - \sum_{j=0}^n \frac{(-i)^j }{j!(n-j)!}\left( \frac{\omega^{n-j} \wedge \Ric\omega^{j}}{\omega^n} - \frac{j}{n-j+1}\Delta\left( \frac{\Ric\omega^{j-1} \wedge \omega^{n-j+1}}{\omega^n}\right)\right),$$ which produces what one might call the \emph{deformed cscK equation}  \begin{equation}\label{dcscK}\Ima(e^{-i\phi(X,L)}\tilde Z(\omega)) =0,\end{equation} which is the manifold analogue of the deformed Hermitian Yang-Mills equation on a holomorphic line bundle. Strictly speaking this equation does not conform to our normalisation of the central charge, but the central charge $-n!(-i)^{3n+1}\overline {Z(X,L)}$ (with $\overline {Z(X,L)}$ denoting the complex conjugate of $Z(X,L)$), which produces an equivalent partial differential equation, does.

Each component of this equation, of the form $$\frac{\Ric\omega^j \wedge \omega^{n-j}}{\omega^n} - \frac{j}{ n-j+1}\Delta\left( \frac{\Ric\omega^{j-1} \wedge \omega^{n-j+1}}{\omega^n}\right),$$ has appeared previously in the work of Chen-Tian \cite[Definition 4.1]{chen-tian} and Song-Weinkove \cite[Section 2]{song-weinkove} in relation to the K\"ahler-Ricci flow. To understand the equation more fully, choose a point $p$ and normal coordinates at $p$ so that $\Ric \omega$ is diagonal with diagonal entries $\lambda_1,\hdots,\lambda_n$. Letting $\sigma_j(\omega)$ be the $j^{th}$ elementary symmetric polynomial in these eigenvalues, so that $$(\omega + t\Ric\omega)^n = \sum_{j=0}^nt^j\sigma_j(\omega)\omega^n,$$ the deformed cscK equation takes the much simpler form $$\Ima\left(e^{-i\phi(X,L)}\left(\sum_{j=0}^n (-i)^j(\sigma_j(\omega) - \Delta\sigma_{j-1}(\omega))\right)\right) =0.$$ This is a close analogue of the deformed Hermitian Yang-Mills equation on a holomorphic line bundle, but the presence of the terms involving the Laplacian seems to present significant new challenges.

We also remark that Schlitzer-Stoppa have studied a coupling of the deformed Hermitian Yang-Mills equation to the constant scalar curvature equation \cite{stoppa-schlitzer}, which should be related to a combination of Bridgeland stability of the bundle and K-stability of the polarised variety, and which is of quite a different flavour to Equation \eqref{dcscK}.

\end{example}

We now focus on the large volume regime of the $Z$-critical equation. 

\begin{lemma}\label{large-volume} Suppose the central charge $Z_k$ is of map type, with $\theta_1 \in \Theta_1$ a real $(1,1)$-form. Then there is an expansion as $k \to \infty$ of the form $$\Ima(e^{-i\phi_k} \tilde Z_k(\omega)) = k^{-1}(\Rea(\rho_{n-1})L^n)(S(\omega) - \Lambda_{\omega}\theta_1 - n\mu_{\Theta_1}(X,L)) + O(k^{-2}),$$ where $\mu_{\Theta_1}(X,L) = \frac{-L^{n-1}.(K_X+\Theta_1)}{L^n}.$ 
 \end{lemma} 
 
 \begin{proof} We first calculate $$\Ima\left(\frac{\tilde Z_k(\omega)}{Z_k(X,L)}\right) = \frac{\Ima \tilde Z_k(\omega) \Rea  Z_k(X,L) - \Rea \tilde Z_k(\omega)  \Ima  Z_k(X,L) }{\Rea  Z_k(X,L)^2 + \Ima  Z_k(X,L)^2}.$$ Since \begin{align*}Z_k(X,L) &= iL^nk^n + \rho_{n-1}L^{n-1}.(K_X+\Theta_1)k^{n-1} + O(k^{n-2}), \\  \tilde Z_k(\omega) &= i - \frac{\rho_{n-1}}{n}(S(\omega) - \Lambda_{\omega}\theta_1)k^{-1} + O(k^{-2}),\end{align*} this is given by $$ \Ima\left(\frac{\tilde Z_k(\omega)}{Z_k(X,L)}\right) = k^{-n-1}(\Rea(\rho_{n-1})(S(\omega) - \Lambda_{\omega}\theta_1 - n\mu_{\Theta_1}(X,L)) + O(k^{-n-2}).$$ Writing $Z_k(X,L) = r_ke^{i\phi_k},$ we have $$\Ima(e^{-i\phi_k(X,L)} \tilde Z_k(\omega)) = r_k(X,L) \Ima\left(\frac{\tilde Z_k(\omega)}{Z_k(X,L)}\right),$$ which since $r_k = L^nk^n + O(k^{n-1})$ implies the result.\end{proof}

Thus, up to multiplication by the non-zero (in fact strictly negative) constant $\Rea(\rho_{n-1})$, the ``large volume limit'' of the $Z$-critical equation is the \emph{twisted cscK equation} $$S(\omega) - \Lambda_{\omega}\theta_1 =n\mu_{\Theta_1}(X,L);$$ the geometry of this equation  is linked with that of the map $p: X \to Y$  \cite[Section 4]{extremal}, where we have assumed $\theta_1$ is the pullback of a K\"ahler metric from $Y$ since the central charge is of map type. 

This result can be seen as a differential-geometric counterpart to Proposition \ref{prop:largevolstability}. When $Z$ is actually  non-degenerate, it follows that the ``large volume limit'' of the $Z$-critical equation is the cscK equation, whereas on the algebro-geometric side, Proposition \ref{prop:largevolstability} shows that asymptotic $Z$-semistability implies K-semistability, so that K-stability is the ``large volume limit'' of asymptotic $Z$-stability. In order to more fully understand the links between the various concepts, we will later be interested in the analytic counterpart to K-semistability:

\begin{definition} We say that $(X,L)$ is \emph{analytically K-semistable} if there is a test configuration $(\X,\L)$ for $(X,L)$ for which $(\X_0,\L_0)$ is a smooth polarised variety which admits a cscK metric. \end{definition}

It is conjectured that a K-semistable polarised variety admits a test configuration whose central fibre is K-polystable. The assumption of analytic K-semistability is thus  a smoothness assumption, since a smooth K-polystable polarised variety is itself expected to admit a cscK metric. It follows from work of Donaldson that analytically K-semistable varieties are actually K-semistable \cite[Theorem 2]{donaldson-lower}.

\section{$Z$-critical metrics on asymptotically $Z$-stable manifolds}\label{mainresults}

Here we prove our main result:

\begin{theorem}
Let $Z$ be an admissible central charge. Suppose that $(X,L)$ is a polarised variety with discrete automorphism group which is analytically K-semistable, and suppose the deformation theory of its cscK degeneration is unobstructed.  Then if $(X,L)$ is asymptotically $Z$-stable, $(X,L)$ admits $Z_k$-critical K\"ahler metrics for all $k \gg 0$.
\end{theorem}

We will also state and prove a local converse, namely that existence implies stability in a local sense, later in Section \ref{converse}. Here we consider only the case that the central charge $Z$ involves powers of  $K_X$ and no higher Chern classes, with the general case, in which the equation has a different flavour, being dealt with in Section \ref{sec:higherrank}. In comparison with the statement in the introduction, we are varying the central charge by $k$ rather than scaling $L$; these operations are clearly equivalent.

Unobstructedness of the deformation theory of the cscK degeneration of $(X,L)$ will be used to ensure its Kuranishi space is smooth (as discussed in Section \ref{sec:kuranishi}); this allows us to only consider genuine complex manifolds rather than almost complex manifolds in the analysis.

Admissibility requires three conditions. All of these conditions hold in the case of the deformed cscK equation described in Example \ref{dcscK-anal}. Firstly, we require that $Z$ is non-degenerate, meaning the large volume limit of the $Z$-critical equation is the cscK equation. Secondly, with the central charge given by $$Z_{k}(X,L) =\sum_{l=0}^n \rho_l k^{l}   \int_X  L^l \left(\sum_{j=0}^n a_j K_X^j\right) \cdot \Theta,$$ we require that $\Rea(\rho_{n-1})<0, \Rea(\rho_{n-2})>0$ and $\Rea(\rho_{n-3})=0$. We also assume that $a_j=1$ for all $j$ for simplicity, though all that one needs is that the real parts are positive for $j=0,1,2,3$. These assumptions are used to control the behaviour of the linearisation of the equation. We expect that the condition on $\rho_{n-3}$ can be removed.

The third condition concerns the form $\theta \in \Theta$. A basic technical assumption we make is that $\theta_2 = \theta_3=0$, though we also expect this assumption can be removed. We furthermore require that $\theta$ extends to a smooth, equivariant  form on the test configuration $(\X,\L)$ producing the cscK degeneration of $(X,L)$ (which exists by analytic K-semistability), and also that $\theta$ extends to certain other deformations of $(\X_0,\L_0)$. More precisely, as we will recall in Section \ref{sec:kuranishi}, the Kuranishi space of $\X_0$ admits an action of $\Aut(\X_0,\L_0)$, and we require that $\theta$ extends smoothly to an equivariant form on the universal family over the Kuranishi space. The condition is modelled on the bundle situation \cite{DMS}, where the differential forms $\theta$ are forms on the base $Y$ of the vector bundle $E$. Then if the polystable degeneration of $E$ is $F$, there is still a map $F \to Y$, meaning one can still make sense of the relevant equation on  $F$ over $Y$.

\subsection{Preliminaries on analytic Deligne pairings}\label{sec:deligne}

As outlined in the Introduction, there are three steps to our work. The final step is to solve an abstract finite dimensional problem in symplectic geometry, whereas the first two steps involve reducing to this finite dimensional problem. A key tool for the first two steps is the theory of analytic Deligne pairings, established in \cite[Section 4]{kahler} and \cite[Section 2.2]{zak}, which give a direct approach to the properties of Deligne pairings in algebraic geometry. The additional flexibility of analytic Deligne pairings will allow us to include the extra forms $\theta$ into the theory, which do not fit into the usual algebro-geometric approach. Although the techniques developed in \cite{kahler, zak} are essentially equivalent, our discussion is closer to that of Sj\"ostr\"om Dyrefelt \cite{zak}. 

The setup is simple case of the general theory, where we have a fixed smooth polarised variety; in general one considers holomorphic submersions. We thus let $(X,L)$ be a smooth polarised variety of dimension $n$ and suppose that $\eta_0, \hdots \eta_{n-p}$ are $n-p+1$ closed $(1,1)$-forms on $X$.  Any other forms $\eta_j' \in [\eta_j]$ are of the form $\eta_j' = \eta_{j} + i\ddbar \psi_j$ for some real-valued function $\psi_j$. We in addition suppose that $\theta$ is a closed real $(p,p)$-form on $X$which we will not be varied in our discussion and which has cohomology class $[\theta] = \Theta$. In our application we will allow $\theta$ to be a closed \emph{complex} $(p,p)$-form, but linearity of our constructions will allow us to reduce to the real case.

\begin{definition} We define the  \emph{Deligne functional}, denoted $$\langle \psi_0, \hdots, \psi_{n-p};\theta\rangle \in \R,$$ by \begin{align*}\langle& \psi_0, \hdots, \psi_{n-p};\theta\rangle = \int_{X}\psi_0(\eta_1+i\ddbar \psi_1) \wedge \hdots \wedge (\eta_{n-p}+i\ddbar \psi_{n-p})\wedge\theta \\ &+ \int_{X} \psi_1\eta_0\wedge(\eta_2+i\ddbar \psi_2) \wedge \hdots \wedge(\eta_{n-p}+i\ddbar \psi_{n-p})\wedge \theta+ \hdots  \\ &+ \int_{X} \psi_{n-p}\eta_0\wedge\hdots \wedge  \eta_{n-p-1}\wedge \theta.\end{align*} \end{definition}

The Deligne functional can be considered as an operator taking $n-p+1$ functions to a real number. The definition is due to Sj\"ostr\"om Dyrefelt \cite[Definition 2.1]{zak} and is implicit in \cite[Section 4]{kahler}, in both cases with $\theta=0$. The inclusion of $\theta$ makes essentially no difference to the fundamental properties of the functional. 

\begin{proposition}\label{properties-deligne} The Deligne functional $\langle\psi_0, \hdots, \psi_{n-p};\theta\rangle$ satisfies the following properties:
\begin{enumerate}[(i)]
\item it is symmetric in the indices $0,1,\ldots, n-p$;
\item it satisfies the ``change of potential'' formula $$\langle\psi_0', \hdots, \psi_{n-p}';\theta\rangle - \langle\psi_0, \hdots, \psi_{n-p};\theta\rangle= \int_{X} (\psi_0'-\psi_0)(\eta_1 + i\ddbar \psi_1) \wedge \hdots \wedge (\eta_m + i\ddbar \psi_{n-p})\wedge \theta,$$ and analogous formulae hold when varying other $\psi_j$.

\end{enumerate}
\end{proposition}

\begin{proof}
$(i)$ This follows from an integration by parts formula when $\theta=0$ \cite[Proposition 2.3]{zak}, and the proof in the general case is identical. The reason is that our form $\theta$ is fixed, so the fact that it is a form of higher degree is irrelevant.

$(ii)$ This is immediate from the definition; this property is really the motivation for the chosen definition. Note that the statement when one changes any other $\psi_j$ follows from the symmetry described as $(i)$. \end{proof}

We will be interested in the behaviour of Deligne functionals in families. The most basic property of these functionals in families is the following.

\begin{proposition}\label{hessian} Suppose $B$ is a complex manifold, and let $\pi: X \times B \to B$ be the projection, and write $\eta_0,\hdots,\eta_{n-p},\theta$ as the forms on $X \times B$ induced by pullback of the corresponding forms on $X$. Let $\psi_0,\hdots,\psi_{n-p}$ be functions on $X \times B$, and denote by $$ \langle \psi_0,\hdots,\psi_{n-p};\theta\rangle_{B}: B \to \R$$ the function of $b \in B$ $$ \langle \psi_0,\hdots,\psi_{n-p};\theta\rangle_{B}(b) = \langle \psi_0|_{X\times\{b\}},\hdots,\psi_{n-p}|_{X\times\{b\}};\theta\rangle_{X\times\{b\}},$$ where this denotes the Deligne functional computed on the fibre $X\times\{b\}$ over $b \in B$. Then $$\int_{X\times B/B} (\eta_0 +i\ddbar \psi_0) \wedge \hdots \wedge  (\eta_{n-p} +i\ddbar \psi_{n-p})\wedge\theta = i\ddbar \langle \psi_0,\hdots,\psi_{n-p};\theta\rangle_{B}.$$
\end{proposition}

This result will produce K\"ahler potentials for natural K\"ahler metrics produced on holomorphic submersions via fibre integrals.

A closely related property of Deligne functionals allows the differential-geometric computation of intersection numbers on the total space of test configurations. To explain this, consider a test configuration $(\X,\L) \to \C$ with central fibre $\X_0$ smooth. It is equivalent to work with test configurations over twice the unit disc $2\Delta \subset \C$ (with the $\C^*$-action then meant only locally on $2\Delta$), and we will sometimes pass between the two conventions. The use of $2\Delta$ is only for notational convenience, to ensure $1 \in 2\Delta$. Fixing a fibre $\X_1 \cong X$, we obtain a form $\beta(t).\theta$ on $\X\setminus\X_0$ which we assume extends to a smooth form with cohomology class $\Theta$ on $\X$, where $\beta(t)$ denotes the $\C^*$-action on $\X$. 

Let $\Omega_0,\Omega_1,\hdots,\Omega_{n-p}$ be $S^1$-invariant forms on $\X$ with $[\Omega_0],[\Omega_1],\hdots,[\Omega_{n-p}]$ $\C^*$-invariant cohomology classes on $\X$. Thus $$\beta(t)^*\Omega_j - \Omega_j = i\ddbar \psi_j^t$$ for some smooth family of functions $\psi_j^t$ depending on $t$, with $\psi_0$ induced by the analogous procedure using $\omega_{\X}$. We next restrict $\psi_j^t$ to our fixed fibre $\X_1 = X$. Set $\tau = -\log|t|^2$, so that $\tau \to \infty$ corresponds to $t\to 0$. The following then links the differential geometry with the intersection numbers of interest. To link with the algebraic geometry to come, we assume $\Omega_j \in c_1(\L_j)$ for some line bundles $\L_j$ on $\X$, though this is not essential.

\begin{lemma}\label{slope-formula} \cite[Theorem 4.9]{zak}\cite[Theorem 1.4]{kahler} We have $$\int_{\X}\L_0 \cdot\L_1\cdot\hdots\cdot \L_{n-p}\cdot \Theta = \lim_{\tau \to \infty} \frac{d}{d\tau} \langle\psi_0^{\tau}, \hdots, \psi_{n-p}^{\tau};\theta\rangle(X).$$
\end{lemma}

Here the intersection number on the left hand side is computed over the compactification of the test configuration $\X \to \pr^1$. The value on the right hand side is the value of the Deligne functional on $X=\X_1$. This Lemma is proven in \cite{zak, kahler} only in the case $\theta=0$, but as above the inclusion of the class $\Theta$ makes no difference to the proofs as $\theta$ extends smoothly to a $\C^*$-invariant form on $\X_0$ by assumption.

\subsection{The $Z$-energy}\label{sec:z-energy}

We next fix a smooth polarised variety $(X,L)$. We fix the value $k$, so that the central charge takes the form $$Z(X,L) =\sum_{l=0}^n \rho_l \int_X  L^l \cdot \left(\sum_{j=0}^na_jK_X^j \right)\cdot \Theta.$$ We also fix a K\"ahler metric $\omega \in c_1(L)$ and denote by $\H_{\omega}$ the space of K\"ahler potentials with respect to $\omega$. We then wish to define an energy functional $$E_Z: \H_{\omega} \to \R$$ whose Euler-Lagrange equation is the $Z$-critical equation.

We proceed by first defining a functional $$F_Z: \H_{\omega} \to \C$$ using the central charge, and then define $$E_Z(\psi) = \Ima(e^{-\phi}F_Z(\psi)).$$ Our process is linear in the $(n,n)$-forms involved in the definition of $\tilde Z(\omega)$, so we fix a term $\int_X  L^l \cdot K_X^j \cdot \Theta,$ where we may assume $\Theta$ is a real cohomology class of degree $(n-l-j,n-l-j)$ again by linearity. 

For this fixed term, we can use the theory of Deligne functionals to produce the desired functional $$F_{Z,l}: \H_{\omega} \to \R.$$ Our reference metric $\omega$ induces a reference form $\Ric \omega \in c_1(X)$. Any potential $\psi \in \H_{\omega}$ with associated K\"ahler metric $\omega_{\psi} = \omega+ i\ddbar \psi$ induces a change in Ricci curvature $$\Ric(\omega_{\psi}) - \Ric(\omega) = i\ddbar\log\left(\frac{\omega^n}{\omega_{\psi}^n}\right).$$ Thus the theory of Deligne functionals over a point (i.e., taking the base $B$ to be a point) produces a value \begin{equation}\label{deligne-pairing-term}\frac{1}{l+1}\left\langle\overbrace{\psi,\hdots, \psi}^{l+1 \text{ times}},\overbrace{\log\left(\frac{\omega^n}{\omega_{\psi}^n}\right),\hdots,\log\left(\frac{\omega^n}{\omega_{\psi}^n}\right)}^{j \text{ times}}; \theta\right\rangle\in \R\end{equation} associated to our term $\int_X  L^l \cdot K_X^j \cdot \Theta$ involved in the central charge. We emphasise that we are abusing notation slightly; $\theta$ is really only one component of the full representative of the unipotent class $\Theta$. But by linearity, with real and imaginary terms handled separately, this produces the functional $E_Z: \H_{\omega} \to \R$, whose Euler-Lagrange equation we must calculate.

\begin{remark} In the special case $\theta=0$, the Deligne functional given by Equation \eqref{deligne-pairing-term} was introduced by Chen-Tian  \cite[Section 4]{chen-tian} in relation to the K\"ahler-Ricci flow, where it was defined through its variation. Song-Weinkove later showed that these functionals can, again in the case $\theta=0$, be obtained through Deligne pairings \cite[Section 2.1]{song-weinkove}. Their work highlights the analytic significance of these functionals. Collins-Yau have introduced an energy functional designed to detect the existence of deformed Hermitian Yang-Mills connections on holomorphic line bundles \cite[Section 2]{collins-yau}, and their functional bears some formal  similarities with our $Z$-energy.

 \end{remark}

\begin{definition} We define the $Z$\emph{-energy} to be the functional $E_Z: \H_{\omega} \to \R$ associated to the central charge $Z$.\end{definition}

\begin{remark}It is straightforward to check that $E_Z(\psi+c) = E_Z(\psi)$, so that one can view $E_Z$ as a functional on K\"ahler metrics rather than K\"ahler potentials. \end{remark}

The most important aspect of the $Z$-energy is its Euler-Lagrange equation.

\begin{proposition}\label{euler-lagrange-derivation}  Given a path of metrics $\psi_t \in \H_{\omega}$ with associated K\"ahler metric $\omega_t$, we have $$\frac{d}{dt} E_Z(\psi_t) = \int \dot \psi_t\Ima(e^{-i\phi}\tilde Z(\omega_t))\omega_t^n.$$ Thus the Euler-Lagrange equation for the $Z$-functional is the $Z$-critical equation.
\end{proposition}

\begin{proof} By linearity, it suffices to calculate the variation of the operator $F_{Z,l}: \H_{\omega} \to \R$, given through Equation \eqref{deligne-pairing-term} as a Deligne pairing, along the path $\psi_t$. We will demonstrate that this variation is given by $$\frac{d}{dt}F_{Z,l}(\psi_t) = \int_X\dot \phi_t\left( \frac{\omega_t^{l}\wedge\Ric\omega_t^{j}\wedge \theta}{\omega_t^n} -\frac{j}{l+1 } \Delta_t\left(\frac{\Ric\omega_{t}^{j-1}\wedge \omega_{t}^{l+1}\wedge \theta}{\omega_t}\right) \right)\omega_t^n,$$ which will imply the result we wish to prove, since by definition of $\tilde Z(\omega_{t})$ it is a sum of terms of the form $$\tilde Z_l(\omega_{t}) = \frac{\omega_t^{l}\wedge\Ric\omega_t^{j}\wedge \theta}{\omega_t^n} -\frac{j}{l+1} \Delta_t\left(\frac{\Ric\omega_{t}^{j-1}\wedge \omega_{t}^{l+1}\wedge \theta}{\omega_t}^n\right).$$

 The calculation from here is closely analogous to that of Song-Weinkove \cite[Proposition 2.1]{song-weinkove}, who proved the desired variational formula when $\theta=0$. By the change of potential formula, our functional is given by \begin{align*}(l+1)F_{Z,l}(\psi) &= \sum_{m=0}^{l}\int_X \phi \omega_{\psi}^m\wedge \Ric\omega^{j}\wedge \omega^{l-m}\wedge \theta \\ & + \sum_{m=0}^{j-1} \int_X  \log\left(\frac{\omega^n}{\omega_{\psi}^n}\right)\Ric(\omega_{\psi})^m\wedge \Ric\omega^{j-m-1}\wedge\omega_{\psi}^{l+1}\wedge \theta.\end{align*} Differentiating along the path $\psi_t$ gives \begin{align*}(l+1)\frac{d}{dt}F_{Z,l}(\psi_t)&= \sum_{m=0}^{l}\int_X \dot\phi_t \omega_{t}^m\wedge \Ric\omega^{j}\wedge \omega^{l-m}\wedge \theta 
\\ & +\sum_{m=0}^{l}m\int_X \phi_t i\ddbar\dot \psi_t\omega_{t}^{m-1}\wedge \Ric\omega^{j}\wedge \omega^{l-m}\wedge \theta 
\\ & - \sum_{m=0}^{j-1} \int_X  \Delta_t\dot\psi_t\Ric(\omega_{t})^m\wedge \Ric\omega^{j-m-1}\wedge\omega_{t}^{l+1}\wedge \theta
\\ &- \sum_{m=0}^{j-1} m\int_X  \log\left(\frac{\omega^n}{\omega_{t}^n}\right) i\ddbar\Delta_t\dot\psi_t\wedge \Ric\omega_{t}^{m-1}\wedge \Ric\omega^{j-m-1}\wedge\omega_{t}^{l+1}\wedge \theta
\\ & + \sum_{m=0}^{j-1} (l+1)\int_X  \log\left(\frac{\omega^n}{\omega_{t}^n}\right)\Ric\omega_{t}^m\wedge \Ric\omega^{j-m-1}\wedge i \ddbar \dot \psi_t\wedge\omega_{t}^{l}\wedge \theta,\end{align*} where $ \Delta_t$ is the Laplacian with respect to the volume form $\omega_t$, and where any term with negative exponent is taken to vanish. We note that this is self-adjoint with respect to $\omega_t^n$. We use that $$i\ddbar \psi_t = \omega_t - \omega, \qquad i\ddbar \log\left(\frac{\omega^n}{\omega_{t}^n}\right) = \Ric\omega_t - \Ric\omega$$ and the self-adjointness of the Laplacian just mentioned to obtain 

\begin{align*}(l+1)\frac{d}{dt}F_{Z,l}(\psi_t)&= \sum_{m=0}^{l}\int_X \dot\phi_t \omega_{t}^m\wedge \Ric\omega^{j}\wedge \omega^{l-m}\wedge \theta 
\\ & +\sum_{m=0}^{l}m\int_X \psi_t( \omega_t - \omega)\omega_{t}^{m-1}\wedge \Ric\omega^{j}\wedge \omega^{l-m}\wedge \theta 
\\ & - \sum_{m=0}^{j-1} \int_X  \dot\psi_t\Delta_t\left(\frac{\Ric(\omega_{t})^m\wedge \Ric\omega^{j-m-1}\wedge\omega_{t}^{l+1}\wedge \theta}{\omega_t^n}\right)\omega_t^n
\\ &- \sum_{m=0}^{j-1} m\int_X \dot\psi_t\Delta_t\left(\frac{(\Ric\omega_t - \Ric\omega) \wedge \Ric\omega_{t}^{m-1}\wedge \Ric(\omega)^{j-m-1}\wedge\omega_{t}^{l+1}\wedge \theta}{\omega_t^n}\right)\omega_t^n
\\ & + \sum_{m=0}^{j-1} (l+1)\int_X \dot \psi_t (\Ric\omega_t - \Ric\omega)\wedge\Ric\omega_{t}^m\wedge \Ric\omega^{j-m-1}\wedge\omega_{t}^{l}\wedge \theta,\end{align*} where we use that $\theta$ is a closed form. We consider first the two terms involving Laplacians, which we see equal  \begin{align*}&- \sum_{m=0}^{j-1} (m+1)\int_X  \dot\psi_t\Delta_t\left(\frac{\Ric\omega_{t}^m\wedge \Ric\omega^{j-m-1}\wedge\omega_{t}^{l+1}\wedge \theta}{\omega_t^n}\right)\omega_t^n
\\ &+ \sum_{m=0}^{j-1} m\int_X \dot\psi_t\Delta_t\left(\frac{\Ric\omega_{t}^{m-1}\wedge \Ric\omega^{j-m}\wedge\omega_{t}^{l+1}\wedge \theta}{\omega_t^n}\right)\omega_t^n  
\\ &= -j\int_X  \dot\psi_t\Delta_t\left(\frac{\Ric\omega_{t}^{j-1}\wedge \omega_{t}^{l+1}\wedge \theta}{\omega_t^n}\right)\omega_t^n.\end{align*} One similarly calculates that the remaining three terms sum to $$(l+1)\int_X \dot \phi_t \omega_t^{l}\wedge\Ric\omega_t^{j}\wedge \theta,$$ meaning that \begin{align*}\frac{d}{dt}F_{Z,l}(\psi_t) &= -\frac{j}{l+1}\int_X  \dot\psi_t\Delta_t\left(\frac{\Ric(\omega_{t})^{j-1}\wedge \omega_{t}^{\alpha}\wedge \theta}{\omega_t^n}\right)\omega_t^n + \int_X \dot \psi_t \omega_t^{l}\wedge\Ric\omega_t^{j}\wedge \theta, \\ &= \int_X\dot \psi_t\left( \frac{\omega_t^{l}\wedge\Ric\omega_t^{j}\wedge \theta}{\omega_t^n} -\frac{j}{l+1} \Delta_t\left(\frac{\Ric(\omega_{t})^{j-1}\wedge \omega_{t}^{l+1}\wedge \theta}{\omega_t^n}\right) \right)\omega_t^n, \end{align*} which is what we wanted to show. \end{proof}

We now suppose that $(\X,\L)$ is a test configuration for $(X,L)$ with smooth central fibre (so that the total space is smooth), and with $\omega_{\X} \in c_1(\L)$ a relatively K\"ahler $S^1$-invariant metric. As will eventually be important, it follows by Ehresmann's theorem that $X_0$ is diffeomorphic to $X$. The relatively K\"ahler metric $\omega_{\X}$ induces a Hermitian metric on the relative holomorphic tangent bundle $T_{\X/\C}$. Here $T_{\X/\C}$ exists as the test configuration has smooth central fibre, meaning that $\pi: \X \to \C$ is a holomorphic submersion. This induces a metric on the relative anti-canonical class $-K_{\X/\C}$ whose curvature we denote $\rho$. Following the process explained immediately before Lemma \ref{slope-formula}, we set $$\beta(t)^*\omega_{\X} - \omega_{\X} = i\ddbar \psi_t.$$

Let $Jv$ be the real holomorphic vector field inducing the $S^1$-action on $\X$ preserving $\omega_{\X}$, and define a function $h$ on $\X$ by $$\L_{v}\omega_{\X} = i\ddbar h,$$ so that $\dot\psi_0 = h$. The form $\omega_{\X}$ restricts to an $S^1$-invariant K\"ahler metric $\omega_0$ on $\X_0$. 

\begin{lemma}\cite[Equation 2.1.4]{gauduchon} The function $h$ restricted to $\X_0$ is a Hamiltonian function with respect to $\omega_0$. \end{lemma}

Note that $\omega_{\X}$ is merely a K\"ahler form on each fibre, hence not actually a symplectic form on $\X$;  nevertheless, one could call $h$ the Hamiltonian even in this situation. In what follows we will also use the related property that \begin{equation}\label{second-equation}\frac{d}{dt}\beta(t)^*\omega_{\X}= i\ddbar \beta(t)^*h,\end{equation} see \cite[Example 4.26]{szekelyhidi-book}. We can now relate the $Z$-energy to the algebro-geometric invariants of interest.

\begin{proposition}\label{slope-for-z} We have equalities
$$\int_{\X_0} h\Ima(e^{-i\phi}\tilde Z(\omega_0))\omega_0^n =  \lim_{\tau \to \infty} \frac{d}{d\tau} E_Z(\phi_{\tau})= \Ima\left(\frac{Z(\X,\L)}{Z(X,L)}\right).$$\end{proposition}

\begin{proof} The second equality is an immediate consequence of our definition of $E_Z$ through Deligne functionals and Lemma \ref{slope-formula}, using  that $$E_Z(\psi) = \Ima(e^{-i\phi}F_Z(\psi)) = \Ima\left(\frac{F_Z(\psi)}{Z(X,L)}\right),$$ which is analogous to the fact used in Lemma \ref{integral-vanishes}. To prove the first equality, unravelling the definition of $\tau$ and the variational formula for $E_Z$ proven in Proposition \ref{euler-lagrange-derivation} with $\omega_t = \beta(t)^*\omega_{\X}|_{X=\X_1}$ we see that $$\frac{d}{d\tau} E_Z(\phi_{\tau}) = \int_{X} (\beta(t)^*h)\Ima(e^{-\phi}\tilde Z(\omega_t))\omega_t^n,$$ where we have used that $\frac{d}{dt}\omega_t = i\ddbar \beta(t)^*h$ by Equation \eqref{second-equation}. But  $$\int_{\X_1} (\beta(t)^*h)\Ima(e^{-i\phi}\tilde Z(\omega_t))\omega_t^n =  \int_{\X_t} h\Ima(e^{-i\phi}\tilde Z(\omega_{\X}|_{\X_t}))\omega_{\X}|_{\X_t}^n,$$ which converges to $\int_{\X_0} h\Ima(e^{-i\phi}\tilde Z(\omega_{0}))\omega_{0}^n$ as $t \to 0$, proving the result.\end{proof}

This also produces an analogue of the classical Futaki invariant associated to a holomorphic vector field. 

\begin{corollary} Suppose $(X,L)$ admits a $Z$-critical K\"ahler metric, and suppose there is an $S^1$-action on $(X,L)$. Then for any $S^1$-invariant K\"ahler metric $\omega\in c_1(L)$ with associated Hamiltonian $h$ we have $$\int_{X} h\Ima(e^{-i\phi}\tilde Z(\omega))\omega^n = 0.$$\end{corollary}

\begin{proof} Note that a product test configuration, just as with any other test configuration, can be compactified to a family $(\X,\L)$ over $\pr^1$. By the previous result, this integral is actually independent of $\omega \in c_1(L)$ as it equals $$\int_{X} h\Ima(e^{-i\phi}\tilde Z(\omega))\omega^n = \Ima\left(\frac{Z(\X,\L)}{Z(X,L)}\right),$$ which is patently independent of $\omega$. But if $\omega'$ is the $Z$-critical K\"ahler metric, the corresponding integral on the left hand side clearly vanishes, as desired. \end{proof}

\subsection{Moment maps} \label{moment-map-section}

\subsubsection{Moment maps in finite dimensions}\label{moment-map-finite}

Many geometric equations have an interpretation through moment maps; this has been especially influential for the cscK equation. We will give two ways of viewing the $Z$-critical equation as a moment map. The first is a finite-dimensional geometric interpretation, on the base of a holomorphic submersion, while the second is closer in spirit to the infinite dimensional viewpoint of Fujiki-Donaldson for the cscK equation \cite{fujiki,donaldson-moment}. 

The setup is modelled on the situation of a test configuration $\pi: (\X,\L)\to \C$ for $(X,L)$ with smooth central fibre. The properties of interest are firstly that there is an $S^1$-action on both $\C$ and $(\X,\L)$, making $\pi$ an $S^1$-equivariant map, secondly that all fibres over the open dense orbit under the associated $\C^*$-action are isomorphic, and thirdly that we may choose an $S^1$-invariant relatively K\"ahler metric $\omega_{\X} \in c_1(\L)$. If we had considered test configurations over the unit disc $\Delta$, the same properties would be true with the $\C^*$-action meant only locally, in the sense that one only obtains an action induced by sufficiently small elements of the Lie algebra of $\C^*$.

More generally, we consider a holomorphic submersion $\pi: (\X,\L) \to B$ over a complex manifold $B$, with $\L$ a relatively ample $\Q$-line bundle. We assume that $B$ admits an effective action of a compact Lie group, which induces an effective local action of the complexification $G$ of $K$. In addition we assume that there is a $K$-action on $(\X,\L)$ making $\pi$ an equivariant map, and fix a $K$-invariant relatively K\"ahler metric $\omega_{\X} \in c_1(\L)$. We lastly assume that there is an open dense orbit associated to $G$ such that all fibres are isomorphic to $(X,L)$; we denote this orbit as $\X^o \to B^o$. In practice, we will apply these results to the special case of an isotrivial family.

Let $v$ be a holomorphic vector field on $\X$ induced by an element of the Lie algebra $\mfk$ of $K$. Denote by $h_v$ the function on $\X$ defined by the equation \begin{equation}\label{use-later}\L_{Jv} \omega_{\X} = i\ddbar h,\end{equation} where $J$ denotes the almost-complex structure of $\X$ and the differentials on the right hand side are also computed on $\X$.  As in Section \ref{sec:deligne} we will refer to $h$ as a Hamiltonian, even though $\omega_{\X}$ is only relatively K\"ahler. Note that $h_v$ does restrict to a genuine Hamiltonian for $v$ on the fibres over $B$ on which $v$ induces a holomorphic vector field; these are the fibres over points for which the corresponding vector field on $B$ vanishes.

We now fix the input of the $Z$-critical equation. Setting $\epsilon = k^{-1}$, our central charge can be written  $$Z_{\epsilon}(X,L) =\sum_{l=0}^n \rho_l \epsilon^{-l}   \int_X  L^l \cdot f(K_X) \cdot \Theta.$$ We have fixed a representative $\theta \in \Theta$, and we assume that the form $G.\theta$ defined on the dense orbit $\X^o$ extends to a smooth form on $\X$ itself, and denote this form abusively by $\theta$, which is automatically a $G$-invariant closed form on $\X$.  The form $\omega_{\X}$ induces a metric on the relative holomorphic tangent bundle $T_{\X/B}$, and hence on its top exterior power $-K_{\X/B}$, and we denote the curvature of the latter metric as $\rho \in c_1(-K_{\X/B})$.

We associate to $Z_{\epsilon}(X,L)$ an $(n+1,n+1)$-form on $\X$ as follows. We will define the $(n+1,n+1)$-form on $\X$ linearly in the terms of this expression, and hence it is sufficient to define an $(n+1,n+1)$-form associated to a term of the form $\int_X  L^l \cdot K_X^j \cdot \Theta$, to which we associate $\frac{1}{l+1}\omega_{\X}^{l+1}\wedge \rho^j \wedge \theta.$ This induces a form $\tilde Z_{\epsilon}(\X,\L)$, and we set \begin{equation}\label{fibreintegralforms}\Omega_{\epsilon} = \Ima\left(e^{-i\phi_{\epsilon}} \int_{\X/B} \tilde Z(\X,\L)\right)\end{equation}  to be the associated fibre integral. By general properties of fibre integrals, this produces a closed $(1,1)$-form on $B$. $K$-invariance of the forms on $\X$ and of the map $\pi: \X \to B$ imply that $\Omega_{\epsilon}$ is $K$-invariant. In general, the form $\Omega_{\epsilon}$ may not be K\"ahler, which in addition requires positivity. In our applications, $\Omega_{\epsilon}$ will however be K\"ahler for $0<\epsilon\ll 1$.

We let $\omega_b$ denote the restriction of $\omega_{\X}$ to the fibre $\X_b$ over $b$, and denote $\Ima(e^{-i\phi_{\epsilon}}\tilde Z_{\epsilon}(\omega_b))$ the $Z$-critical operator computed on $\X_b$ with respect to $\omega_b$. We similarly set $h_{v,b}$ be the restriction of a Hamiltonian $h_v$ to the fibre $\X_b$. 
Define a map $\mu_{\epsilon}: B \to \mfk^*$ by $$\langle \mu_{\epsilon}, v \rangle(b) =-\frac{1}{2} \int_{\X_b} h_{v,b} \Ima(e^{-i\phi_{\epsilon}}\tilde Z_{\epsilon}(\omega_b)) \omega_b^n,$$ where $v \in \mfk$ is viewed as inducing a holomorphic vector field on $\X$ to induce the Hamiltonian $h_v$.

\begin{theorem}\label{moment-map-in-finite-dimensions-thm} $\mu_{\epsilon}$ is a moment map with respect to the $K$-action on $B$ and with respect to the form $\Omega_{\epsilon}$.\end{theorem}

Here we mean that the defining conditions of a moment map are satisfied, namely that $$d\langle \mu_{\epsilon}, v \rangle = - \iota_v \Omega_{\epsilon},$$ and $\mu$ is $K$-equivariant when $\mfk^*$ is given the coadjoint action; in general we emphasise that $\Omega_{\epsilon}$ is not actually a symplectic form (although  for $\epsilon$ sufficiently small it will be in our applications, producing genuine moment maps). In the contraction $\iota_v \Omega_{\epsilon}$ we view $v\in \mfk$ as inducing a holomorphic vector field on $B$.

\begin{proof}

We first show that the equation $d\langle \mu_{\epsilon}, v \rangle = - \iota_v \Omega_{\epsilon}$ holds. Note that it is enough to show that this holds on the dense locus $B^o$, since both sides of the equation extend continuously to $B$. 

We fix a point $b \in B$ at which we wish to demonstrate the moment map equation, and consider the orbit $U$ of $b \in B$ under the $G$-action. We fix an isomorphism $(\X_b,\L_b) \cong (X,L)$ and simply write $\omega_b = \omega$. The $G$-action induces an isomorphism $$(\X,\L) \cong (X \times U, L)$$ where $B^o \cong U \subset G$ is a submanifold. Since we only obtain a local action of $G$ on $B$, $U$ may not consist of all of $G$ in general. The isomorphism $\X \cong X \times U$ is in addition compatible with the projections to $B^o$. The relatively K\"ahler metric $\omega_{\X} \in c_1(\L)$ thus induces a form $\omega_{X\times U}$ on $X \times U$ and we can define $$E_Z: U \to \R$$ defined as the $Z$-energy with respect to the reference metric $\omega$ (or rather its pullback to $X \times U$) and the varying metric $\omega_{X \times U}$ on the fibres over $U$.  Proposition \ref{hessian} then implies that on $U \cong B^0$ we have  \begin{equation}\label{truelocally}i\ddbar E_{Z_{\epsilon}} = \Omega_{\epsilon}.\end{equation}

By Proposition \ref{euler-lagrange-derivation}, the derivative of the $Z$-energy along any path $\omega_t = \omega + i\ddbar \psi_t$ satisfies $$\frac{d}{dt}E_{Z_{\epsilon}}(\psi_t) = \int_X \dot \psi_t  \Ima(e^{-i\phi_{\epsilon}}\tilde Z_{\epsilon}(\omega_t))\omega_t^n.$$ Considering the path $\omega_t$ defined above, the defining property of the Hamiltonian $h$ means that $\dot \psi_0 = h_{v,b}$ on $\X_b \cong X$. Thus  $$\frac{d}{dt}\Bigr|_{t=0}\int_X\dot\psi_t\Ima(e^{-i\phi_{\epsilon}}\tilde Z(\omega_t))\omega_t^n =  \int_X h_{v,b}\Ima(e^{-i\phi_{\epsilon}}\tilde Z(\omega))\omega^n.$$  But this is then all we need: by a standard calculation \cite[Lemma 12]{szekelyhidi-I} $$\iota_v(i\ddbar E_{Z_{\epsilon}}) = \frac{1}{2}d(Jv(E_{Z_{\epsilon}})),$$ so it follows that \begin{equation}\label{contraction-eqn}\iota_v(\Omega_{\epsilon}) = \frac{1}{2}d(Jv(E_{Z_{\epsilon}})) = \frac{1}{2}d\left(\frac{d}{dt}E_{Z_{\epsilon}}(\exp(Jvt).p)\right) = \frac{1}{2}d\left(\frac{d}{dt}E_{Z_{\epsilon}}(\psi_t)\right) = -d\langle \mu_{\epsilon},h\rangle,\end{equation} proving the first defining property of a moment map with respect to $v$ at the point $b$. But by continuity  this implies that the same property holds on all of $B$.

What remains to prove is $K$-equivariance of $\mu_{\epsilon}$, which requires us to show for all $g \in K$ $$\langle \mu_{\epsilon}(g.b),v\rangle = \langle \mu_{\epsilon}(g.b),g^{-1}.v\rangle,$$ where $K$ acts on $\mfk$ by the adjoint action. However the Hamiltonian on $\X$ with respect to $g^{-1}.v$ is simply the pullback $g^*h_v$, meaning $g^*h_{v,g(b)} = h_{g^{-1}.v,b}$. Thus, using $K$-invariance of $\omega_{\X}$, the equality $$ \int_{\X_{g(b)}} h_{v,{g(b)}} \Ima(e^{-i\phi_{\epsilon}}\tilde Z_{\epsilon}(\omega_{g(b)})) \omega_{g(b)}^n =  \int_{\X_b} g^*h_{v,g(b)} \Ima(e^{-i\phi_{\epsilon}}\tilde Z_{\epsilon}(\omega_b)) \omega_b^n$$ is enough to imply equivariance. \end{proof}

\begin{remark} All results in this section hold assuming less regularity than smoothness, for example considering $L^2_k$-K\"ahler metrics for $k$ sufficiently large.

\end{remark}

\begin{remark}Our approach is partially inspired by an early use of Deligne pairings by Zhang, where he considered a family of K\"ahler metrics induced by embeddings of a fixed projective variety into projective space, and where he showed that the existence of a balanced embedding is equivalent to Chow stability \cite{zhang}. The approach above seems to be new even in the limiting case that $Z_{\epsilon}(X,L) = \epsilon^{-n}\int_X (iL^n -\epsilon K_X.L^{n-1})$, where this gives a new interpretation of the scalar curvature as a moment map. 
\end{remark}

\subsubsection{Moment maps in infinite dimensions}\label{sec:infinite}

We next demonstrate how the $Z$-critical equation appears as a moment map in infinite dimensions. While this is important conceptually, it will also be used in order to understand the linearisation of the $Z$-critical equation. Analogously to the Fujiki-Donaldson moment map interpretation of the cscK equation \cite{fujiki,donaldson-moment}, we fix a compact symplectic manifold $(M,\omega)$ and consider the space  $\J(M,\omega)$ of almost complex structures compatible with $\omega$. We refer to Scarpa \cite[Section 1.3]{scarpa} and Gauduchon \cite[Section 8]{gauduchon} for a good exposition of this space its properties .  $\J(M,\omega)$ naturally has the structure of an infinite dimensional complex manifold; its tangent space at $J \in \J(M,\omega)$ is given by $$T_J\J(M,\omega)= \{A: TM \to TM \ | \ AJ+JA =0,\  \omega(\cdot , A \cdot) = \omega(A \cdot, \cdot)\},$$ with complex structure defined by $A \to JA$ on $T_J\J(M,\omega)$. At an almost complex structure $J$, the tangent space can be identified with $\Omega^{0,1}(TX^{1,0})$, the space of $(0,1)$-forms with values in holomorphic vector fields \cite[p. 14]{scarpa}.

We let $\G$ denote the group of exact symplectomorphisms of $(M,\omega)$, which acts naturally on $\J(M,\omega)$. The Lie algebra of $\G$ can be identified with $C^{\infty}_0(M)$, the functions which integrate to zero, through the Hamiltonian construction. For $h \in C^{\infty}_0(M)$, we denote by $v_h$ the associated Hamiltonian vector field. The infinitesimal action of $\G$ is then given by  \begin{align}\begin{split}\label{infinitesimal-action} P: C^{\infty}_0(X,\R)& \to T_J\J(M,\omega), \\ P h &= \L_{v_h} J.\end{split}\end{align} Under the identification of $T_J\J(M,\omega)$ with $\Omega^{0,1}(TX^{1,0})$, the operator $P$ corresponds to the operator \cite[Lemma 1.4.3]{scarpa}\begin{align}\begin{split}\label{Ddef}\D: C^{\infty}_0(X,\R)& \to \Omega^{0,1}(TX^{1,0}), \\ \D h &= \bar\partial\nabla^{1,0} h.\end{split}\end{align}  The operator $\D$ plays a central role in the theory of cscK metrics. Note, for example, that its kernel consists of functions generating global holomorphic vector fields, there are called \emph{holomorphy potentials}.

Now let $(X,L)$ be a smooth polarised variety with complex structure $J_X \in \J(M,\omega)$. We assume for the moment that $\Aut(X,L)$ is trivial, and will later consider the other case of interest for our main results, namely that $\Aut(X,L)$ is finite. We denote by $\J_X(M,\omega) \subset \J(M,\omega)$ the set of $J' \in \J(M,\omega)$ such that there is a diffeomorphism $\gamma$, which lies in the connected component of the identity inside the space of diffeomorphisms of $M$,  with $\gamma \cdot J' = J_X$. Thus $\J_X(M,\omega)$ corresponds to complex structures producing manifolds biholomorphic to $X$. This space is discussed by Gauduchon \cite[Section 8.1]{gauduchon}; for us an important point will be that $\J_X(M,\omega)$ is actually a complex submanifold of $\J(M,\omega)$ \cite[Proposition 8.2.3]{gauduchon}. As in the work of Fujiki \cite[Section 8]{fujiki}, the space $\J(M,\omega)$ admits a universal family $(\U,\L_{\U}) \to \J(M,\omega)$ which hence restricts to a family $(\U,\L_{\U})$ over $\J_{X}(M,\omega)$.  The fibre over a complex structure $J_b \in \J(M,\omega)$ is simply the complex manifold $(M,J_b)$. 

We next induce a form $\theta_{\U}$ on $\U \to \J_{X}(M,\omega)$ associated to the form $\theta$ on $\U$, using the fact that each fibre of $\U\to \J_X(M,\omega)$ is isomorphic to $X$. For any $B \subset \J_{X}(M,\omega)$ a finite dimensional complex submanifold, the Fischer-Grauert theorem produces an isomorphism $\U|_B \cong X \times B$ commuting with the maps to $B$. One can extend this isomorphism to an isomorphism of line bundles \begin{equation}\label{family-isom}\Psi_B: (\U|_B, \L|_B) \cong (X,L) \times B,\end{equation} perhaps after shrinking $B$ \cite[Lemma 5.10]{newstead} (while the proof given by Newstead assumes algebraicity of $B$, it also holds in the holomorphic category \cite[Lemma 6.3]{hallam}). Then as we have assumed $\Aut(X,L)$ is actually trivial, the isomorphism $\Psi_B$ is actually unique. Pulling back $\theta$ on $X$ via $\Psi_B$ induces a closed form $\theta_B$ on $\U|_B$, and uniqueness then means that the forms $\theta_B$ glue to a closed form $\theta_{\U}$ on all of $\U$. Similarly, using these isomorphisms, the $Z$-energy induces a function \begin{equation}\label{local-z-energy}E_Z: \J_X(M,\omega) \to \R\end{equation} after fixing the reference K\"ahler metric $\omega$ on $X$.

Denote by $\Omega_{\epsilon}$ the family of closed $(1,1)$-forms on $\J_X(M,\omega)$ given by \begin{equation}\label{infinite-form}\Omega_{\epsilon} =\Ima\left(e^{-i\phi_{\epsilon}}  \int_{\U/\J_X(M,\omega)} \tilde Z_{\epsilon}(\U,\L_{\U})\right),\end{equation} where $\tilde Z_{\epsilon}(\U,\L_{\U})$ is defined just as in Equation \eqref{fibreintegralforms} using the relatively K\"ahler metric $\omega_{\X} \in c_1(\L)$, the  form $\rho \in c_1(-K_{\U/ \J_{X}(M,\omega)})$ induced by the relatively K\"ahler metric $\omega_{\X}$ and $\theta_{\U}$. The forms $\Omega_{\epsilon}$ are then  closed $\G$-invariant $(1,1)$-forms which are not, however, positive in general. 

The $Z$-critical operator can be viewed as a function $$\Ima(e^{-i\phi_{\epsilon}}\tilde Z): \J_X(M,\omega) \to C^{\infty}_0(X),$$ which we wish to demonstrate is a moment map with respect to the $\Omega_{\epsilon}$. Thus we need to understand the behaviour of $\Ima(e^{-i\phi_{\epsilon}}\tilde Z)$ under a change in complex structure. We will use a similar idea to Section \ref{moment-map-finite}, namely to realise the $Z$-energy as a K\"ahler potential, which requires us to relate the change in complex structure to the change in metric structure. Consider a path $J_t \in  \J_X(M,\omega)$, and let $F_t \cdot J_t = J_X$ for $F_t$ diffeomorphisms of $X$. Then we obtain a corresponding path of K\"ahler metrics $F_t^*\omega = \omega_t = \omega+i\ddbar \psi_t$ compatible with $J_X$. Then the key fact we need is that the path $J_t$ satisfies \cite[p. 1083]{szekelyhidi-deformations} \begin{equation}\label{change-in-metric}\frac{d}{dt}\Bigr|_{t=0} J_t = JP\dot \psi_0.\end{equation}

\begin{theorem}\label{infinite-dimensional-moment-maps-thm} The map $$\mu_{\epsilon} = \Ima(e^{-i\phi_{\epsilon}}\tilde Z): \J_X(M,\omega) \to C^{\infty}_0(X)$$ is a moment map for the $\G$-action on $\J_X(M,\omega)$ with respect to the forms $\Omega_{\epsilon}$.
\end{theorem}

Here the statement means that the moment map condition is satisfied, note again that $\Omega_{\epsilon}$ may not actually be positive (hence K\"ahler) in general.

\begin{proof} Fix a point $b \in \J_X(M,\omega)$ at which we wish to demonstrate the moment map property, and let $Ph$ be the tangent vector at $b$ induced the element $h \in \Lie \G \cong C^{\infty}_0(M)$. We show that for any finite dimensional complex submanifold $B \subset \J_X(M,\omega)$ containing $Ph$, the moment map equality $$-\iota_{Ph} \Omega_{\epsilon} = d\langle \mu_{\epsilon}, Ph \rangle$$ holds. The proof of this is essentially the same as that of Theorem \ref{moment-map-in-finite-dimensions-thm}. 

Perhaps after shrinking $B$, the family $(\U_B,\L_B) \to B$ satisfies \begin{equation}\label{the-isom-again}(\U_B, \L_B) \cong (X,L) \times B\end{equation} by the argument of Equation \eqref{family-isom}. We thus obtain a function $$E_Z: B \to \R$$ by Equation \ref{local-z-energy}, which by the argument of Theorem \ref{moment-map-in-finite-dimensions-thm} satisfies $$i\ddbar E_Z = \Omega_{\epsilon},$$ an equality of $(1,1)$-forms on $B$. Since this holds for each $B$, it also holds on $\J_X(M,\omega).$

Consider a path $J_{b_t} \in B$ of almost complex structures such that the induced tangent vector at $t=0$ is given by $JPh$. Then we obtain a corresponding path of K\"ahler metrics $\omega_t = \omega + i\ddbar \psi_t$ through the isomorphism of Equation \eqref{the-isom-again}, and Equation \eqref{change-in-metric} implies that $\dot \psi_0 = h$. It follows that $$\frac{d}{dt}\Bigr|_{t=0} E_Z(J_t) = \int_X h  \Ima(e^{-i\phi_{\epsilon}}\tilde Z(J_b))\omega^n,$$ which means that as functions on $\J_X(M,\omega)$ we have $$\langle dE_Z, JPh\rangle = \int_X h  \Ima(e^{-i\phi_{\epsilon}}\tilde Z(J_b))\omega^n = -\langle \mu_{\epsilon}(b), Ph\rangle.$$ Then the same argument as Equation \eqref{contraction-eqn} implies that $$\iota_{Ph} \Omega_{\epsilon} = \iota_{Ph} i\ddbar E_Z = -d\langle \mu_{\epsilon}(b), Ph\rangle,$$ which proves the defining equation of the moment map. 

The $\G$-action on $\Lie(\G)$ is the adjoint action, which corresponds to pullback of Hamiltonians \cite[Equation (1.5)]{scarpa}. Then equivariance follows by the same argument as Theorem \ref{moment-map-in-finite-dimensions-thm}. 
\end{proof}

\begin{remark} Gauduchon has given another proof that the scalar curvature is a moment map on $\J_X(M,\omega)$ in a similar spirit, but using more direct properties of the Mabuchi functional rather than Deligne pairings \cite[Proposition 8.2]{gauduchon}.\end{remark}

While positivity is not guaranteed for all $\epsilon$, it will be important to have positivity for finite-dimensional submanifolds and for $\epsilon$ sufficiently small. 

\begin{proposition}\label{fujiki-result-fibre} Let $B \subset \J_X(M,\omega)$ be a complex submanifold. Then $\Omega_{\epsilon}$ restricts to a K\"ahler metric for all $0<\epsilon\ll 1$.

 \end{proposition}

\begin{proof} Fujiki proves that  the form $$\Omega_0 = - \int_{\U/\J(M,\omega)} \rho \wedge\omega^n + \frac{n}{n+1}\mu(X,L) \int_{\U/\J(M,\omega)}\omega^{n+1}$$ is actually a K\"ahler metric on $\J(M,\omega)$, and agrees with the usual K\"ahler metric on $\J(M,\omega)$ used in the moment map interpretation of the scalar curvature on $\J(M,\omega)$ \cite[Theorem 8.3]{fujiki}. Thus since $$Z_{\epsilon}(X,L) = \epsilon^{-n}\int_X (iL^n +\Rea(\rho_{n-1})\epsilon K_X.L^{n-1}) + O(\epsilon^{-n+2}),$$ we have $$\Omega_{\epsilon} = \Ima\left(e^{-i\phi_{\epsilon}} \int_{\X/B} \tilde Z(\X,\L)\right)= -\epsilon\Rea(\rho_{n-1})n\Phi^*\Omega + O(\epsilon^2),$$ which implies the result.

One can also prove Fujiki's result, namely the equality of the fibre integral $\Omega_0$ and the usual K\"ahler metric on $\J_X(M,\omega)$, directly, giving another proof.  By \cite[Equation 8.1.10]{gauduchon}, the tangent space $T_J\J_X(M,\omega)$ is spanned by elements of the form $Ph, JPh$, for $h \in C^{\infty}_0(M,\omega)$. But it follows from the argument of Theorem \ref{infinite-dimensional-moment-maps-thm} that the moment map for the $\G$-action on $\J_X(M,\omega)$ is given by the scalar curvature. But then since $$\iota_{Ph}\Omega_{\J} = \iota_{Ph}\Omega_0$$ for all $h \in C^{\infty}_0(M)$, it follows that the forms actually agree on $\J_X(M,\omega)$.\end{proof}

In particular, if $B \subset \J_X(M,\omega)$ is a complex submanifold invariant under $K\subset \G$, we obtain a genuine sequence of moment maps $\mu_{\epsilon}$ for $\epsilon\ll 0$ with respect to genuine K\"ahler metrics $\Omega_{\epsilon}$.

\begin{remark}\label{automorphisms-used}In the case $\Aut(X,L)$ is non-trivial, but still finite, we denote by $G =\Aut(X,L)$, assume $\theta$ is $G$-invariant and work $G$-equivariantly. Let $\J_X(M,\omega)^G$ denote the fixed locus of the $G$-action on $\J_X(M,\omega).$ Then while the isomorphisms$$(\U_B, \L_B) \cong (X,L) \times B$$ of Equation \ref{the-isom-again}, which were used to construct the form $\theta_{\U}$ on $\U$ and function $E_Z$ on $\J_X(M,\omega)$ are no longer unique, they are unique up to the action of $G$. But since $\theta$ is $G$-invariant by assumption, working on $\J_X(M,\omega)^G$ instead allows us to construct functions $E_Z$ on $\J_X(M,\omega)^G$ and a form $\theta_{\U^G}$ on $\U^G \to \J_X(M,\omega)^G$. The proof of the moment map property is then identical to the case $G$ is trivial.\end{remark}

\begin{remark}\label{smooth-singular}As with the previous section, all results in this section hold assuming less regularity than smoothness, for example considering $L^2_k$-complex structures for $k$ sufficiently large.

\end{remark}

\subsection{The core analytic argument}\label{finite-dim-reduction}

We next turn to analytic aspects of the $Z$-critical equation necessary to prove our main result, for which we assume $Z$ is admissible. We assume here that $(X,L)$ admits a cscK metric, and construct extremal solutions of the $Z$-critical equation for $k\gg 0$. In particular, this gives a general construction of $Z$-critical metrics. The reason we produce extremal solutions is that we allow $(X,L)$ to have automorphisms; in the proof of our main theorem, we will apply these techniques to the cscK degeneration of the manifold of interest. At a key point in proving this result, we will assume that the manifold is a cscK degeneration of a polarised manifold with discrete automorphism group, and will emphasise this point when it arises.

\begin{theorem} Suppose $(X,L)$ admits a cscK metric and is a degeneration of a polarised manifold with discrete automorphism group. Then $c_1(L)$ admits solutions of the equation $$\bar\partial\nabla_{\epsilon}^{1,0}(\Ima (e^{-i\phi_{\epsilon}} \tilde Z_{\epsilon}(\omega_{\epsilon}))=0$$ for all $0<\epsilon \ll 1$.\end{theorem}
 
We call solutions of this equation $Z$-extremal metrics. Here the gradient is defined using $\omega_{\epsilon}$, so the condition asks that $\Ima (e^{-i\phi_{\epsilon}} \tilde Z_{\epsilon}(\omega_{\epsilon})$ is a holomorphy potential.  In particular, if $\Aut(X,L)$ is discrete, this gives a general construction of $Z$-critical K\"ahler metrics. 

The difficulty in proving this result is that the $Z$-critical equation is a sixth-order PDE, while the cscK equation is fourth-order. The basic idea to circumvent this is to use quantitive inverse function theorem, analogously to a strategy of Hashimoto for another problem in K\"ahler geometry \cite{hashimoto}. This requires us to firstly construct approximate solutions of the $Z$-critical equation, which is straightforward. The main difficulty is then to understand the mapping properties of the linearised operator, which occupies much of the current section. When $(X,L)$ admits automorphisms, the kernel of the linearised operator is nontrivial, which forces us to consider the more general $Z$-extremal equation, much as in the classical work of LeBrun-Simanca \cite{lebrun-simanca}. When applying this result to the case of an analytically K-semistable manifold, we will then employ Kuranishi theory as in the fundamental work of Br\"onnle and Sz\'ekelyhidi \cite{bronnle, szekelyhidi-deformations}, and will use the moment map property of the $Z$-critical equation and a version of the Kempf-Ness theorem to apply our assumption of asymptotic $Z$-stability.

We consider the $Z$-critical operator as an operator on the space of K\"ahler potentials with respect to a reference metric $\omega \in c_1(L)$ \begin{align*}G_{\epsilon}&: \H_{\omega} \to \R, \\ G_{\epsilon}(\psi) &= \Ima(e^{-i\phi_{\epsilon}} \tilde Z_{\epsilon}(\omega+i\ddbar \psi)).\end{align*} We set $\omega_{\psi} = \omega+i\ddbar \psi.$ The main goal will be to understand the mapping properties of the linearisation of $G_{\epsilon}$ and variants of $G_{\epsilon}$, in order to apply a quantitative version of the implicit function theorem.

We recall our central charge takes the form $$Z_{k}(X,L) =\sum_{l=0}^n \rho_l \epsilon^{-l}   \int_X  L^l \left(\sum_{j=1}^n a_j K_X^j\right) \cdot \Theta.$$ The simplest, but rather degenerate, case of this equation is when $a_j = 0$ for all $j \geq 2$, which means that the terms in the definition of the $Z$-critical equation involving the Laplacian vanishes; see Equation \eqref{eqndef}. In this case, for $\epsilon \ll 1$ the equation is a \emph{fourth} order elliptic partial differential equation. In the general case which is of interest to us, the equation jumps from a fourth order equation at $\epsilon=0$ to a sixth order equation for $\epsilon>0$, which causes several additional analytic difficulties.  

\begin{lemma}\label{ellipticity-lemma} Suppose $\rho_{n-2} \neq 0$ and $a_2 \neq 0$. Then for all $0 < \epsilon \ll 1$, the $Z$-critical equation is a sixth order elliptic partial differential equation. \end{lemma}

\begin{proof} Clearly $G_{\epsilon}$ is a sixth order partial differential operator as  $\rho_{n-2} \neq 0$ and $a_2 \neq 0$, and we must show that it is elliptic, which means that we must show that its linearisation is elliptic. This is a condition on the highest order derivatives, so we replace the $Z$-critical operator with the sum of the terms involving six derivatives. Since we are interested in the case $\epsilon \ll 1$, we need only consider the lowest order terms in $\epsilon.$ When one scales $0 \ll \epsilon < 1$, the lowest order term in $\epsilon$ is then $$\psi \to  c\Delta_{\psi}\left(\frac{\Ric\omega_{\psi} \wedge \omega_{\psi}^{n-1}}{\omega_{\psi}^n}\right),$$ where $\Delta_{\psi}$ is the Laplacian with respect to $\omega_{\psi}$ and $c \neq 0$, since any forms involving the unipotent class $\Theta$ will be of higher order in $\epsilon$. By the product rule, the linearisation of this operator along the path $t\psi$ is given by $$\Delta^3 \psi + \textrm{ lower order derivatives},$$ since the linearisation of the scalar curvature operator is given by $$\frac{d}{dt}\Bigr|_{t=0} S(\omega+i\ddbar \psi) = \Delta^2\psi - S(\omega)\Delta\psi + n(n-1) \frac{i\ddbar \psi \wedge \Ric\omega \wedge \omega^{n-2}}{\omega^n}.$$ This demonstrates ellipticity. \end{proof}

Note that the condition $a_2 \neq 0$ is part of our hypothesis that $Z$ is admissible, used to prove our main result.

\subsubsection{Understanding the model operator} Let $$\F_{\epsilon}:  C^{\infty}(X,\R) \to \R$$ denote the linearisation of the $Z$-critical operator $G_{\epsilon}$. In order to understand the mapping properties of $\F_{\epsilon}$, we will compare it to a simpler model operator. Much as with the linearisation of the scalar curvature, a key operator will be the operator $$\D\psi = \bar\partial \nabla^{1,0}\psi,$$ as mentioned in Equation \eqref{Ddef}, whose kernel $\ker \D $ consists of functions inducing holomorphic vector fields on $X$. We denote the vector space of such functions, namely the holomorphy potentials, by $\mfk$; we include the constant functions in our definition. The space of holomorphy potentials of integral zero is isomorphic to the Lie algebra of the automorphism group $\Aut(X,L)$. Letting $\D^*$ be the $L^2$-adjoint of $\D$ with respect to the inner product induced by $\omega$,  the Lichnerowicz operator is given by $\D^*\D$; this is a fourth order elliptic linear partial differential operator, whose kernel consists of holomorphy potentials \cite[Definition 4.3]{szekelyhidi-book}. It is then well-known that the linearisation of the scalar curvature at a cscK metric is given by $-\D^*\D$ \cite[Lemma 4.4]{szekelyhidi-book}.

Another important term involved in the model operator is a sixth order elliptic operator, defined as follows. As the vector bundle $TX^{1,0}$ is a holomorphic vector bundle, it admits a $\bar\partial$-operator; we let $\bar \partial^*$ denote its $L^2$-adjoint. We will then also consider the operator $\D^*\bar \partial^*\bar\partial\D$, which can also be written $$\nabla^{1,0*}(\bar\partial^*\bar\partial)^2 \nabla^{1,0} =\nabla^{1,0*}\Delta_{\bar \partial}^2 \nabla^{1,0} ,$$ where $\Delta_{\bar \partial}$ denotes the $\bar\partial$-Laplacian. In particular its symbol agrees with that of $\Delta^3$.

We will also need to consider two further operators $H_1, H_2$,  which  are arbitrary self-adjoint operators satisfying for $j=1,2$ $$\int_X \gamma H_j \psi \omega^n = \int_X (\D\gamma, \D\psi)_{g_j} d\mu_j,$$ where each $d\mu_j$ is a smooth $(n,n)$-form and each $$g_j: \Gamma(T^{1,0}X\otimes \Omega^{0,1}(X)) \otimes \Gamma(T^{1,0}X\otimes \Omega^{0,1}(X)) \to \R$$ is a smooth bilinear pairing, but not necessarily a metric. Our model operator will then take the form \begin{equation}\label{model-operator}\G_{\epsilon} = c_0\D^*\D +\epsilon(c_1\D^*\bar \partial^*\bar\partial\D + H_1) + \epsilon^2(c_2\D^*\bar \partial^*\bar\partial\D + H_2),\end{equation} where $c_0$ and $c_1$ are strictly positive. Note that this is a self-adjoint elliptic operator for $\epsilon$ sufficiently small, as its symbol agrees with that of $\epsilon c_1 \Delta^3 + \epsilon^2 c_2 \Delta^3$, which is elliptic for $\epsilon$ sufficiently small since $c_1>0$. As we explain in Remark \ref{remark-add-higher-order-terms}, the $\epsilon^2$-term is included as the estimates we prove will only allow us to perturb the operator by an $O(\epsilon^3)$-term while retaining the relevant mapping properties.

We now work with Sobolev spaces $L^2_k$ for some large $k$. We let $ \mfk_{k, \perp}^2$ denote the $L^2$-orthogonal complement of the holomorphy potentials inside $L^2_k$. Note that the holomorphy potentials themselves are actually smooth, being the kernel of the elliptic operator $\D^*\D$, but we will sometimes also denote the  space of holomorphy potentials as $\mfk_{k}^2$ when considered as a subspace of $L^2_k$.

\begin{lemma}\label{eigenvalue-bound} There is a constant $c>0$ such that for all sufficiently small $\epsilon$ and for all $\psi \in \mfk_{k, \perp}^2 $ we have $$\langle \psi, \G_{\epsilon} \psi\rangle_{L^2} \geq c\|\psi\|^2_{L^2}.$$ Furthermore, the kernel of $\G_{\epsilon}$ consists of holomorphy potentials.

\end{lemma}

\begin{proof}We first consider the operator $$c_0\D^*\D + \epsilon H_1 + \epsilon^2 H_2.$$ The desired bound for the operator $\D^*\D$ is well-known:  there is a constant $c'$ such that for all  $\psi \in \mfk_{k, \perp}^2 $ we have $$\langle \psi, \D^*\D \psi\rangle_{L^2} \geq c'\|\psi\|^2_{L^2},$$ see for example Br\"onnle \cite[Lemma 37]{bronnle2}. We can obtain uniform bounds for $j=1,2$ $$-C_1(\D\gamma, \D\psi)_\omega \leq (\D\gamma, \D\psi)_{g_j} \leq C_j (\D\gamma, \D\psi)_{\omega}$$ for some $C_j>0$, independent of $\psi, \gamma$ and hence can obtain uniform bounds for some possibly different $C_j$ $$-C_j\int_X(\D\gamma, \D\psi)_\omega \omega^n \leq \int_X(\D\gamma, \D\psi)_{g_j}d\mu_j \leq C_j \int_X(\D\gamma, \D\psi)_{\omega}\omega^n.$$ Here we view $\omega$ as inducing a metric on $TX^{1,0}\otimes\Omega^{0,1}.$  It follows that for $\epsilon$ sufficiently small we have a bound$$\langle \psi, c_0\D^*\D \psi + \epsilon H_1 + \epsilon H_2\psi\rangle_{L^2} \geq c\|\psi\|^2_{L^2}$$for some $c>0$. 

The remaining terms are non-negative for $\epsilon$ sufficiently small. Indeed for $\epsilon$ sufficiently small the coefficient $\epsilon  c_1 + \epsilon^2c_2$ is positive and  $$\langle \psi,(\epsilon  c_1 + \epsilon^2c_2)\D^*\bar \partial^*\bar\partial\D\psi \rangle_{L^2}^2= (\epsilon  c_1 + \epsilon^2c_2)\| \bar \partial^*\D\psi \|_{L^2} \geq 0.$$ It follows that $$\langle \psi, \G_{\epsilon} \psi\rangle_{L^2} \geq c\|\psi\|^2_{L^2},$$ as required.

What remains is to characterise the kernel of $\G_{\epsilon}$. Note that certainly $\mfk \subset \ker \G_{\epsilon}$, since $\mfk = \ker \D.$ Otherwise we may write $\psi \in L^2_k$ as $\psi = \psi_{\mfk_{k}^2} + \psi_{\mfk_{k, \perp}^2}$ where $ \psi_{\mfk_{k}^2} \in \mfk_{k}^2$ and $\psi_{\mfk_{k, \perp}^2} \in \mfk_{k, \perp}^2$ are $L^2$-orthogonal and we may assume $\psi_{\mfk_{k, \perp}^2}  \neq 0$, and we see that \begin{equation*}\label{thekernel}\langle \psi, \G_{\epsilon} \psi\rangle_{L^2} = \langle \psi_{\mfk_{k, \perp}^2}, \G_{\epsilon} \psi_{\mfk_{k, \perp}^2}\rangle_{L^2} \geq c\|\psi_{\mfk_{k, \perp}^2}\|^2_{L^2}>0,$$ where we have used that $$\langle \psi_{\mfk_{k,}^2}, \G_{\epsilon} \psi_{\mfk_{k, \perp}^2}\rangle = 0\end{equation*} since $\G_{\epsilon}$ is self-adjoint and $\G_{\epsilon} \psi_{\mfk_{k}^2} = 0$. \end{proof}

\begin{corollary}\label{right-inverse-existence} For sufficiently small $\epsilon$, the operator $$\G_{\epsilon}: \mfk_{k, \perp}^2 \to  \mfk_{k-6, \perp}^2$$ is an isomorphism. Furthermore, the induced map \begin{align*}\hat \G_{\epsilon}: L^2_k \times \mfk &\to L^2_{k-6}, \\ (\psi, h) &\to \G_{\epsilon}\psi + h\end{align*} is surjective, and admits a right inverse.

 \end{corollary}

\begin{proof} We first show that $\G_{\epsilon}$ does actually send $\mfk_{k, \perp}^2$ to $\mfk_{k-6, \perp}^2$. In fact, for any $\psi \in L^2_k$ and any $h \in \mfk$ we have $$\langle h, \G_{\epsilon} \psi\rangle_{L^2} = 0$$ again by self-adjointness of $\G_{\epsilon}$. Since $$\G_{\epsilon}: \mfk_{k, \perp}^2 \to  \mfk_{k-6, \perp}^2$$ has trivial kernel by Lemma \ref{eigenvalue-bound}, it is a a self-adjoint elliptic partial differential operator with trivial kernel, hence is an isomorphism by the Fredholm alternative. 

Surjectivity of the induced map $\hat \G_{\epsilon}: L^2_k \times \mfk \to L^2_{k-6},$ is an immediate consequence, while a right inverse can be constructed explicitly. Indeed, since the operator $\G_{\epsilon}: \mfk_{k, \perp}^2 \to \mfk_{k-6, \perp}^2$ is an isomorphism, it admits some inverse $\G_{\epsilon}^{-1}: \mfk_{k-6, \perp}^2 \to \mfk_{k, \perp}^2$. Write $\psi \in  \mfk_{k-6, \perp}^2 $ as $\psi$  as $\psi_{\mfk_{k-6}^2} + \psi_{\mfk_{k-6, \perp}^2}$ where $ \psi_{\mfk_{k-6}^2} \in \mfk_{k-6}^2$ and $\psi_{\mfk_{k-6, \perp}^2} \in \mfk_{k-6, \perp}^2$ are $L^2$-orthogonal. Note that $\psi_{\mfk_{k-6}^2}$ is actually smooth as it is a holomorphy potential. Then a right inverse is given by \begin{equation}\label{def-right-inverse}\M_{\epsilon}(\psi) = (\G_{\epsilon}^{-1}\psi_{\mfk_{k-6}^2}, \psi_{\mfk_{k-6}^2}).\end{equation}\end{proof}

We next obtain an operator norm of the inverse operator $\G_{\epsilon}^{-1}: \mfk_{k, \perp}^2 \to \mfk_{k-6, \perp}^2$. We will use  elliptic regularity estimates for this, so it is more convenient to consider the rescaled operator $\epsilon^{-1}\G_{\epsilon}$, so that the ellipticity constants are actually uniformly bounded in $\epsilon$; here we recall that ellipticity follows from the fact that the sixth order coefficient of $\G_{\epsilon}$ is $(\epsilon c_1+\epsilon^2 c_2) \Delta^3$, where we have assumed $c_1>0$, so scaling by $\epsilon^{-1}$ gives a family of  operators whose ellipticity constants are actually bounded independently of $\epsilon$.

\begin{proposition}\cite[Chapter 5, Theorem 11.1]{taylor}\label{schauder} There is a constant $c>0$ such that for any $\psi \in \mfk_{k-6, \perp}^2$ and for all sufficiently small $\epsilon$ there is a bound of the form $$ \|(\epsilon^{-1}\G_{\epsilon})^{-1} \psi\|_{L_{k}^2} \leq c\epsilon^{-1}\left( \|(\epsilon^{-1}\G_{\epsilon})^{-1} \psi\|_{L^2} + \|\psi\|_{L_{k-6}^2}\right).$$
\end{proposition}

The point here is that our model operator $(\epsilon^{-1}\G_{\epsilon})^{-1}$ has uniformly bounded ellipticity constants, but the norm of the coefficients of the equation are actually only bounded uniformly by $C\epsilon^{-1}$ for some constant $C$, and hence are blowing up as $\epsilon \to 0$. Explicitly, the term which is blowing up is the leading term $\epsilon^{-1}\D^*\D$. In this situation, one obtains an elliptic regularity estimate where the coefficient in the bound is $c\epsilon^{-1}$. We learned that such a elliptic regularity estimate holds from an observation of Hashimoto for general elliptic operators  \cite[p. 800]{hashimoto}; the dependence of the coefficient in the bound on the norm of the coefficients is standard for second-order elliptic operators \cite[p. 92]{gilbarg-trudinger}.

\begin{corollary}\label{bound-on-genuine-inverse}There is a bound of the form $$\|\G_{\epsilon}^{-1}\|_{op} \leq C\epsilon^{-2}$$ for the operator $\G_{\epsilon}^{-1}: \mfk_{k-6, \perp}^2 \to \mfk_{k, \perp}^2$, for some $C>0$.
\end{corollary}

\begin{proof}

Let $\psi \in  \mfk_{k-6, \perp}^2$ and set $\gamma = \G_{\epsilon}^{-1} \psi, $ so that $\G_{\epsilon} \gamma = \psi.$  The elliptic regularity estimate gives $$\frac{\|(\epsilon^{-1} \G_{\epsilon})^{-1}\psi\|_{L^2_k}}{\|\psi\|_{L^2_{k-6}}} \leq c\epsilon^{-1} + c\epsilon^{-1}\frac{ \|(\epsilon^{-1}\G_{\epsilon})^{-1} \psi\|_{L^2}}{\|\psi\|_{L^2_{k-6}}} = c\epsilon^{-1} +c \frac{ \|\G_{\epsilon}^{-1} \psi\|_{L^2}}{\|\psi\|_{L^2_{k-6}}} .$$ By Cauchy-Schwarz we have $$\|\gamma\|_{L^2}\|\G_{\epsilon}\gamma\|_{L^2} \geq \langle \gamma, \G_{\epsilon}\gamma\rangle_{L^2},$$ so the bound $$\langle \gamma, \G_{\epsilon} \gamma\rangle_{L^2} \geq   \tilde c\|\gamma\|^2_{L^2}$$ for some $\tilde c>0$ given by Lemma \ref{eigenvalue-bound} implies $$\|\G_{\epsilon} \gamma\|_{L^2} \geq  \tilde c \|\gamma\|_{L^2}.$$ Thus $$\frac{ \|\G_{\epsilon}^{-1} \psi\|_{L^2}}{\|\psi\|_{L^2_{k-6}}} \leq  \frac{ \|\G_{\epsilon}^{-1} \psi\|_{L^2}}{\|\psi\|_{L^2}} = \frac{ \|\gamma \|_{L^2}}{\|\G_{\epsilon}\gamma\|_{L^2}} \leq \tilde c^{-1}.$$ It follows that $$\frac{\|(\epsilon^{-1} \G_{\epsilon})^{-1}\psi\|_{L^2_k}}{\|\psi\|_{L^2_{k-6}}} \leq c \epsilon^{-1} + c(\tilde c^{-1}) \leq C\epsilon^{-1}$$ for $\epsilon$ sufficiently small and some $C>0$, as required.\end{proof}

Recall that a right inverse to the induced map \begin{align*}\hat \G_{\epsilon}: L^2_k \times \mfk &\to L^2_{k-6}, \\ (\psi, h) &\to \G_{\epsilon}\psi + h\end{align*} is given through Equation \eqref{def-right-inverse} by $$\M_{\epsilon}(\psi) = (\G_{\epsilon}^{-1}\psi_{\mfk_{k-6}^2}, \psi_{\mfk_{k-6}^2}),$$ where $ \psi_{\mfk_{k-6}^2} \in \mfk$ is the $L^2$-projection of $\psi$ onto $\mfk$.

\begin{corollary} There is a bound on the operator norm of  $\M_{\epsilon}$ of the form $$\|\M_{\epsilon}^{-1}\|_{op} \leq C\epsilon^{-2}$$ for some $C>0$.
\end{corollary}

\begin{proof} The operator $\psi \to \psi_{\mfk_{k-6}^2}$ has operator norm bounded independently of $\epsilon$, so this is a direct consequence of Corollary \ref{bound-on-genuine-inverse}. 
\end{proof}

We will eventually be interested in perturbations of $\hat \G_{\epsilon}$. The following is then a consequence of  standard linear algebra (see for example \cite[Lemma 4.3]{bronnle2} for the result in linear algebra).

\begin{corollary}\label{linear-algebra} Suppose $L_{\epsilon}:  L^2_k \to  L^2_{k-6}$ is a sequence of bounded operators with $\|L_{\epsilon}\|_{op} \leq K$ for some $K$ independent of $\epsilon$. Then for all sufficiently small $\epsilon$ the operator $$(\psi, h) \to \hat \G_{\epsilon}\psi + \epsilon^3 L_{\epsilon}\psi + h$$   
is surjective and admits a right inverse $\tilde \M_{\epsilon}$. Moreover there is a constant $C>0$ such that $$\|\tilde\M_{\epsilon}^{-1}\|_{op} \leq C \epsilon^{-2}.$$
\end{corollary}

\begin{remark}\label{remark-add-higher-order-terms} This result is the reason we must include the $\epsilon^2$ term in our model operator: our bound on the operator norm of the right inverse means we can only add additional terms at order $\epsilon^3$ and retain the desired mapping properties.
\end{remark}
%
%
%
%The leading order term is simply the Lichnerowicz operator, with the subleading order term being a sixth order operator. 
%
% $$\frac{\Ric\omega^2\wedge\omega^{n-2}}{\omega^n}-\frac{2}{n-1}\Delta\left( \frac{ \Ric\omega\wedge \omega^{n-1}}{\omega^n}\right) \\ = \frac{\Ric\omega^2\wedge\omega^{n-2}}{\omega^n}-\frac{2}{n(n-1)}\Delta S(\omega)$$ in more detail. Linearising will produce a sixth order term $\psi \to \Delta^3\psi$ arising from linearising the Laplacian of the scalar curvature. 

\subsubsection{The approximate solution}\label{sub-approx-soln} We now assume that $\omega \in c_1(L)$ is cscK. Lemma \ref{large-volume} then implies that we have $$\Ima(e^{-i\phi_{\epsilon}} \tilde Z_{\epsilon}(\omega)) = O(\epsilon^2).$$ In order for our model linear operator to be a good approximation of the genuine linearised operator, we will need to consider a better approximation to a $Z$-critical K\"ahler metric. Since we are considering the general case when the Lichnerowicz operator $\D^*\D$ may have non-trivial kernel, or equivalently the case when $\Aut(X,L)$ may not be discrete, rather than finding approximate $Z$-critical K\"ahler metrics, we will instead try to find a $\omega_{\epsilon}$ approximately solving the condition that \begin{equation}\label{approx-with-holomorphy-potentials}\Ima(e^{-\phi_{\epsilon}}\tilde Z_{\epsilon}(\omega_{\epsilon})) \in \ker \D_{\epsilon},\end{equation} where $\D_{\epsilon} = \overline \partial \nabla^{1,0}_{\epsilon}$ is defined using $\omega_{\epsilon}$. That is to say, the function $\Ima(e^{-\phi_{\epsilon}}\tilde Z_{\epsilon}(\omega_{\epsilon}))$ is a holomorphy potential with respect to $\omega_{\epsilon}$. To this end, we recall that if $\nu$ is a K\"ahler potential and if $h$ is the holomorphy potential with respect to $\omega$ for some holomorphic vector field, then the function \begin{equation}\label{change-holomorphy-potential}h + \frac{1}{2}\langle \nabla \nu, \nabla h\rangle\end{equation} is the holomorphy potential with respect to the K\"ahler metric $\omega_{\nu} = \omega+i\ddbar \nu$ (see for example \cite[Lemma 12]{szekelyhidi-I}).

Analogously to Corollary \ref{right-inverse-existence}, the operator \begin{align*} L^2_k \times \mfk &\to L^2_{k-4}, \\ (\psi, h) &\to \D^*\D\psi + h\end{align*} is surjective. Although we have worked in Sobolev spaces, since the operator is elliptic the same holds for smooth functions.  Thus given $e \in C^{\infty}(X)$, there is a pair $(\psi,h)$ with \begin{equation}\label{kill-error-strategy} \D^*\D\psi + h = e.\end{equation}

\begin{lemma}\label{approximatesolution} Suppose $\omega$ is a cscK metric. Then for any fixed $m$ there is a sequence $\psi_j$ and holomorphy potentials $h_j$ such that $$\Ima\left(e^{-i\phi_{\epsilon}}\tilde Z\left(\omega+ \sum_{j=1}^{m} \epsilon^j i\ddbar \psi_j\right)\right) = \sum_{j=2}^{m+1} \epsilon^2\left(h_j + \frac{1}{2}\left\langle h_j, \sum_{i=l}^m \epsilon^l\psi_l\right\rangle \right)+  O(\epsilon^{m+2}).$$ 
\end{lemma}

These are approximate solutions to Equation  \eqref{approx-with-holomorphy-potentials}.

\begin{proof}

The linearisation of the scalar curvature at a cscK metric is the operator $-\D^*\D$ \cite[Lemma 4.4]{szekelyhidi-book}. As we have assumed $\omega$ is cscK, we have $$\Ima\left(e^{-i\phi_{\epsilon}}\tilde Z_{\epsilon}(\omega)\right) = e_2 \epsilon^2 + O(\epsilon^3).$$ By right-invertibility of the Lichnerowicz operator there is a function $\psi_2$ and a holomorphy potential $h_2 \in \mfk$ such that $$(\Rea(\rho_{n-1})L^n)\D^*\D\psi_1 = e_2 - h_2.$$ Since $\F_{\epsilon} = \epsilon(\Rea(\rho_{n-1})L^n)\D^*\D + O(\epsilon^2)$, it follows that $$\Ima\left(e^{-i\phi_{\epsilon}}\tilde Z_{\epsilon}\left(\omega+\epsilon \ddbar  \psi_1\right)\right) = h_2\epsilon^2+ O(\epsilon^3).$$

Next consider the error term $$\Ima\left(e^{-i\phi_{\epsilon}}\tilde Z_{\epsilon}\left(\omega+\epsilon \ddbar  \psi_1\right)\right) - \epsilon^2\left(h_2 +\frac{1}{2} \langle \nabla h_2, \nabla \epsilon\psi_1 \rangle\right) = e_3\epsilon^3.$$ We continue by applying Equation \eqref{kill-error-strategy} to find a function $\psi_2$ and a holomorphy potential $h_3$ such that $$(\Rea(\rho_{n-1})L^n)\D^*\D\psi_2 = e_2 - h_3.$$ Again since the leading order linear operator is $(\Rea(\rho_{n-1})L^n)\D^*\D$, it follows that  $$\Ima\left(e^{i-\phi_{\epsilon}}\tilde Z_{\epsilon}\left(\omega+i\ddbar(\epsilon  \psi_1+\epsilon^2 \psi_2)\right)\right)  =\left(h_2 + \frac{1}{2} \langle \nabla h_2, \nabla \epsilon\psi_1 \rangle\right) \epsilon^2 +h_3\epsilon^3 +O(\epsilon^4).$$ In particular $$\Ima\left(e^{-\phi_{\epsilon}}\tilde Z_{\epsilon}\left(\omega+i\ddbar(\epsilon  \psi_1+\epsilon^2 \psi_2)\right)\right) = \sum_{j=2}^{3} \epsilon^2\left(h_j + \frac{1}{2}\left\langle h_j, \sum_{l=1}^2 \epsilon^l\psi_l\right\rangle \right) + O(\epsilon^4).$$  Iterating this process gives the result.
\end{proof}

We will only require the approximate solution \begin{equation}\label{approximate-solution-eqn}\omega_{\epsilon} = \omega+\sum_{j=1}^{3}\epsilon^j i\ddbar \psi_j,\end{equation} which satisfies $$\Ima\left(e^{-i\phi_{\epsilon}}\tilde Z_{\epsilon}\left(\omega+\sum_{j=1}^{3}\epsilon^j i\ddbar \psi_j\right)\right) =\sum_{j=1}^3 \epsilon^2\left(h_j + \frac{1}{2}\left\langle \nabla h_j, \nabla\left(\sum_{i=1}^3 \epsilon^j\psi_j\right)\right\rangle \right)+  O(\epsilon^{5}).$$  We then set $$\gamma_{\epsilon} = \sum_{j=1}^{3}\epsilon^j i\ddbar \psi_j,$$ so that if $h$ is a holomorphy potential with respect to $\omega$, then $h+\frac{1}{2}\left\langle \nabla h,\nabla \gamma_{\epsilon}\right\rangle$ is a holomorphy potential with respect to $\omega_{\epsilon}$ by Equation \eqref{change-holomorphy-potential}.

We return to the model operator $\G_{\epsilon}$, however now defined with respect to the approximate solution $\omega_{\epsilon}$. In order to understand its properties, for clarity we consider the K\"ahler metric $\omega_{\delta}$ the approximate solution to order $O(\delta^5)$ given by Equation \eqref{approximate-solution-eqn} (namely we replace $\epsilon$ with $\delta$). Denote by $\mfk_{\delta}$ the space of holomorphy potentials with respect to $\omega_{\delta}$. Then the results we have already established imply that for each fixed ${\delta}$, the operator \begin{align*}\begin{split}\label{delta-epsilon}\hat \G_{\epsilon,{\delta}}: L^2_k \times \mfk_{\delta} &\to L^2_{k-6}, \\ (\psi, h) &\to \G_{\epsilon,{\delta}}\psi + h\end{split}\end{align*} is surjective for $\epsilon$ sufficiently small.

We claim that one can take the $\epsilon$ for which surjectivity of $\hat \G_{\epsilon,{\delta}}$ holds to be independent of ${\delta}$ for ${\delta}$ sufficiently small. More precisely, we claim that there is an $\epsilon_0$ and a ${\delta}_0$ such that $\hat \G_{\epsilon,{\delta}}$ is surjective for all ${\delta} \leq {\delta}_0$ and $\epsilon\leq \epsilon_0$. But this follows since in the ``eigenvalue bound'' of Lemma \ref{eigenvalue-bound} $$\langle \psi, \G_{\epsilon,{\delta}}\psi\rangle_{L^2} \geq c_{\delta}\|\psi\|^2_{L^2},$$ for $\psi$ orthogonal to $\mfk_{\delta}$, the value $c_{\delta}$ is actually continuous in ${\delta}$. Similar continuity statements in ${\delta}$ then further imply that the right inverse $\M_{\epsilon, \delta}: L^2_{k-6} \to L^2_k \times \mfk_{\delta}$ has operator norm which satisfies a uniform bound $$\|\M_{\epsilon,{\delta}}\|_{op} \leq C \epsilon^{-2},$$ where $C$ is independent of both $\delta$ and $\epsilon$. Here the continuity used is in the elliptic regularity estimate of Proposition \ref{schauder}. It follows that we can take ${\delta}=\epsilon$ and obtain a bound with respect to the approximate solution $\omega_{\epsilon}$. We will rephrase this in a form in which we will use these results.

\begin{corollary}\label{we-want-to-perturb-to-this}
Denote by $\G_{\epsilon}$  model operator with respect to the approximate solution $\omega_{\epsilon}$. Then the operator  \begin{align*}\tilde \G_{\epsilon}: L^2_k \times \mfk_{\epsilon} &\to L^2_{k-6}, \\ (\psi, h) &\to \G_{\epsilon}\psi + h+\frac{1}{2}\langle \nabla h, \nabla \gamma_{\epsilon}\rangle \end{align*}  is surjective and admits a right inverse $\tilde \M_{\epsilon}$. There is a bound on the operator norm of $\tilde \M_{\epsilon}$ of the form $\|\tilde \M_{\epsilon}\|_{op} \leq C \epsilon^{-2}$. 

Thus if $L_{\epsilon}:  L^2_k \to  L^2_{k-6}$ is a sequence of operators satisfying a uniform bound $\|L_{\epsilon}\|_{op} \leq K$ independent of $\epsilon$, then the operator $$(\psi, h) \to \G_{\epsilon}\psi + h +\frac{1}{2}\langle \nabla h, \nabla \gamma_{\epsilon}\rangle  + \epsilon^3 L_{\epsilon}$$ is surjective and right-invertible. The resulting right inverse also has operator norm satisfying a uniform bound by $C'\epsilon^{-2}$ for some $C'>0$.
\end{corollary}

\begin{proof}

We first consider the operator $\tilde \G_{\epsilon}$ itself. In comparison to the discussion immediately preceding the statement, the only difference is in the range of the operator. The discussion involves $\mfk_{\epsilon}$ rather than $\mfk$ itself. But if $h \in \mfk$, then $h + \frac{1}{2}\langle \nabla h, \nabla \gamma_{\epsilon}\rangle \in \mfk_{\epsilon}$. So the statement of the Corollary is simply a rephrasing of the discussion. The statements about perturbations are consequences of linear algebra as in Corollary \ref{linear-algebra}. \end{proof}

\subsubsection{Understanding the expansion of the operator}\label{expansion-calculations} We next consider some general aspects of the structure of the $Z$-critical equation. We will consider its expansion in powers of $\epsilon$, and to match with what we have considered it will be convenient to consider the ``rescaled'' equation $$-\epsilon^{-1}\Ima\left(\frac{\tilde Z_{\epsilon}(\omega)}{Z_{\epsilon}(X,L)}\right) = \Rea(\rho_{n-1})L^nS(\omega) + O(\epsilon),$$ so that if $\omega$ is a cscK metric its linearisation takes the form $- \Rea(\rho_{n-1})L^n\D^*\D+O(\epsilon)$. We will be interested in understanding the terms of order $\epsilon$ and $\epsilon^2$; controlling these will allow us to see the full linearised operator as a perturbation of the sum involving only terms of order up to $\epsilon^2$ which will be sufficiently by Corollary \ref{we-want-to-perturb-to-this}.  We will begin only by considering $\omega$, and will then later consider the approximate solution $\omega_{\epsilon}$.

We use our assumptions that:
\begin{enumerate}[(i)]
\item $\theta_1=0 = \theta_2=\theta_3=0$. The condition on $\theta_1$ is used so that the leading order term in the expansion is the scalar curvature, rather than the twisted scalar curvature, while the conditions on $\theta_2$ and $\theta_3$ are of a more technical nature and allow us to understand the $\epsilon^2$-term of the linearised operator. We expect that the conditions on $\theta_2$ and $\theta_3$ can be removed.
\item $\Rea(\rho_{n-1})<0$,  $\Rea(\rho_{n-2})>0$ and $\Rea(\rho_{n-3})=0$. The condition on $\Rea(\rho_{n-1})$ is essentially a sign convention, what is really needed is that these two real parts have opposite sign. This is essential to the analysis and is used in the $L^2$-bound for the model operator proved in Lemma \ref{eigenvalue-bound}. The condition on $\Rea(\rho_{n-3})$ is a technical assumption which we expect can be removed.
\end{enumerate}

As in Lemma \ref{large-volume} we write $Z_{\epsilon}(X,L) = r_{\epsilon}e^{i\phi_{\epsilon}},$ so that \begin{align*}\Ima(e^{-i\phi_{\epsilon}(X,L)} \tilde Z_{\epsilon}(\omega)) &= r_{\epsilon}(X,L) \Ima\left(\frac{\tilde Z_{\epsilon}(\omega)}{Z_{\epsilon}(X,L)}\right),  \\ &= r_{\epsilon}(X,L) \frac{\Ima \tilde Z_{\epsilon}(\omega) \Rea  Z_{\epsilon}(X,L) - \Rea \tilde Z_{\epsilon}(\omega)  \Ima  Z_{\epsilon}(X,L) }{\Rea  Z_{\epsilon}(X,L)^2 + \Ima  Z_{\epsilon}(X,L)^2},\end{align*}where we recall \begin{align*}Z_{\epsilon}(X,L) &= iL^n{\epsilon}^{-n} + \rho_{n-1}L^{n-1}.K_X{\epsilon}^{-n+1} + \rho_{n-2}L^{n-2}.K_X^2\epsilon^{-n+2} +\hdots,\\  \tilde Z_{\epsilon}(\omega) &= i -  \rho_{n-1}\frac{\Ric\omega\wedge\omega^{n-1}}{\omega^n}{\epsilon} + O({\epsilon}^2).\end{align*} Here we have used our assumptions 

Our equation takes the form $$\Ima\left(\frac{\tilde Z_{\epsilon}(\omega)}{Z_{\epsilon}(X,L)}\right) = \frac{\Ima \tilde Z_{\epsilon}(\omega) \Rea  Z_{\epsilon}(X,L) - \Rea \tilde Z_{\epsilon}(\omega)  \Ima  Z_{\epsilon}(X,L) }{\Rea  Z_{\epsilon}(X,L)^2 + \Ima  Z_{\epsilon}(X,L)^2},$$ where explicitly \begin{align*}Z_{\epsilon}(X,L) &= iL^n{\epsilon}^{-n} +\rho_{n-1}\alpha_1{\epsilon}^{-n+1} +\rho_{n-2}  \alpha_2\epsilon^{-n+2} +\rho_{n-1} \alpha_3\epsilon^{-n+3} + O(\epsilon^{-n+4}),\\  \tilde Z_{\epsilon}(\omega) &= i +\rho_{n-1} \tilde\alpha_1{\epsilon} +\rho_{n-2}  \tilde \alpha_2{\epsilon}^2 + \rho_{n-3} \tilde \alpha_3\epsilon^3+O(\epsilon^4),\end{align*} and where $\alpha_1 = L^{n-1}.K_X$, $\alpha_2 = L^{n-2}.K_X^2$, $\alpha_3 = L^{n-3}.K_X^3 $, while \begin{align*}\tilde \alpha_1 &= -\frac{\Ric\omega\wedge\omega^{n-1}}{\omega^n}, \qquad \tilde \alpha_2 =  \frac{\Ric\omega^2\wedge\omega^{n-2}}{\omega^n} - \frac{2}{n-1}\Delta  \frac{\Ric\omega\wedge\omega^{n-1}}{\omega^n}, \\ \tilde \alpha_3 &= -\frac{ \Ric\omega^3\wedge\omega^{n-3}}{\omega^n} + \frac{3}{n-2}\Delta  \frac{\Ric\omega^2\wedge\omega^{n-2}}{\omega^n}.\end{align*} 

The factor $$\frac{r_{\epsilon}(X,L)}{\Rea  Z_{\epsilon}(X,L)^2 + \Ima  Z_{\epsilon}(X,L)^2}$$ plays only a minor role in our expansion of  $\Ima\left(\frac{\tilde Z_{\epsilon}(\omega)}{Z_{\epsilon}(X,L)}\right)$. Indeed, we will have good control over the leading order two terms in $\epsilon$, while the third order (for our rescaled equation) $\epsilon^2$ term will require the most care to manage. So we can ignore this factor in controlling the linearisation. In addition, all relevant terms below have a uniform factor of $L^n$ arising from the leading order term of the expansion $Z_{\epsilon}(X,L) = iL^n{\epsilon}^{-n}+\ldots$, and we also omit  this uniform factor. Thus we need only understand the leading order three terms in the expansion of $$\Ima \tilde Z_{\epsilon}(\omega) \Rea  Z_{\epsilon}(X,L) - \Rea \tilde Z_{\epsilon}(\omega)  \Ima  Z_{\epsilon}(X,L).$$ Recall that we have assumed $\theta_1=\theta_2=\theta_3=0$. 

We see that the leading order term is $$\epsilon^{-n+1}\Rea(\rho_{n-1})\left(L^{n-1}.K_X + \frac{\Ric\omega\wedge \omega^{n-1}}{\omega^n}\right).$$ For the $\epsilon^{-n+2}$-term, we will for the moment only be interested in the degree six operator, which we see is given by $$-\epsilon^{-n+2}\frac{2\Rea(\rho_{n-2})}{n-1}\Delta \left( \frac{\Ric\omega\wedge\omega^{n-1}}{\omega^n}\right).$$ For the $\epsilon^{-n+3}$-term, we see that the sixth order component is given by, for some topological constant $c$ \begin{equation}\label{epsiloncubed}-\frac{3\Rea(\rho_{n-3})}{n-2}\Delta  \left(\frac{\Ric\omega^2\wedge\omega^{n-2}}{\omega^n} \right)+ c\Ima(\rho_{n-2})\Delta \left( \frac{\Ric\omega\wedge\omega^{n-1}}{\omega^n}\right).\end{equation} In particular if $\Rea(\rho_{n-3})=0$, the first of these two terms vanishes. 

While we have considered $\omega$ rather than the approximate solution $\omega_{\epsilon} = \omega+i\ddbar \gamma_{\epsilon}$, essentially the same statements hold using $\omega_{\epsilon}$. If we write $\alpha_{j,\epsilon}$ for the coefficients of $\epsilon^j$ in $\tilde Z_{\epsilon}(\omega_{\epsilon})$, then we still have $$\tilde Z_{\epsilon}(\omega_{\epsilon}) = i +\rho_{n-1} \tilde\alpha_{1,\epsilon}{\epsilon} +\rho_{n-2}  \tilde \alpha_{2,\epsilon}{\epsilon}^2\tilde \alpha_2 + \rho_{n-3} \tilde \alpha_{3,\epsilon}\epsilon^3+O(\epsilon^4),$$ implying the linearisation has similar properties up to order $\epsilon^4$, but for example with the leading order term replaced with  $$\epsilon^{-n+1}\Rea(\rho_{n-1})\left(L^{n-1}.K_X + \frac{\Ric\omega_{\epsilon}\wedge \omega_{\epsilon}^{n-1}}{\omega_{\epsilon}^n}\right).$$

\subsubsection{Properties of the linearisation} We now turn to the linearisation of the $Z$-critical equation. The aim is to compare the linearisation at the approximate solution $\omega+i\ddbar \gamma_{\epsilon}$ to the model operator $\G_{\epsilon}$, and in particular to use Corollary \ref{we-want-to-perturb-to-this} to infer properties of the genuine linearised operator.

We begin with a general result. We fix a $K$-equivariant K\"ahler metric $\omega \in c_1(L)$, not assumed to be cscK, and denote by $\F_{\epsilon}$ the linearisation of the operator $$\psi \to \Ima\left(e^{-i\phi_{\epsilon}}\tilde Z_{\epsilon}\left(\omega+ i\ddbar \psi \right)\right).$$ Denote also $\mfk$ the space of holomorphy potentials with respect to $\omega$.

\begin{proposition}\label{initialprop} For all $0<\epsilon \ll 1$ the map \begin{align*}\hat \F_{\epsilon}: L^2_k \times \mfk &\to L^2_{k-6}, \\ (\psi, h) &\to \F_{\epsilon}\psi - \langle \nabla \Ima(e^{-i\phi_{\epsilon}}\tilde Z_{\epsilon}(\omega)), \nabla \psi \rangle + h\end{align*}  is surjective. In addition exists a right inverse $\hat \scP_{\epsilon}$ of $\hat \F_{\epsilon}$ whose operator norm satisfies a bound of the form $\|\hat \scP_{\epsilon}\|_{op} \leq C \epsilon^{-2}$.
\end{proposition}

We recall our assumption, which will be used in the proof, that $(X,L)$ is a degeneration of a polarised manifold with discrete automorphism group. 

\begin{remark} To compare Proposition \ref{initialprop} to a well-known result in K\"ahler geometry, recall that the scalar curvature operator $\psi \to S(\omega+i\ddbar \psi)$ has linearisation \cite[Lemma 4.4]{szekelyhidi-book} $$ \psi \to  -\D^*\D\psi +  \langle \nabla S(\omega), \nabla\psi\rangle,$$ so subtracting   $\langle \nabla S(\omega), \nabla\psi\rangle$ leads to an operator whose kernel is precisely given by $\mfk$. Thus adding $h$ leads to a surjective operator, mirroring Proposition \ref{initialprop} . \end{remark}

The proof will use the moment map techniques developed in Section \ref{sec:infinite}. We continue to denote by $\J_X(M,\omega)$ the space of complex structures biholomorphic to the reference complex structure $J$, and recall the closed $(1,1)$-forms $\Omega_{\epsilon}$ defined on $\J_X(M,\omega)$ through Equation \eqref{infinite-form}. 
Any functions $u, v \in C^{\infty}(X,\R)$ induce tangent vectors on $\J_X(M,\omega)$ through the assignment $u \to Pu$ of Equation \eqref{infinitesimal-action}; the same as true for functions in $L^2_k$. As in Section \ref{sec:kuranishi}, this process can be integrated, associating to $\psi$ a new complex structure $F_{\psi}(J)$. We will use that the differential of the map $\psi \to F_{\psi}(J)$  at $\psi=0$ is \cite[Equation 3]{szekelyhidi-deformations} $$\psi \to JP(\psi).$$

\begin{proof}[Proof of Proposition \ref{initialprop}] We use many of the ideas of Section \ref{moment-map-section} to understand the general properties of the linearised operator. Consider $\omega_t = \omega + t i\ddbar v$, so that the derivative of $$\int_X u \Ima(e^{-i\phi_{\epsilon}} \tilde Z_{\epsilon}(\omega_t)) \omega_t^n$$ is given by \begin{equation}\label{linearisation-via-deligne}\frac{d}{dt}\int_X u \Ima(e^{-i\phi_{\epsilon}} \tilde Z_{\epsilon}(\omega_t)) \omega_t^n = \int_X u \F_{\epsilon}v  \omega^n + \int_X u\Ima(e^{-i\phi_{\epsilon}} \tilde Z_{\epsilon}(\omega)) \Delta v \omega^n.\end{equation} We are interested in the first of these terms, but the advantage of this perspective is that from the proof of Theorem \ref{infinite-dimensional-moment-maps-thm} we know that for each $t$ $$\frac{d}{ds}\Bigr|_{s=0}E_Z(tv + su) = \int_X u  \Ima(e^{-i\phi_{\epsilon}}\tilde Z(\omega_t))\omega_t^n,$$ so that $$\frac{d^2}{dtds}\Bigr|_{s,t=0}E_Z(tv + su) = \int_X u \F_{\epsilon}v  \omega^n + \int_X u\Ima(e^{-i\phi_{\epsilon}} \tilde Z_{\epsilon}(\omega)) \Delta v \omega^n.$$  It follows that the integral on the right hand side, considered as a pairing on functions, is actually symmetric.

We need to identify the $\epsilon^2$ and $\epsilon^3$ terms in the expansion of $\F_{\epsilon}$ in order to compare it to the model operator $\G_{\epsilon}$. For this we will link with the space $\J_X(M,\omega)$ and the moment map interpretation of the $Z$-critical equation established in Section \ref{moment-map-section}. We first consider the case $\Aut(X,L)$ is discrete, which allows us to use the results of Section \ref{moment-map-section}, which were proven under that assumption. Our functions $u,v$ can be viewed as inducing tangent vectors to $\J_X(M,\omega)$ at the point $J_X$ and we see from Equation \eqref{contraction-eqn} that \begin{equation}\label{equality-differentials}\Omega_{\epsilon}(Pu, JPv) = \frac{d}{dt}\Bigr|_{t=0}\int_X u \Ima(e^{-\phi_{\epsilon}} \tilde Z_{\epsilon}(J_t)) \omega^n,\end{equation} where we emphasise that we take the perspective that the complex structure is changing but the symplectic form $\omega$ is fixed. 

We next compare this to the linearisation with fixed complex structure and varying symplectic structure. Let $f_t$ be the diffeomorphisms of $X$ such that $f_t^*\omega_t = \omega$ and $f_t \cdot J = J$. Then $f_t^*\omega_t^n = \omega^n$, while $f_t^*\Ima(e^{-\phi_{\epsilon}}\tilde Z_{\epsilon} (\omega_t) )=\Ima(e^{-i\phi_{\epsilon}} \tilde Z_{\epsilon}(J_t)).$  We also need to understand the infinitesimal change in $u$ as we pull-back along $f_t$, for which we need to understand the construction of $f_t$ in more detail. As we only need to understand the infinitesimal construction of $f_t$ near $t=0$, it suffices to note that $f_t$ is given by taking the gradient flow along a path of vector fields $\nu_t$ on $X$ such that $\nu_0$ is the Hamiltonian vector field associated with the function $v$. Thus the infinitesimal change in $u$ is simply the Lie derivative $$\L_{\nu_0} u = \langle \nabla u, \nabla v\rangle,$$ where we have used the relationship between the Poisson bracket of functions (that is, the pairing of the induced Hamiltonian vector fields with respect to $\omega$) and the inner products of the Riemannian gradients. That is, \begin{align*}\frac{d}{dt}\Bigr|_{t=0}\int_X u \Ima(e^{-\phi_{\epsilon}} \tilde Z_{\epsilon}(J_t)) \omega^n =\frac{d}{dt}\Bigr|_{t=0}& \int_X u \Ima(e^{-i\phi_{\epsilon}} \tilde Z_{\epsilon}(\omega_t)) \omega_t^n  \\  &-  \int_X  \langle \nabla u, \nabla v\rangle \Ima(e^{-i\phi_{\epsilon}}\tilde Z_{\epsilon} (\omega) )\omega^n.\end{align*}

We now use Equation \eqref{linearisation-via-deligne}, from which it follows that \begin{align*}\frac{d}{dt}\Bigr|_{t=0}&\int_X u \Ima(e^{-\phi_{\epsilon}} \tilde Z_{\epsilon}(J_t)) \omega^n = \int_X u \F_{\epsilon}v  \omega^n  \\ &+ \int_X u\Ima(e^{-i\phi_{\epsilon}} \tilde Z_{\epsilon}(\omega)) \Delta v \omega^n - \int_X  \langle \nabla u, \nabla v\rangle \Ima(e^{-i\phi_{\epsilon}}\tilde Z_{\epsilon} (\omega) )\omega^n.\end{align*} Since the final two terms on the right hand side sum to $-\int_X u\langle \nabla \Ima(e^{-i\phi_{\epsilon}} \tilde Z_{\epsilon}(\omega)),\nabla v\rangle \omega^n ,$ we have \begin{align*}  \Omega_{\epsilon}(Pu,JPv) &= \frac{d}{dt}\Bigr|_{t=0}\int_X u \Ima(e^{-\phi_{\epsilon}} \tilde Z_{\epsilon}(J_t)) \omega^n \\ &= \int_X u \F_{\epsilon}v  \omega^n- \int_X u\langle \nabla \Ima(e^{-i\phi_{\epsilon}} \tilde Z_{\epsilon}(\omega)), \nabla v\rangle\omega^n.\end{align*} Thus the operator \begin{align}\begin{split}\label{togetselfadjoint}(u, v) \to \int_X u (\F_{\epsilon}v  -\langle \nabla \Ima(e^{-i\phi_{\epsilon}} \tilde Z_{\epsilon}(\omega)), \nabla v\rangle)\omega^n \end{split}\end{align}  is a self-adjoint operator which only depends on $Pu, Pv$.  As this is true for all $\epsilon$, it is true for each term in the associated expansion in powers of $\epsilon$.

When $\Aut(X,L)$ is not discrete, we use the key assumption that $(X,L)$ is a degeneration of a polarised manifold with discrete automorphism group. That is, $(X,L)$ is the central fibre of a test configuration for a polarised manifold with discrete automorphism group (to compare with our previous notation, we are considering $(X,L)$ to be what was previously denoted $(\X_0,\L_0)$). Thus we obtain a family $J_t$ of complex structures on the fixed underlying smooth manifold $M$ converging to $J_0$, the complex structure inducing $X$. Since the linearisation satisfies Equation \eqref{togetselfadjoint} for each $t$, the same equation holds at $t=0$. In particular self-adjointness, and dependence only on $Pu, Pv$ hold also with respect to $J_0$ as well.

We use the results of Section \ref{expansion-calculations} to identify the $\epsilon, \epsilon^2$ and $\epsilon^3$ terms in the expansion of the operator $$v \to \F_{\epsilon}v  -\langle \nabla \Ima(e^{-i\phi_{\epsilon}} \tilde Z_{\epsilon}(\omega)), \nabla v\rangle),$$ in order to compare them to the model operator. By what we have just proven, this operator must be self-adjoint, and the pairing $$(u, v) \to \int_X u (\F_{\epsilon}v  -\langle \nabla \Ima(e^{-i\phi_{\epsilon}} \tilde Z_{\epsilon}(\omega)), \nabla v\rangle)\omega^n$$ can only depend on $\D u$ and $\D v$, due to the identification of  Equation \eqref{Ddef}. 

The leading order $\epsilon$-term is given by $-\Rea(\rho_{n-1})\D^*\D$, since the leading order $\epsilon$-term in the expansion of $\Ima(e^{-\phi_{\epsilon}} \tilde Z_{\epsilon}(\omega))$ is simply the scalar curvature. The sixth-order operator in the $\epsilon^2$-term arises from linearising $\frac{2\Rea(\rho_{n-1})}{n(n-1)}\Delta S(\omega)$, meaning that the linearisation inherits a term of the form $-\frac{2\Rea(\rho_{n-1})}{n(n-1)}\Delta \D^*\D$. As we know the $\epsilon^2$-term only depends on $\D u,\D v$, the difference between the $\epsilon^2$-term and $-\frac{2\Rea(\rho_{n-1})}{n(n-1)} (\bar\partial^*\D)^* \bar \partial^* \D$ must be a fourth order operator depending only on $\D u, \D v$ as both are of the form $-\frac{2\Rea(\rho_{n-1})}{n(n-1)} \Delta^3$ plus some fourth order operator. In particular the $\epsilon^2$-term must be of the form $c_1 \D^*\bar \partial^*\bar\partial \D+ H_{1}$ where $$\int_X u H_{1} v \omega^n = \int_X (\D u, \D v)_{g_{1}} d\mu_{1},$$ and where $d\mu_{1}$ is a smooth $(n,n)$-form and $$g_{1}: \Gamma(T^{1,0}X\otimes \Omega^{0,1}(X)) \otimes \Gamma(T^{1,0}X\otimes \Omega^{0,1}(X)) \to \R$$ is a smooth bilinear pairing, but not necessarily a metric. In particular this is of the same form as the $\epsilon^2$-term of our model operator of Equation \eqref{model-operator} computed with respect to $\omega.$

We finally show that the $\epsilon^3$-term of our linearisation takes the same form as the model operator, for which we use that $\Rea(\rho_{n-3}) = 0$ and $\theta_3=0$. From Equation \eqref{epsiloncubed} it follows that the only sixth order term arises from linearising a multiple of $\Ima(\rho_{n-2}) \Delta S(\omega)$, which contributes one term which is involved in the $\epsilon^3$-term of the model operator. The remaining order terms are fourth-order and so again are given by some $H_2$ of the same form as $H_1$.

What we have demonstrated is that the linearised operator agrees with the model operator to order $\epsilon^3$. In particular Corollary \ref{linear-algebra} applies to give the statement of the Proposition. \end{proof}

In general we wish to solve the equation $$\Ima\left(e^{-i\phi_{\epsilon}}\tilde Z_{\epsilon}\left(\omega+ i\ddbar \psi \right)\right) - f - \frac{1}{2}\langle \nabla \psi, \nabla f\rangle = 0,$$ for $f \in \mfk$ and $\psi$ a K\"ahler potential. The linearisation of this operator is given by $$d\scS_{0,f}(\psi, h) = \F_{\epsilon} \psi - \frac{1}{2}\langle \nabla \psi, \nabla f\rangle  - h.$$ The following is an immediate consequence of  Proposition \ref{initialprop}.

\begin{corollary} For all $0 < \epsilon \ll 1$, the operator $$(\psi, h) \to d\scS_{0,f}(\psi, h) + \frac{1}{2}\left\langle \nabla \psi, \nabla\left(f - 2\Ima\left(e^{-i\phi_{\epsilon}}\tilde Z_{\epsilon}(\omega)\right) \right)\right\rangle$$ is surjective, admits a right inverse, and the operator norm of the inverse is bounded by $C\epsilon^{-2}$ for some $C>0$.
\end{corollary}

Here  $h$ and $f$ are holomorphy potentials with respect to $\omega$, which was arbitrary. We apply this to the approximate solutions $\omega_{\epsilon}$ constructed in Lemma \ref{approximatesolution}. Rescaling the holomorphy potentials by a factor of two, $\omega_{\epsilon}$ satisfies$$\Ima\left(e^{-i\phi_{\epsilon}}\tilde Z(\omega_{\epsilon})\right) - \frac{1}{2}f_{\epsilon} = O(\epsilon^5),$$ where the $f_{\epsilon} \in \mfk_{\epsilon}$, hence the term $$ \frac{1}{2}\left\langle \nabla \psi, \nabla\left(f - 2\Ima\left(e^{-i\phi_{\epsilon}}\tilde Z_{\epsilon}(\omega)\right) \right)\right\rangle = O(\epsilon^5)$$ is of high order in $\epsilon$. In particular this term does not affect the mapping properties of the linearised operator. The following is then the statement of ultimate interest from the present section.

\begin{corollary}\label{final-linear-bound} The linearisation $d\scS$ computed at the approximate solution $\omega_{\epsilon}$ is surjective, and right invertible. Moreover its right inverse has operator norm bounded by  $C\epsilon^{-2}$ for some $C>0$. \end{corollary}

\subsubsection{Applying the quantitative inverse function theorem}

We can now construct $Z$-critical K\"ahler metrics in the large volume limit, as well as their extremal analogue. We continue with the notation and hypotheses of the previous sections.

\begin{theorem}\label{thm:existence-large-volume} Suppose $(X,L)$ admits a cscK metric $\omega$, and is a degeneration of a polarised manifold with discrete automorphism group. Then $(X,L)$ admits solutions $\tilde \omega_{\epsilon}$ to the equation $$ \Ima(e^{-\phi_{\epsilon}} \tilde Z_{\epsilon}(\tilde \omega_{\epsilon})) \in \mfk_{\epsilon},$$ where $\mfk_{\epsilon}$ denotes the space of holomorphy potentials with respect to $\tilde \omega_{\epsilon}$.\end{theorem}

This result proves the existence of the analogue of extremal $Z$-critical K\"ahler metrics. It is a straightforward consequence that $(X,L)$ admits $Z_{\epsilon}$-critical K\"ahler metrics if and only if the analogue of the Futaki invariant described  in Proposition \ref{slope-for-z} vanishes for all holomorphic vector fields. In the discrete automorphism group case this produces the following. 

\begin{corollary} Suppose $(X,L)$ has discrete automorphism group and admits a cscK metric. Then $(X,L)$ admits $Z_{\epsilon}$-critical K\"ahler metrics for all $\epsilon \ll 1.$
\end{corollary}

To prove these results we will apply the quantitative implicit function theorem:

\begin{theorem}\cite[Theorem 4.1]{bronnle}\label{IFT} Let $G: B_1 \to B_2$ be a differentiable map between Banach spaces, whose derivative at $0 \in B_1$ is surjective with right inverse $P$. Let
\begin{enumerate}[(i)]
\item $\delta'$ be the radius of the closed ball in $B_1$ around the origin on which $G-dG$ is Lipschitz with Lipschitz constant $1/(2\|P\|)$, where we use the operator norm;
\item $\delta = \delta'/(2\|P\|).$
\end{enumerate}

Then whenever $y \in B_2$ satisfies $\|y - G(0)\| < \delta$, there is an $x \in B_1$ such that $G(x)=y$.
\end{theorem}

Denote by $G_{\epsilon}$ the operator $$G_{\epsilon}(\psi) = \Ima(e^{-i\phi_{\epsilon}}\tilde Z_{\epsilon}(\omega_{\epsilon}+i\ddbar \psi)).$$ Then the linearisation of the map  $\tilde G_{\epsilon}: L^2_k \times \mfk \to L^2_{k-6}$ defined by $$(\psi, h) \to G_{\epsilon}\psi - h - \frac{1}{2}\langle \nabla h,  \nabla\gamma_{\epsilon}\rangle$$ is the map $\tilde \F_{\epsilon}: L^2_k \times \mfk \to L^2_{k-6}$ defined by $$(\psi, h) \to \F_{\epsilon}\psi - h - \frac{1}{2}\langle \nabla h, \nabla \gamma_{\epsilon}\rangle,$$ since the terms not involving $G_{\epsilon}$ are actually linear in both factors.
Corollary  \ref{final-linear-bound} then implies that the linearisation of $\tilde G_{\epsilon}$ is surjective and admits a right inverse, and moreover provides a uniform bound on the operator norm of this right inverse in terms of a constant multiple of $\epsilon^{-2}$. 

To apply Theorem \ref{IFT} we thus need to obtain a bound on the operator norm of the operators $\tilde G_{\epsilon}-d\tilde G_{\epsilon}$. Denote $\scN_{\epsilon} = \tilde G_{\epsilon} -  \tilde \F_{\epsilon}$ the non-linear terms of the $Z$-critical operator, calculated with respect to the approximate solution $\omega_{\epsilon}$.

\begin{lemma}\label{nonlinear} For all $\epsilon$ sufficiently small, there are constants $c,C >0$ such that for all sufficiently small $\epsilon$, if $\psi,\psi' \in L^2_{k}(X,\R)$ satisfy $\|\psi \|_{L^2_{k}} , \| \psi' \|_{L^2_{k}} \leq c$ then $$ \| \scN_{\epsilon} (\psi) - \scN_{\epsilon} (\psi') \|_{L^2_{k-6}} \leq C \big( \| \psi \|_{L^2_{k}} + \| \psi' \|_{L^2_{k}} \big) \| \psi - \psi' \|_{L^2_{k}}.$$
\end{lemma}

\begin{proof} Since the two terms involving the Hamiltonian $h$ in $\tilde G$ are actually linear in $h$ and $\psi$, we may replace $ \scN_{\epsilon} (\psi) - \scN_{\epsilon} (\psi')$ with the terms only involving $G_{\epsilon}(\psi)$.

 The proof is then similar to a situation considered by Fine \cite[Lemma 7.1]{fine}, and is a straightforward consequence of the mean value theorem, which gives a bound $$ \| \scN_{\epsilon} (\psi) - \scN_{\epsilon} (\psi') \|_{L^2_{k-6}} \leq \sup_{\chi_t} \|(D\scN_{\epsilon})_{\chi_t}\|_{op}  \| \psi - \psi' \|_{L^2_{k}},$$ where $\chi_t = t \psi + (1-t)\psi'$ and $t \in [0,1]$. But $$(D\scN_{\epsilon})_{\chi_t} = \F_{\epsilon,\chi_t} - \F_{\epsilon,m},$$ where $\F_{\epsilon,\chi_t}$ is the linearisation of the $Z_{\epsilon}$-critical operator at $\omega_{\epsilon}+i\ddbar\chi$. So we seek a bound on the difference of the linearisations when we change the K\"ahler potential, but for $\epsilon\ll 1$ this can be bounded by $$\|\F_{\epsilon,\chi_t} - \F_{\epsilon}\|_{op} \leq c' \|\chi\|_{L^2_k},$$ where $c'$ is independent of $\epsilon$, which completes the proof. 
\end{proof}

\begin{remark}\label{varying-structure} In fact, as explained by Fine \cite[Section 2.2 and Lemma 8.10]{fine}, the above proof applies very generally, even varying in addition the complex structure. In the case the complex structure is varying, one obtains a bound where $\| \psi \|_{L^2_{k}} + \| \psi' \|_{L^2_{k}}$ is replaced by the norm of the difference $(J,\psi) - (J',\psi')$ \cite[Lemma 2.10]{fine}, so the constant obtained can be taken to be continuous when varying the complex structure. Fine explains this for the linearisation of the scalar curvature, but all that is needed is that the operator in question is a polynomial operator in the curvature tensor, which is true for $\tilde Z_{\epsilon}$ and which implies the same result for $\Ima(e^{-\phi_{\epsilon}}\tilde Z_{\epsilon})$. 
\end{remark}

This is everything needed  to apply the quantitative inverse function theorem, as in Fine \cite[Proof of Theorem 1.1]{fine}.

\begin{proof}[Proof of Theorem \ref{thm:existence-large-volume}] We consider the approximate solution $\omega_{\epsilon}$ which satisfies $\Ima(e^{-\phi_{\epsilon}}\tilde Z_{\epsilon} (\omega_{\epsilon})) = O(\epsilon^5)$. We note that as all of the input is invariant under a maximal compact torus $\Aut(X,L)$, the output produced will also be invariant.  There are three ingredients which we have established necessary to apply the implicit function theorem:

\begin{enumerate}[(i)]

\item Since we are considering the approximate solution, we have $\|G_{\epsilon}(0)\| = O(\epsilon^{5})$.

\item Next, note that the operator $\tilde \F_{\epsilon}$ is an surjective for $\epsilon$ small and the right inverse $\tilde P_{\epsilon}$ satisfies $$\|\tilde P_{\epsilon}\|_{op} \leq \epsilon^{-2} K_1$$ by Corollary \ref{final-linear-bound}.

\item Finally, note that there is a constant $M$ such that for all sufficiently small $\kappa$, the operator $\tilde G_{\epsilon} - D\tilde G_{\epsilon}$ is Lipschitz with constant $\kappa$ on $B_{M\kappa}$.

\end{enumerate}

The second and third of these imply that the radius $\delta'_{\epsilon}$ of the ball around the origin on which $\tilde G_{\epsilon} - D\tilde G_{\epsilon}$ is Lipschitz with constant $(2\|\tilde P_{\epsilon}\|)^{-1}$ is bounded below by $C\epsilon^2$ for a positive constant $C$. From the statement of the quantitive inverse function theorem, the radius $\delta_{\epsilon}$ of interest is defined by $$\delta_{\epsilon} = \delta'_{\epsilon}(2\|\tilde P_{\epsilon}\|)^{-1},$$ meaning that $\delta_{\epsilon} $ is bounded below by $C^2\epsilon^4.$ It follows that for $\epsilon \ll 1$, if  $\|\tilde G_{\epsilon}(0)\| < C^2\epsilon^4$,  then there is a solution to the equation $G_{\epsilon}(0)$, which is what we wanted to produce. Thus as our approximate solutions satisfy $\|G_{\epsilon}(0)\|=  O(\epsilon^{5})$, for $0<\epsilon\ll 1$, the proof is complete. Note that this produces solutions in some Sobolev space, but elliptic regularity produces smooth solutions as our equation is elliptic for sufficiently small $\epsilon$ by Lemma \ref{ellipticity-lemma}. 

Finally, while our definition of a $Z$-critical K\"ahler metric requires the positivity condition $\Rea(e^{-i\phi_{\epsilon}(X,L)}\tilde Z_{\epsilon}(\omega)>0$, one calculates that this is automatic for $0 < \epsilon \ll 1$.
\end{proof}

\subsubsection{Analysis over the Kuranishi space}\label{sec:kuranishi}

We next apply the quantitive implicit function theorem in a similar manner in \emph{families}. Recall our polarised manifold of interest $(X,L)$ is analytically K-semistable, so that it degenerates to a cscK manifold $(X_0,L_0)$. We will be interested in the Kuranishi space of $(X_0,L_0)$, which captures all deformations of $(X_0,L_0)$. Our setup and discussion is  based on that of Sz\'ekelyhidi \cite[Section 3]{szekelyhidi-deformations}, to which we refer for more details (see Inoue \cite[Section 3.2]{inoue} for another clear exposition). We denote by $(M,\omega)$ the underlying symplectic manifold of $(X_0,\omega)$, with $\omega$ the cscK metric, and denote by $\J(M,\omega)$ the space of almost complex structures on $M$ compatible with $\omega$. As in the work of Sz\'ekelyhidi, using the operator  $Ph =\bar\partial \nabla^{1,0}h$ of Equation \eqref{Ddef} (computed on $X_0$) we denote $$\tilde H^1 = \{ \alpha \in T_{J_0}\J: P^*\alpha = \bar \partial \alpha = 0\};$$ this is a finite dimensional vector space as it is the kernel of the elliptic operator $P^*P + \bar\partial^*\bar\partial$, and is the first cohomology of the elliptic complex \cite[Equation (4)]{szekelyhidi-deformations}$$C^{\infty}_0(M,\C) \to T_{J_0}\J(M,\omega) \to \Omega^{0,2}(M),$$ where the first morphism is given by $P$ and the second morphism is given by the $\bar\partial$-operator (again both associated to the complex manifold $X_0$). In the following we assume that the deformation theory of $(X_0,L_0)$ is \emph{unobstructed}, in the sense that the second cohomology of this complex vanishes. 

Denote by $K$ the stabiliser of $J_0$ under the action of $\mathcal{G}$, so that $K$ is the group of biholomorphisms of $(X_0,L_0)$ preserving the K\"ahler metric $\omega$ and the complexification $K^{\C}$ equals $\Aut(X_0,L_0)$ by a result of Matsushima \cite[Theorem 3.5.1]{gauduchon}.  The vector space $\tilde H^1$ admits a linear $K$-action.

Note that any holomorphic map $q: B \to \J(M,\omega)$ from a complex manifold $B$ and with image lying in the space of integrable complex structures produces a family of complex manifolds $\X \to B$ where the fibre is given by $\X_b = (M,J_{q(b)})$. Fixing a point $b \in B$, recall that we say that $\X$ is a \emph{versal deformation space} for $\X_b$ if every other holomorphic family $\U\to B'$ with $\U_{b'} \cong \X_b$ is locally the pullback of $\X$ through some holomorphic map $B' \to B$. We recall Kuranishi's result:

\begin{theorem}\label{kuranishi}\cite[Proposition 7]{szekelyhidi-deformations}\cite[Lemma 6.1]{chen-sun} There is an open neighbourhood $B \subset \tilde H^1$ of the origin and an embedding $$\Phi: B \to \J(M,\omega)$$ with $\Phi(0) = J_0$ and which produces a versal deformation space for $X_0$. Points in $B$ inside the same $K^{\C}$-orbit correspond to biholomorphic complex manifolds. The universal family $\X \to B$ admits a holomorphic line bundle $\L$ and a holomorphic $K$-action making $\X \to B$ a $K$-equivariant map. The form $\omega$ induces a $K$-invariant relatively K\"ahler metric which we denote $\omega_{\X} \in c_1(\L)$.
\end{theorem}

Unobstructedness allows us to assume that $B$ is an open neighbourhood of the origin, hence smooth. By construction, the underlying smooth manifold $M$ of $X$ and symplectic form are fixed in the Kuranishi family, while the complex structure varies. Thus we may smoothly write $\X = B \times M$, with $\X$ admitting a $K$-invariant relatively K\"ahler metric $\omega_{\X}\in c_1(\L)$ which restricts to the cscK metric on $X_0$. The Kuranishi family is precisely the pullback of the universal family over $\J(M,\omega)$.

\begin{remark} When $\Aut(X,L)$ is discrete but not finite, we have assumed that the test configuration producing the cscK degeneration of $(X,L)$ is $\Aut(X,L)$-equivariant, producing an $\Aut(X,L)$-action on $(X_0,L_0)$, since by Remark \ref{automorphisms-used} this was used in our proof of the moment map interpretation of the equation (and hence in understanding the linearisation of the equation). In this case we use the equivariant Kuranishi family as in the work of Inoue \cite[Section 3.2]{inoue}, which has a universal family admitting an $\Aut(X,L)$-action, and which has the property that maps to the equivariant Kuranishi space correspond to deformations of $(X_0,L_0)$ which are $\Aut(X,L)$-equivariant. 
\end{remark}

We now apply the analysis to the entire Kuranishi space. We perturb the initial relatively K\"ahler metric $\omega_{\X}$ in such a way as to allow us to later employ the moment map property of the $Z$-critical operator. To do so, we define a finite-dimensional function space depending on $b \in B$, which as a vector space will be  isomorphic to the Lie algebra $\mfk$ of $K$ for each $b\in B$ (perhaps after shrinking $B$). The definition of the function space is motivated by the moment map interpretation of the $Z$-critical operator given in Section \ref{moment-map-finite}, and implicitly appears there.

For each $v\in\mfk$, much as in Equation \eqref{use-later}, to the relatively K\"ahler metric $\omega_{\X}$ we associate the function  $h_{\X,v} \in C^{\infty}(\X)$ satisfying $$\L_{J_\X v}\omega_{\X} = \ddb h_{\X,v},$$ where $J_{\X}$ is the almost complex structure of $\X$. Such a $h_{\X,v}$ is unique up to the addition of a function $f$ satisfying $\ddb f=0$; as such, $f$ is constant along the fibres of the Kuranishi family (as its fibres are compact), hence $f$ is the pullback of a function $f_B \in C^{\infty}(B)$ satisfying $i\partial_B\bar\partial_B f_B=0$. The choice of such a function $f_B$ will not affect the definition of the function space of interest to us, so any choice suffices. 

Through the identification $\X = B \times M$, by pullback any ($K$-invariant) function $\psi \in C^{\infty}(M,\R)$ induces a ($K$-invariant) function on $\X$ which we still denote $\psi \in C^{\infty}(\X)$.  If we change $\omega_{\X}$ to a new $K$-invariant relatively K\"ahler metric $\omega_{\omega_{\X}}+\ddb \psi$ (so that we assume $\psi$ is a K\"ahler potential on each fibre), then by a similar calculation to Equation \eqref{contraction-eqn} the function $ h_{\X,v}$ changes to $$h_{\X,v,\psi} = h_{\X,v} + (J_{\X}v)\psi \in C^{\infty}(\X);$$ we use this notation also when $\psi$ is not $K$-invariant. We further denote $$h_{b,v,\psi} = h_{\X,v}|_{\X_b} \textrm{ and }\omega_{\X_b,\psi} =  (\omega_{\X}+\ddb\psi)|_{\X_b},$$ which allows us to write for  for each $\psi \in C^{\infty}(M)$ and $b\in B$ write $$\mfk_{b,\psi} = \{h_{b,v,\psi} - \hat h_{b,v,\psi}: v \in \mfk \},$$ where the constant $\hat h_{b,v,\psi}$ is the average of $h_{b,v,\psi}$, defined in such a way that the function space consists of functions which integrate to zero:$$\int_{\X_b}  (h_{b,v,\psi} - \hat h_{b,v,\psi}) \omega_{\X_b,\psi}^n=0.$$ For $b=0$, these are the holomorphy potentials with respect to $\omega_{0,\psi}$, hence $\mfk_{0,\psi} \cong \mfk$ as vector spaces; it follows that this isomorphism extends to a neighbourhood of $0\in B$, and we shrink $B$ so that this is the case. Our discussion extends to functions $\psi$ of lower regularity without change.

For each $b\in B$ we are interested in the operator $L^2_k \times \mfk \to L^2_{k-6}$ defined by $$(\psi, v) \to \Ima(e^{-i\phi_{\epsilon}}\tilde Z_{\epsilon}(\X_b, \omega_{\X_b,\psi} )) + h_{b,v,\psi}.$$ Here we have included the space $\X_b$ in the notation for clarity. Note that for $b=0$, provided $\psi$ is $K$-invariant the function $h_{0,v,\psi}$ is genuinely the holomorphy potential for the real holomorphic vector field $v$ on $\X_0$ with respect to the K\"ahler metric $\omega_{\X_0}$. Importantly, it follows that for $b=0$ this operator is precisely the operator considered throughout the present section, and in particular Theorem \ref{thm:existence-large-volume} allows us to conclude for $b=0$ that there exists a sequence $$(\psi_{\epsilon},v_{\epsilon}) \in L^2_k\times \mfk$$ such that $$\Ima(e^{-i\phi_{\epsilon}}\tilde Z_{\epsilon}(\X_0,\omega_{b,\Psi})) = h_{0,v_{\epsilon},\psi_{\epsilon}} \in \mfh_{0,\psi_{\epsilon}}.$$ We next explain how the same technique proves the following analogue for $b \neq 0$:

\begin{proposition}\label{reduction} Perhaps after shrinking $B$, for all $0<\epsilon\ll 1$ there is a map $$\Psi_{\epsilon}: B \to L^2_k(M,\R)$$ such that for all $b \in B$ $$\Ima(e^{-i\phi_{\epsilon}}\tilde Z_{\epsilon}(\X_b,\omega_{b,\Psi})) \in \mfk_{b,\Psi(b)}.$$
\end{proposition}

We mention two further properties of the $\Psi_{\epsilon}$ thus produced, before explaining how the same technique as  Theorem \ref{thm:existence-large-volume} establishes Proposition \ref{reduction}. Firstly, as the proof of Proposition \ref{reduction} ultimately uses the contraction mapping theorem, as a standard consequence of the  contraction mapping theorem the $b$-dependent functions $\Psi_{\epsilon}(b)$ are as regular as possible in $b\in B$ and $\epsilon$ (just as in \cite[Proof of Theorem 1]{szekelyhidi-I}, for example). We will only need that they are actually, say, $C^8$, to ensure that the $Z$-critical operator is twice differentiable, which is then guaranteed for sufficiently large $k$. Secondly, through  the (smooth) identification $\X = B\times M$, we may identify $\Psi_{\epsilon}$ with a function on $\X$ such that \begin{equation}\label{perturbed-forms}\omega_{\X,\epsilon} = \omega_{\X} + \ddb \Psi_{\epsilon}\end{equation} is relatively K\"ahler for all sufficiently small $\epsilon$. The remaining property we will need is that the $\omega_{\X} + \ddb \Psi_{\epsilon}$ produced in this manner is $K$-invariant, which is ultimately a consequence of $K$-invariance of all objects involved: the map $\X \to B$, the form $\omega_{\X}$ and the $Z$-critical operator itself.

Proposition \ref{reduction} follows directly from the arguments on a fixed complex structure, so we only sketch the differences. On the central fibre $(\X_0,\L_0)$, the result is precisely Theorem \ref{thm:existence-large-volume}. The three key ingredients in Theorem \ref{thm:existence-large-volume} were the construction of approximate solutions, the bound on the operator norm of the right inverse of the linearised operator, and the control of the non-linear operator. We mention how each aspect in turn adapts.

As a first step, we replace the initial form $\omega_{\X}$ with a $K$-invariant relatively K\"ahler metric (still denoted $\omega_{\X})$ that satisfies $$S(\omega_b) \in \mfk_{b,0};$$ this leaves the K\"ahler metric on $X_0$ unchanged, but perturbs the metric on nearby fibres. The construction of such an $\omega_{\X}$ follows from an application of the implicit function theorem analogous to \cite[Proposition 7]{szekelyhidi-deformations}, perturbing the K\"ahler metric on each fibre. The application of the implicit function theorem uses that the linearisation of the scalar curvature on the cscK manifold $(X_0,\omega)$ takes the form $-\D_0^*\D_0$, so that its kernel is isomorphic to $\mfk_{0,0}$, that the linearisation for general $b\in B$ is a perturbation of this, and that the function spaces $\mfh_{b,0}$ are similarly perturbations of $\mfh_{0,0}$.

The approximate solutions can then be constructed for all $b \in B$, since the property used to construct the approximate solutions was that the linearisation was to leading order the Lichnerowicz operator $-\D_0^*\D_0$ on the central fibre $(\X_0,\L_0)$. Since the linearisation for general $b\in B$ is a perturbation of $-\D_0^*\D_0$, it remains an isomorphism orthogonal to $\mfh_{b,0}$ (as this is itself is a perturbation of $\mfh_{0,0}$), so the same argument applies to produce approximate solutions to any order. Similarly, calculated at the approximate solution, the linearised operator at $b$ is a perturbation of the linearisation at $b=0$, hence the mapping properties are inherited from those on $(\X_0,\L_0)$, producing the appropriate bound on the operator norm of the right inverse. Here we use in addition regularity of the operator $(b,\psi,v) \to h_{b,v,\psi}$ to ensure the linearised operators at the approximate solutions converge as $b\to 0$. As noted in Remark \ref{varying-structure}, the bounds on the non-linear terms apply also with the complex structure allowed to vary. Thus one can produce the desired $\Psi_{\epsilon}(b)$ for each $b \in B$, and as above the contraction mapping theorem ensures regularity of $\Psi_{\epsilon}(b)$ as one varies $b$ or $\epsilon$.

As we see in Corollary \ref{momentmap-application}, the important consequence of Proposition \ref{reduction} is that a zero of a natural moment map on $B$ is then actually a genuine $Z$-critical K\"ahler metric. Thus we have reduced to a finite dimensional moment map problem.

\subsection{Solving the finite dimensional problem}\label{finite-dim-solution}

We now turn to the solution of the finite dimensional moment map problem. Recall from Section  \ref{moment-map-finite} that to the Kuranishi family $(\X,\L) \to B$, the central charge $Z_{\epsilon}$ and a relatively K\"ahler metric $\omega_{\X}$, we have associated a closed complex $(n+1,n+1)$-form on $\X$, which we denote here $\tilde Z_{\epsilon}(\X,\omega_{\X})$, hence producing a closed $(1,1)$-form on $B$ through taking the associated fibre integral. Using the forms $\omega_{\X,\epsilon} = \omega_{\X} + \ddb \Psi_{\epsilon}$ produced by Proposition \ref{reduction} (as in Equation \eqref{perturbed-forms}), we set $$\Omega_{\epsilon} = \Ima\left(e^{-i\phi_{\epsilon}} \int_{\X/B} \tilde Z_{\epsilon}(\X,\omega_{\X,\epsilon})\right).$$  By $K$-invariance of $\Psi_{\epsilon}$ and $\X \to B$, the forms $\Omega_{\epsilon}$ are $K$-invariant for all $\epsilon$, and are further K\"ahler for $\epsilon$ sufficiently small as they are perturbations of (the pullback under the Kuranishi map of) the Donaldson-Fujiki K\"ahler metric on $\J(M,\omega)$, just as in Proposition \ref{fujiki-result-fibre}. 

\begin{lemma}

There exist a sequence of moment maps  $$\mu_{\epsilon}: B \to \mfk^*$$  for the $K$-action on $(B, \Omega_{\epsilon})$.

\end{lemma}

\begin{proof} The $K$-action on $B$ is induced by the linear $K$-action on the vector space $ \tilde H^1(X_0,TX^{1,0}_0)$. The linear $K$-action on $B$ admits a canonical moment map with respect to the flat K\"ahler metric, and hence by the equivariant Darboux theorem also admits a moment map with respect to any other $K$-invariant symplectic form \cite[Theorem 3.2]{symplectic-book}. One may alternatively more directly used that one can write $\Omega_{\epsilon} = d\lambda_{\epsilon}$ for $\lambda_{\epsilon}$ a $K$-invariant one-form to conclude the existence of a moment map  \cite[Exercise  5.2.2]{MS}.\end{proof}

The moment maps $\mu_{\epsilon}$ are only unique up to the addition of an element of $(\mfk^{*})^K$, where the latter denotes $K$-invariant elements of $\mfk^*$ under the coadjoint action. The next result ensures that we have chosen the geometrically appropriate sequence of moment maps. In what follows for $v \in \mfk$ we denote by $h_{\epsilon,v}$ a function satisfying $$\ddb h_{\epsilon,v} = \L_{J_{\X}v}\omega_{\X,\epsilon};$$as in Section \ref{sec:kuranishi}, as discussed there, such a choice is unique up to the addition of the pullback of a function from $B$, and any choice suffices. We then denote $h_{\epsilon,v,b}$ its restriction to a fibre $\X_b$ and similarly denote $\omega_{b,\epsilon} = \omega_{\X,\epsilon}|_{\X_b}$, with the corresponding function spaces as in Section \ref{sec:kuranishi} denoted $\mfh_{b_{\epsilon},\epsilon}$.
\begin{lemma}
We may normalise $\mu_{\epsilon}$ such that $$ \langle \mu_{\epsilon}, v\rangle(0)  = \int_{\X_0} h_{\epsilon,v,0}\Ima(e^{-i\phi_{\epsilon}}\tilde Z_{\epsilon}(X_0, \omega_{b,\epsilon})) \omega_{b,\epsilon}^n.$$
\end{lemma}

\begin{proof}

Since adding an element of $(\mfk^*)^K$ preserves the moment map condition, we need only check that the map $\mfk \to \C$ defined by $$v \to  \int_{\X_0} h_{\epsilon,v,0}\Ima(e^{-i\phi_{\epsilon}}\tilde Z_{\epsilon}(X_0, \omega_{b,\epsilon})) \omega_{b,\epsilon}^n$$ is $K$-invariant, where $K$ acts on $\mfk$ by the adjoint action. But this follows since $\Ima(e^{-i\phi_{\epsilon}}\tilde Z_{\epsilon}(X_0, \omega_{b,\epsilon})) \omega_{b,\epsilon}^n$ is a $K$-invariant $(n,n)$-form and $$h_{\epsilon,k\cdot v}=k^*h_{\epsilon,v},$$ where $k\cdot v$ denotes the adjoint action.\end{proof}

While we have little control over the moment maps  $\mu_{\epsilon}$ in general, on the orbit of interest we next show that solving the moment map problem produces $Z$-critical K\"ahler metrics. For this we denote by $b_0\in B$ a point corresponding to the complex manifold $(X,L)$ of interest, so that $(X,L) \cong (\X_b,\L_b)$.

We claim that $0 \in \overline{K^{\C}.b_0} \cap B$. Since there is a test configuration for our polarised manifold with central fibre $X_0$, by versality of $B$ there is a sequence of points $b_l$ in $B$ with $\X_{b_0} \cong X$ and  $b_l \to 0$ as $b\to \infty$ To see this, note that the only point with fibre isomorphic to $X_0$ is $0$ itself, and that by a result of Sz\'ekelyhidi there is \emph{some} $\C^* \hookrightarrow K^{\C}$ such that the specialisation of $b_0$ corresponds to a complex structure admitting a cscK metric \cite[Theorem 2]{szekelyhidi-deformations}. But cscK specialisations are actually unique by a result of Chen-Sun \cite[Corollary 1.8]{chen-sun}, implying $0 \in \overline{K^{\C}.b_0} \cap B$. We note that we do not need to rely on the  deep result of  Chen-Sun  to prove our main result: without appealing to this, one could instead  consider the cscK degeneration $(\X_{b_{\infty}},\L_{b_{\infty}})$ of $(X,L)$ in $B$ produced by Sz\'ekelyhidi, and consider $(X,L)$ as a deformation of $(\X_{b_{\infty}},\L_{b_{\infty}})$ instead, arguing in the same way.

\begin{corollary}\label{momentmap-application}
For any $b \in \overline{K^{\C}.b_0}$, the moment map $\mu_{\epsilon}(b)$ is given by $$\langle \mu_{\epsilon}(b),v\rangle = \int_{X_b} h_{\epsilon,v,b}\Ima(e^{-i\phi_{\epsilon}}\tilde Z(\X_b, \omega_{\epsilon},b)) \omega_{\epsilon,b}^n.$$
\end{corollary}

\begin{proof} By Theorem \ref{moment-map-in-finite-dimensions-thm}, the operator $B \to \mfk^*$ assigning $$b \to \left[v \to  \int_{X_b} h_{\epsilon,v,b}\Ima(e^{-i\phi_{\epsilon}}\tilde Z_{\epsilon}(\X_b, \omega_{\epsilon},b)) \omega_{\epsilon,b}^n\right]$$ is a moment map for the $K$-action on $K^{\C}.b_0$. It follows that this operator agrees with $\mu_{\epsilon}$ up to the addition of an element of $(\mfk^*)^K$, namely that for all $b \in K^{\C}.b_0$ $$\langle \mu_{\epsilon}, v\rangle(b) = \int_{X_b} h_{\epsilon,v,b}\Ima(e^{-i\phi_{\epsilon}}\tilde Z_{\epsilon}(\X_b, \omega_{\epsilon},b)) \omega_{\epsilon,b}^n + \langle \xi_{\epsilon},v\rangle$$ for some $\xi_{\epsilon} \in (\mfk^*)^K$ independent of $b \in K^{\C}.b_0$. We claim that $\xi_{\epsilon}=0$, which will imply the result for all $b \in K^{\C}.b_0$. Fixing $v \in \mfk$, and taking any sequence $b_t \in K^{\C}.b_0$ converging to $0 \in \overline{K^{\C}.b_0}$, the fact that $\langle \mu_{\epsilon},v\rangle(b)$ and $ \int_X h_{\epsilon,v}\Ima(e^{-i\phi_{\epsilon}}\tilde Z_{\epsilon}(X, \omega_{\epsilon})) \omega_{\epsilon}^n$ both converge to $\int_{\X_0} h_{\epsilon,v}\Ima(e^{-i\phi_{\epsilon}}\tilde Z_{\epsilon}(\X_0, \omega_{\epsilon})) \omega_{\epsilon}^n$ implies that $\langle \xi_{\epsilon},v\rangle = 0$, as claimed. The result for points $b \in \overline{K^{\C}.b_0}$ also allowed to lie in the closure of the orbit of $b_0$ follow by continuity.\end{proof}

We appeal to a version of the Kempf-Ness theorem to construct zeroes of the moment maps $\mu_{\epsilon}$. 

\begin{proposition}\cite[Corollary 4.6, Propositions 4.8, 4.9]{DMS}\label{kempf-ness} Suppose $b_0$ satisfies the condition the following stability condition: for all $0<\epsilon\ll 1$, and for all one-parameter subgroups $\lambda_v$ of $K^{\C}$ associated to $v \in \mfk$ such that $$\lim_{t\to 0} \lambda_v(t).b = \lim_{t\to\infty} \exp(-itv).b_0 = b' \in B$$ exists and lies in $B$, we have $$\langle \mu_{\epsilon}, v\rangle(b')<0.$$ Then there exists a sequence of points $b_{\epsilon} \in K^{\C}.b_0$ such that $\mu_{\epsilon}(b_{\epsilon}) = 0$.\end{proposition}

\begin{remark}
The proof of this result uses that $B$ is an open ball in a vector space where the action as linear, as it passes to an associated projective problem. It is also important that $\epsilon$ is taken to be sufficiently small, as the proof proceeds by considering the gradient flow associated to the moment map problem, and the condition that $\epsilon$ be taken to be small is used to ensure that the flow converges in $B.$
\end{remark}

We are now in a position to prove our main result, for which are assumptions are the same as throughout.

\begin{corollary}\label{main-result-as-a-cor}
Suppose $(X,L)$ is asymptotically $Z$-stable. Then $(X,L)$ admits $Z_{\epsilon}$-critical K\"ahler metrics for all $0 < \epsilon \ll 1$.
\end{corollary}

\begin{proof}
A one-parameter subgroup $\lambda_v$ of $G$ associated to $v \in \mfk$ such that $$\lim_{t\to 0} \lambda_v(t).b = \lim_{t\to\infty} \exp(-itv).b_0 = b' \in B$$ exists and lies in $B$ induces a test configuration $(\Y,\L_{\Y})$ for $(X,L)$, with central fibre $(\X_{b'},\L_{b'})$, by restricting the Kuranishi family $(\X,\L) \to B$ (as in \cite[Proof of Theorem 2]{szekelyhidi-deformations}). By the form of the moment map $\mu_{\epsilon}$ given in Corollary \ref{momentmap-application} and Proposition \ref{slope-for-z}, \begin{equation}\label{used-futaki}\langle \mu_{\epsilon}, v\rangle(b') = -\Ima\left(\frac{Z_{\epsilon}(\Y,\L_{\Y})}{Z_{\epsilon}(X,L)}\right).\end{equation} Thus asymptotic $Z$-stability of $(X,L)$ forces the condition $\langle \mu_{\epsilon}, v\rangle(b')<0$, meaning that Proposition \ref{kempf-ness} implies the existence of zeroes $b_{\epsilon}$ of the moment maps $\mu_{\epsilon}$. In terms of the function space $\mfh_{b,\epsilon}$, this means that for any $h \in \mfh_{b_{\epsilon},\epsilon}$ $$\int_{\X_{b_{\epsilon}}}h \Ima(e^{-i\phi_{\epsilon}}\tilde Z_{\epsilon}(\X_{b_{\epsilon}}, \omega_{\epsilon},b_{\epsilon})) \omega_{\epsilon,b_{\epsilon}}^n = 0,$$ again by  the form of the moment map $\mu_{\epsilon}$ given in Corollary \ref{momentmap-application}. But since by Proposition \ref{reduction}  $$ \Ima(e^{-i\phi_{\epsilon}}\tilde Z_{\epsilon}(\X_{b_{\epsilon}}, \omega_{\epsilon},b_{\epsilon}))  \in \mfh_{b_{\epsilon},\epsilon}, $$ it follows from non-degeneracy of the $L^2$-inner product that  $$\Ima(e^{-i\phi_{\epsilon}}\tilde Z_{\epsilon}(\X_{b_{\epsilon}}, \omega_{\epsilon},b_{\epsilon}))=0.$$  Elliptic regularity implies these solutions are actually smooth, concluding the result. \end{proof}

\subsection{Existence implies stability}\label{converse}

We return to the base of the Kuranishi space $B$ and along with its universal family $(\X,\L) \to B$. Our hypothesis is that $(X,L)$ admits $Z_{\epsilon}$-critical K\"ahler metrics for all $\epsilon$ sufficiently small, in a way that is compatible with our proof of that ``stability implies existence''. That is, we assume that there is a sequence of relative K\"ahler metrics  $\omega_{\X,\epsilon} \in c_1(\L)$, such that for each $\epsilon$ there is a $b_{\epsilon}  \in  K^{\C}.b_0$, where $(\X_{b_0},\L_{b_0}) \cong (X,L)$, such that $$\Ima(e^{-i\phi_{\epsilon}}\tilde Z_{\epsilon}(\X_b, \omega_{\epsilon}|_{\X_b})) = 0.$$ Recall from the proof of Corollary \ref{main-result-as-a-cor} that each $\C^*\hookrightarrow K^{\C}$ produces a test configuration for $(X,L)$.

\begin{theorem} In the above situation,  for each test configuration $(\Y,\L_{\Y})$ arising from the action of $K^{\C}$, we have $$\Ima\left(\frac{Z_{\epsilon}(\Y,\L_{\Y})}{Z_{\epsilon}(X,L)}\right)>0$$ for all $0 < \epsilon \ll 1$.
\end{theorem}

This of course is equivalent to our definition of asymptotic $Z$-stability with respect to these test configurations, which used $k = \epsilon^{-1}$ rather than $\epsilon$. We note that, in principle, $(X,L)$ could admit $Z_{\epsilon}$-critical K\"ahler metrics which are ``far'' from the cscK degeneration $(X_0,L_0)$ and hence do not arise from this construction. Thus this is a truly local result.

\begin{proof} This is a formal consequence of standard finite dimensional moment map theory. By Corollary \ref{momentmap-application}, each $b_{\epsilon}$ is actually a zero of a genuine finite dimensional moment map with respect to the K\"ahler metrics $\Omega_{\epsilon}$ on $B$. It then follows by convexity of the log norm functional associated to the moment map that for any $\C^*$-action induced by $J_Bv$, with $J_B$ the almost complex structure on $B$ and $b_{\epsilon,0}$ the specialisation of $b_{\epsilon}$, the value $\langle  \hat\mu_{\epsilon}, v \rangle(b_{\epsilon,0})$ is negative. But by Equation \eqref{used-futaki}, we have $$\langle  \hat\mu_{\epsilon}, v \rangle(b_0) =   -\Ima\left(\frac{Z_{\epsilon}(\Y,\L_{\Y})}{Z_{\epsilon}(X,L)}\right),$$ proving the result. \end{proof}

\begin{remark} This is truly a local result, and it would be very interesting to obtain a global analogue. In principle, $(X,L)$ could admit $Z_{\epsilon}$-critical K\"ahler metrics that are ``far'' from the cscK metric on $(\X_0,\L_0)$ to which our result would not apply, although this seems unlikely in practice. Furthermore, there are many other test configurations for $(X,L)$ not arising from the Kuranishi space of $(\X_0,\L_0)$ for which we do not obtain stability with respect to. \end{remark}

\section{The higher rank case}\label{sec:higherrank}

We now extend our results to central charges involving higher Chern classes. Our exposition is brief, as the details are broadly similar to the ``rank one'' case, with a small number of exceptions. The first main difficulty is to extend the slope formula for the $Z$-energy to the setting where higher Chern classes are involved. The idea to overcome this is to reduce to the ``rank one'' case by projectivising, so we use of the Segre classes  $s_k(X)$ of $X$. The second difficulty is that, it is not clear that taking the variation of the $Z$-energy in this context actually produces a partial differential equation, so we simply include this as a hypothesis.

We thus consider a central charge of the form $$Z_{k}(X,L) =\sum_{l=0}^n \rho_l k^l   \int_X  L^l \cdot f(s(X)) \cdot \Theta,$$ for some $\rho$, $\Theta$ and $f(s(X))$ now an arbitrary polynomial in the Segre classes $s_1(X),\hdots, s_n(X)$ of $X$. The substantial difference is in the equation itself: the Euler-Lagrange equation of the $Z$-energy no longer produces a partial differential equation.

\subsection{Stability}

It is straightforward to extend the notion of stability, provided the central fibre of the test configuration is smooth, which we hence assume. Given such a test configuration $(\X,\L)$ for $(X,L)$, we associate to a term of the central charge of the form $$\int_X L^l\cdot s_{m_1}(X)\cdot\hdots\cdot s_{m_j}(X)\cdot \Theta$$ an intersection number $$\int_{\X} \L^{l+1}\cdot s_{m_1}(T_{\X/B})\cdot\hdots\cdot s_{m_j}(T_{\X/B})\cdot \Theta,$$ where $s_{m}(T_{\X/\pr^1})$ is the $m^{th}$ Segre class of the relative holomorphic tangent bundle $T_{\X/B},$ which is a holomorphic vector bundle as $\X \to \pr^1$ is a holomorphic submersion. The notion of stabilitya is then just as before: we require $$\Ima\left(\frac{Z_{\epsilon}(\X,\L)}{Z_{\epsilon}(X,L)}\right)>0 \textrm{ for } 0<\epsilon\ll 1$$

\begin{remark} One approach to defining the numerical invariant of interest more generally, when $\X$ is smooth but $\X_0$ is singular, is as follows.  Recall that the Segre classes are multiplicative in short exact sequences. Thus when $\X\to \pr^1$ is a smooth morphism, we have $$s(T_{\X}) = s(T_{\X/\pr^1})s(T_{\pr^1}),$$ where each of these denotes the holomorphic tangent bundle. When $\X$ has smooth total space but $\X_0$ is singular, so that $s(T_{\X/\pr^1})$ and $s(T_{\pr^1})$ are both defined, one can use this to define analogues of $s(T_{\X/\pr^1})$ and as $\X$ is smooth, one can still make sense of the intersection of cycles on $\X$ itself. It seems challenging to give a reasonable definition when $\X$ is singular, meaning intersection theory of cycles is not defined.
\end{remark}

\subsection{$Z$-energy}

We now fix a K\"ahler metric $\omega \in c_1(L)$ and recall some general theory of Bott-Chern forms. Good expositions are given by Donaldson \cite[Section 1.2]{donaldson-hermite-einstein} and Tian \cite[Section 1]{tian-BC}. The K\"ahler metric $\omega$ induces a Hermitian metric on the holomorphic tangent bundle, and hence induces a Chern-Weil representative $s_j(\omega)$ of the Segre classes $s_j(X)$ for all $j$ through the general theory of Bott-Chern forms. Suppose now that $\omega_{\psi}=\omega+i\ddbar \psi$ is another K\"ahler metric in the same class, producing another representative of $s_j(X)$. Then the theory of Bott-Chern forms implies that there is a $(j-1,j-1)$-form $\BC_j(\psi)$ such that $$s_j(\omega+i\ddbar \psi)  - s_j(\omega)= i\ddbar \BC_j(\psi).$$ To draw the parallel with the theory we have developed in the rank one case, note that that $s_1(\omega) = -\Ric(\omega)$, so $$\BC_1(\psi) = \log\left(\frac{\omega_{\psi}^n}{\omega^n}\right),$$  which is a function that appeared many times in Section \ref{sec:z-energy}.

With this in hand, we define Deligne functionals in an similar manner to Section \ref{sec:deligne}, and analogously to work of Elkik \cite{elkik}. A K\"ahler metric $\omega\in c_1(L)$ induces a metric on the holomorphic tangent bundle $T_{X}$. This produces representatives of the Segre classes $s_j(T_{\X})$, and changing $\omega$ to $\omega_{\psi}$ changes the representatives of the Segre classes through the Bott-Chern forms. We also fix a representative $\theta \in \Theta$.

We associate to the intersection number $\int_X L^l\cdot s_{m_1}(X)\cdot\hdots\cdot s_{m_j}(X)\cdot \Theta$ the value $$\frac{1}{l+1}\langle \psi, \hdots,\psi;\BC_{m_1}(\psi),\hdots,\BC_{m_j}(\psi);\theta\rangle \in \R$$ given by \begin{align*}\langle &\psi, \hdots,\psi;\BC_{m_1}(\psi),\hdots,\BC_{m_j}(\psi);\theta\rangle  \\ &= \int_{X}\psi\omega_{\phi}^l\wedge s_{m_1}(\omega_{\psi}) \wedge \hdots \wedge \wedge s_{m_j}(\omega_{\psi})\wedge\theta) \\ &+ \hdots + \int_{X} \BC_{m_j}(\psi)\omega^{l}\wedge s_{m_1}(\omega)\wedge\hdots\wedge s_{m_{j-1}}(\omega)\hdots \wedge  \theta,\end{align*}  by analogy with the usual theory of Deligne functionals. The basic properties of this functional extend directly: there is a natural analogue of the ``change of metric'' formula, which follows by definition, and the curvature property of Proposition \ref{hessian}. The curvature property is proven by Tian when $\theta=0$ \cite[Proposition 1.4]{tian-BC} for general functionals of this kind, but the proof applies to the general case. 

By linearity we have produced a functional $E_Z: \H_{\omega} \to\R$  on the space of K\"ahler metrics, which we call the $Z$-energy as before. In the case that $\theta=0$, a variational formula for the Deligne functional can be found in the work of Donaldson \cite[Proposition 6 (ii)]{donaldson-hermite-einstein}, and a similar result holds in general.  We will not make use of the precise variational formula, beyond the fact that the Euler-Lagrange equation is \emph{independent of initial K\"ahler metric} $\omega$ chosen. Thus the Euler-Lagrange equation is only a condition on $\omega_{\psi}$ and not $\omega$ itself. We note, however, that to phrase the Euler-Lagrange equation as a partial differential equation requires a further understanding of the linearisation of the Bott-Chern classes.

\begin{definition} We say that $\omega_{\psi}$ is a \emph{$Z$-critical K\"ahler metric} if it is a critical point of the $Z$-energy. \end{definition}

To clarify this condition, let  \begin{equation}\label{FZ}F_{Z,\psi}: f \to \frac{d}{dt} E_Z(\omega_{\psi}+ti\ddbar f)\end{equation} be the derivative of the $Z$-energy. Then a $Z$-critical K\"ahler metric is a zero of the map \begin{align*} C^{\infty}(X,\R) &\to C^{\infty}(X,\R)^*, \\ \psi &\to F_{Z,\psi}.\end{align*} In the ``rank one'' case, from Proposition \ref{euler-lagrange-derivation} the map $F_{Z,\psi}$ is given by $$F_{Z,\psi}(f) = \int f\Ima(e^{-i\phi}\tilde Z(\omega_{\psi}))\omega_{\psi}^n,$$ resulting in the Euler-Lagrange equation being equivalent to the partial differential equation $$\Ima(e^{-i\phi}\tilde Z(\omega_{\psi})) = 0.$$ Note in general that the operator $F_{Z,\psi}$ is linear in $\psi$ and so takes the form $$F_{Z,\psi} = \int_X L(\psi) \Ima(e^{-i\phi} \hat Z(\omega))\omega^n,$$ for some linear differential operator $L$ and some $\hat Z(\omega)$ which we do not explicitly derive. Let $L^*$ denote the formal adjoint of $L$.

\begin{definition} We say that $Z$ is \emph{analytic} if the condition $$\Ima(e^{-i\phi_{\epsilon}} L^*\hat  Z_{\epsilon}(\omega))=0$$ is a partial differential equation for $\omega$ for all $0 < \epsilon \ll 1$.
\end{definition}
We remark  that Pingali has, in a special case, linearised $c_2(\omega)$ and has even established an ellipticity result under hypotheses on the geometry of the manifold in question \cite[Lemma 3.1]{pingali}. 

\begin{example} Set $$Z_k(X,L) = \sum_{l=0}^n \int_X k^li^{n-l+1}L^l.c_{n-l}(X).$$ The variation of the Deligne functional associated to each term $\int_X L^l.c_{n-l}(X)$ has been calculated by Weinkove \cite[Lemma 5.1]{weinkove} (who does not use the Deligne functional terminology) to be $$\int_X \psi c_{n-l}(\omega)\wedge \omega^l,$$ so the induced equation is a fourth order partial differential equation only involving the Chern forms of $\omega$. In fact, for $k \gg 0$ small variants of the resulting $Z$-critical equation have been studied by Leung (under the name ``almost K\"ahler-Einstein metrics'' \cite{leung2}) and Futaki (under the name ``constant perturbed scalar curvature K\"ahler metrics'' \cite{futaki}). Note that, as the equation is fourth order, it is automatically elliptic for $k \gg 0$ as the leading order term of the linearisation is $\Delta^2$,  with this term coming from the linearisation of the scalar curvature. Thus this is an  analytic central charge. Leung and Futaki both use the inverse function theorem to produce solutions to their equations for $k \gg 0$; as these equations are fourth order, their applications of the inverse function theorem do not require the techniques we developed in Theorem \ref{thm:existence-large-volume}, where the main difficulties were caused by  the jump from a fourth order to a sixth order partial differential equation.
\end{example}

We must produce an analogue of the slope formula of Proposition \ref{slope-for-z}, which is the reason we make use of Segre classes rather than Chern classes. As in that situation, a test configuration smooth over $\C$ gives rise to a path $\psi_t$ of K\"ahler potentials, which in addition induces representatives of the Segre classes. Denote, as was done in the earlier situation of Section \ref{sec:deligne}, $h$ the function on $\X$ induced by the $\C^*$-action and the relatively K\"ahler metric $\omega_{\X}$. In addition denote $\omega_0$ the restriction of $\omega_{\X}$ to $\X_0$ and set $\tau = -\log |t|^2$.

\begin{proposition} We have equalities $$F_{Z,\X_0,\omega_0}(h) =  \lim_{\tau \to \infty} \frac{d}{d\tau} E_Z(\psi_{\tau})=\Ima\left(\frac{Z(\X,\L)}{Z(X,L)}\right).$$
\end{proposition}

\begin{proof} The Segre classes are defined in such a way that $$s_j(X) = \sigma_*(\scO(1)^{n-1+j}),$$ where $\sigma_*$ denotes the push-forward of a cycle through the map $\sigma: \pr(T_{X}) \to X$ and $\scO(1)$ is the relative hyperplane class. On the analytic side, the Hermitian metric on $TX$ induces a Hermitian metric on $\scO(1)$, with curvature $\omega_{FS}$ which restricts to a Fubini-Study metric on each fibre. We then have, for example from \cite[Proposition 1.1]{diverio} or \cite{guler}, an equality of forms $$\int_{\pr(T_{X})/X} \omega_{FS}^{n-1+j} = s_j(\omega),$$ which is simply the metric counterpart of the usual defining property of the Segre classes. 

Now suppose $\omega_{\psi}$ is another K\"ahler metric on $X$, giving representatives $s_j(\omega_{\psi})$ of the Segre classes. Then $$ s_j(\omega_{\psi}) -  s_j(\omega) = \int_{\pr(T_{X})/X} (\omega_{\psi,FS}^{n-1+j} - \omega_{FS}^{n-1+j}).$$ Writing $$\omega_{\psi,FS}- \omega_{FS} = i\ddbar \psi_{FS},$$ this means that \begin{equation}\label{BCint}\int_{\pr(T_{X})/X} \psi_{FS}\wedge \left(\sum_{q=0}^{n-2+j}\omega_{\psi,FS}^q \wedge \omega_{FS}^{n-2+j-q}\right) = \BC_j(\psi),\end{equation} since taking $i\ddbar$ commutes with the fibre integral and  \begin{equation*}\int_{\pr(T_{X})/X} (\omega_{\psi,FS}^{n-1+j} - \omega_{FS}^{n-1+j}) = \int_{\pr(T_{X})/X} i\ddbar \psi_{FS}\wedge \left(\sum_{q=0}^{n-2+j}\omega_{\psi,FS}^q \wedge \omega_{FS}^{n-2+j-q}\right).\end{equation*} We note here that Bott-Chern classes are only defined modulo closed forms of one degree lower, and so strictly speaking this is merely a representative of the Bott-Chern class.

We return to our integral $E_Z(\psi_{\tau})$ of interest, and as usual we focus on one term of the form $$\langle \psi, \hdots,\psi;\BC_{m_1}(\psi),\hdots,\BC_{m_j}(\psi);\theta\rangle.$$ The Segre class construction allows us to reduce to the line bundle case, where the result has already been established.  

Suppose first that $j=1$, meaning we only have one Segre class involved in the intersection number. Then  the equality $$\int_{\pr(TX)} \psi_{FS}\left(\sum_{l=0}^{n-2+j}\omega_{\psi,FS}^l \wedge \omega_{FS}^{n-2+j-l}\right)\wedge \sigma^*\beta = \int_{X}\BC_j(\psi)\wedge \beta$$ that we have established in Equation \eqref{BCint} allows us to conclude that the Deligne functional  $$\langle \psi, \hdots,\psi,\BC_{m}(\psi);\theta\rangle$$ can be computed on $\pr(T_X)$ as $$\langle \psi, \hdots,\psi,\psi_{FS},\hdots,\psi_{FS};\theta\rangle_{\pr(TX)},$$ where we pull back $\omega_{\psi}$ to $\pr(TX)$ to consider it as a form on $\pr(TX)$. 

In the case that multiple Bott-Chern forms are involved, we simply iterate this construction as follows. After following this procedure once, we have only $j-1$ Segre classes remaining on $\pr(TX)$. But we can pull back $TX$ through $\sigma: \pr(TX) \to X$, and in this way by functoriality the Segre forms computed with respect to the metric induced by $\sigma^*\omega$ are the pullback of the Segre form computed on $X$. Thus applying the same procedure, we reduce to only $j-2$ higher Segre classes, and repeating we eventually reduce to the line bundle case. What remains is to compute the asymptotic slope of the Deligne functional along the path of metrics induced by the test configuration. 

Projectivising $T_{\X/\C}$, we obtain a family $\pr(T_{\X/\C}) \to \C$ which admits a $\C^*$-action, and is essentially a smooth test configuration for $\pr(T_X)$ without a choice of line bundle. The relatively K\"ahler metric $\omega_{\X}$ produces a Hermitian metric on $T_{\X/\C}$ and, assuming there is only one Segre class $s_m(X)$ involved in the intersection number, we obtain that the limit derivative of the Deligne functional is $$\int_{\pr(T\X/\pr^1)} \L^{l+1}\cdot \scO(1)^{m+n-1}\cdot \Theta = \int_{\X} \L^{l+1}\cdot s_m(T_{\X/\pr^1})\cdot \Theta.$$ Iterating this procedure by pulling back the relative tangent bundle to $\pr(T_{\X/\C})$ produces the slope formula in general. The computation of the slope as an integral over $\X_0$ is completely analogous. \end{proof}

\subsection{Final steps} We now assume that $Z$ is an admissible central charge, in the sense of Section \ref{mainresults}, which means that $\Rea(\rho_{n-1})<0, \Rea(\rho_{n-2})>0, \Rea(\rho_{n-3})=0$ and $\theta_1=\theta_2=\theta_3=0$. These mean that the new terms in the Segre class enter at order $\epsilon^4$, meaning the structure of the equation at lower order is the same as in the ``rank one'' case.

We finally explain how to prove our main result in the higher rank case:

\begin{theorem}Let $Z$ be an analytic admissible central charge. Suppose $(X,L)$ has discrete automorphism group and is analytically K-semistable. If it is in addition asymptotically $Z$-stable, then it admits $Z_{\epsilon}$-critical K\"ahler metrics for all $\epsilon$ sufficiently small.\end{theorem}

The proof is, from here, very similar to the ``rank one'' case. The moment map interpretation is exactly as in the ``rank one'' case. Indeed, the construction of the sequence of K\"ahler metrics $\Omega_{\epsilon}$ on $B$ is identical to Equation \eqref{infinite-form}, as it does not use anything concerning the structure of the equation. Then the moment map property proven there does not actually use that the Euler-Lagrange equation is actually a partial differential equation, but rather just uses formal properties. Thus we see obtain an analogous  moment map interpretation.

The application of the implicit function theorem is much the same. By analyticicty of the central charge, the same reasoning as Section \ref{finite-dim-reduction} demonstrates that the linearisation  is an isomorphism, and the quantitative inverse function theorem allows us to construct a potential $\Psi$ such that the $Z_{\epsilon}$-critical operator lies in $\mfk^{2}_{k,B}$, where we use the same notation as Section \ref{finite-dim-reduction}. Here we use  that analyticity of the central charge implies the equation is an \emph{elliptic} partial differential equation.

The solution to the finite dimensional problem applies, as it is a general result in K\"ahler geometry, and the local converse is, again, identical in the higher rank case.

\end{document}